\documentclass[11pt, leqno]{amsart}

\usepackage[a4paper, hdivide={3cm,,3cm},vdivide={3.4cm,,3.4cm}]{geometry}

\allowdisplaybreaks

\usepackage{amssymb}
\usepackage{amsmath}
\usepackage{enumerate}
\usepackage{bm}

\newcommand{\R}{{\mathbb R}}

\newcommand{\N}{{\mathbb N}}
\newcommand{\Z}{{\mathbb Z}}

\newcommand{\pr}{\textnormal{pr}}
\newcommand{\per}{\textnormal{per}}
\newcommand{\sym}{\textnormal{sym}}

\newcommand{\shortspacing}{\\[1ex]}

\newcommand{\vol}{\textnormal{vol}}
\newcommand{\pot}{\textnormal{pot}}
\newcommand{\kin}{\textnormal{kin}}
\newcommand{\brotkiste}{\kappa}
\newcommand{\bratwurst}{\theta}
\newcommand{\bauernkeks}{\mathbf M}

\newcommand{\supp}{\textnormal{supp}}

\newtheorem{theorem}{Theorem}[section]
\newtheorem{predefinition}[theorem]{Definition}

\newtheorem{predefandlemma}[theorem]{Definition and Lemma}

\newtheorem{preremark}[theorem]{Remark}
\newenvironment{remark}{\begin{preremark}\normalfont}{\end{preremark}}
\newtheorem{precondition}[theorem]{Condition}

\newtheorem{lemma}[theorem]{Lemma}
\newtheorem{predoubts}{doubts}[section]

\newtheorem{preexample}[theorem]{Example}

\newtheorem{corollary}[theorem]{Corollary}
\newtheorem{proposition}[theorem]{Proposition}

\begin{document}

\title[$N/V$-limit for Langevin dynamics]{$N/V$-limit for Langevin dynamics in continuum}
\author{Florian Conrad, Martin Grothaus}

\address{Florian Conrad, Mathematics Department, University of Kaiserslautern, P.O.Box 3049, 67653
Kaiserslautern, Germany.
\newline{\rm\texttt{Email: fconrad@mathematik.uni-kl.de}}
\newline
Martin Grothaus, Mathematics Department,
University of Kaiserslautern, \newline P.O.Box 3049, 67653
Kaiserslautern, Germany; BiBoS, Bielefeld University, 33615
Bielefeld, Germany and SFB 611, IAM, University of Bonn, 53115
Bonn, Germany.
\newline {\rm \texttt{Email: grothaus@mathematik.uni-kl.de},
\texttt{URL: http://www.mathematik.uni-kl.de/$\sim$grothaus/ }}
}

\date{\today}

\thanks{\textit{2000 AMS Mathematics Subject Classification}: 60B12,82C22,60K35,60H10.\\ Financial support by the DFG through the project GR 1809/5-1 is gratefully acknowledged.}

\keywords{Limit theorems, Non-sectorial diffusion processes, interacting continuous particle systems, periodic boundary conditions.}

\date{\today}

\begin{abstract} We construct an infinite particle/infinite volume Langevin dynamics on the space of configurations in $\R^d$ having velocities as marks. The construction is done via a limiting procedure using $N$-particle dynamics in cubes $(-\lambda,\lambda]^d$ with periodic boundary conditions. A main step to this result is to derive an (improved) Ruelle bound for the canonical correlation functions of $N$-particle systems in $(-\lambda,\lambda]^d$ with periodic boundary conditions. After proving tightness of the laws of finite particle dynamics, the identification of accumulation points as martingale solutions of the Langevin equation is based on a general study of properties of measures on configuration space (and their weak limit) fulfilling a uniform Ruelle bound. Additionally, we prove that the initial/invariant distribution of the constructed dynamics is a tempered grand canonical Gibbs measure. All proofs work for general repulsive interaction potentials $\phi$ of Ruelle type (e.g.~the Lennard-Jones potential) and all temperatures, densities and dimensions $d\geq 1$. \end{abstract}

\maketitle

\begin{section}{Introduction}\label{sec:introduction}

The infinite particle Langevin equation 

\begin{eqnarray}\label{eqn:langevin}
dx^i_t&=&v^i_t\,dt\\
dv^i_t&=& \sqrt\frac{2\brotkiste}{\beta}dw^i_t-\brotkiste v^i_t\,dt-\sum_{i\neq j} \nabla\phi(x^i_t-x^j_t)\,dt\nonumber
\end{eqnarray}
where $\brotkiste>0$, $\beta>0$, describes the motion of particles at positions $x^i_t\in \R^d$ having velocities $v^i_t\in\R^d$, $i\in\N$, $t\in [0,\infty)$. This motion is influenced by a surrounding medium causing friction (corresponding to the second summand in the second line of (\ref{eqn:langevin})) and stochastic perturbation, modelled by a sequence of independent $\R^d$-valued Brownian motions $(w^i_t)_{t\geq 0}$. Moreover, the particles interact via a symmetric pair potential $\phi$.\shortspacing
For investigating the equilibrium fluctuations of infinite systems of interacting particles the first main step is the construction of equilibrium (martingale) solutions for the corresponding model (cf.~\cite{OT03}). In \cite{Fr79}, strong solutions are constructed to (\ref{eqn:langevin}) in the case $d=2$ for a wide class of symmetric pair potentials $\phi$ and initial configurations. In particular, the construction given there allows a singularity of $\phi$ in the origin and assumes $\phi$ to be $C^1(\R^d\setminus\{0\})$ with derivatives fulfilling some local Lipschitz continuity (we do not give all the details on the conditions). Another construction for arbitrary $d$, but with more restrictions on the potential can be found in \cite{OT03}. The potentials treated there are assumed to be positive, of finite range and $C^2$, which, in particular, does not allow any singularities. There are also constructions of deterministic dynamics for infinitely many particles ($\brotkiste=0$), see e.g.~\cite{MPP75}, \cite{Sie85}, \cite{BPY99}, some of which work in more general situations. However, note that for the above mentioned purpose of considering a scaling limit the stochastic dynamics is preferable, since one expects it to exhibit a better long-time behaviour. (See e.g.~\cite{Spo86} for the correspondence between ergodic properties and the Boltzmann-Gibbs principle, which is crucial for the derivation of hydrodynamic limits as in \cite{Spo86}, \cite{OT03}.)\shortspacing
Up to now there are no results on the construction of equilibrium Langevin dynamics covering physically realistic situations, such as e.g.~the Lennard-Jones potential in dimension $d=3$. Moreover, generalizations to the case of non-continuous forces $\nabla\phi$ have never been considered and are impossible when using the method from \cite{Fr79}, \cite{OT03}.\shortspacing
Therefore, in this article we present a completely different approach to construct for a wide class of potentials a martingale solution to (\ref{eqn:langevin}) in the sense of \cite{GKR04}, having a grand canonical Gibbs measure as initial distribution. The general method is the one used there for the construction of stochastic gradient dynamics. As assumptions on the potential we only need weak differentiability in $\R^d\setminus\{0\}$, boundedness of the weak derivatives away from $0$ and some quite weak assumption on integrability of the weak derivatives. As mentioned above, we consider this as the basis and first important step towards investigating the hydrodynamic behaviour of infinite particle systems in most general physically realistic situations.\shortspacing
Before describing the construction, we make the expression ``martingale solution'' more precise. To do so, we have to introduce some notation. Let us consider the space
\begin{equation}\label{eqn:Gammav}
\Gamma^v=\{\gamma\subset \R^d\times\R^d| \pr_x(\gamma)\in \Gamma\}
\end{equation}
of locally finite simple velocity marked configurations in $\R^d$, where $\Gamma=\{{\hat\gamma}\in\R^d| \sharp({\hat\gamma}\cap \Lambda)<\infty \mbox{ for all $\Lambda\subset\R^d$ compact}\}$ and $\pr_x$ denotes the projection to the first $d$ coordinates, i.e.~$\pr_x(\gamma)=\{x\in\R^d| (x,v)\in\gamma\}$, $\gamma\subset\R^d\times\R^d$. $\sharp A$ denotes the cardinality of a set $A$. By $\mathcal FC_b^\infty(\mathcal D_s,\Gamma^v)$ we denote the space of smooth cylinder functions on $\Gamma^v$ of the form $F(\cdot)=g_F(\langle f_1,\cdot\rangle,\cdots,\langle f_K,\cdot\rangle)$, where $K\in\N$, $g_F\in C^\infty_b(\R^K)$ (which means $g_F$ is infinitely often differentiable and all derivatives are bounded) and $f_i\in \mathcal D_s:=C_{sbs}^\infty(\R^d\times\R^d)$. Here some notation is to be clarified: We define $C_{sbs}^\infty(\R^d\times\R^d)$ to be the space of $C^\infty_b$ functions with \emph{s}patially \emph{b}ounded \emph{s}upport, i.e.~the subset of $C_b^\infty(\R^d\times\R^d)$ of functions having support in $\Lambda\times\R^d$ for some compact $\Lambda\subset \R^d$. Moreover, one defines $\langle f,\gamma\rangle:=\sum_{(x,v)\in\gamma} f(x,v)$ for $f$ having spatially bounded support (or also for more general $f$, e.g.~$f\geq 0$) and $\gamma\in\Gamma^v$. \\
Now observe that any $N$-particle solution $(x^i_t,v^i_t)_{t\geq 0,1\leq i\leq N}$ of (\ref{eqn:langevin}) solves the martingale problem for the generator $L_N$, defined by
\begin{align*}
L_{N}f(x,v):=&\sum_{i=1}^N \left(\frac{\brotkiste}{\beta}\Delta_{v^i}f(x,v)-\brotkiste v^i\nabla_{v^i}f(x,v)+v^i\nabla_{x^i}f(x,v)\right)\\
&+\sum_{i,j=1}^N \nabla\phi(x^i-x^j)(\nabla_{v^i}f(x,v)-\nabla_{v^j}f(x,v)),
\end{align*}
where $f: \R^{Nd}\times\R^{Nd}\to\R$ is sufficiently regular and $(x,v)=(x^1,\cdots,x^N,v^1,\cdots,v^N)\in\R^{Nd}\times\R^{Nd}$. Defining $\sym_N: \R^{Nd}\times \R^{Nd}\to \Gamma^v$ by $\sym_N(x,v)=\{(x^1,v^1),\cdots,(x^N,v^N)\}$ and ignoring the non-well-definedness on the diagonal, we may map the dynamics to $\Gamma^v$.\\
We then find that the law of the resulting $\Gamma^v$-valued process solves the martingale problem for $(L,\mathcal FC_b^\infty(\mathcal D_s,\Gamma^v))$, defined by
\begin{align}\label{eqn:generatorL}
LF(\gamma):=&\sum_{l,l'=1}^K \frac{\brotkiste}{\beta}\partial_{l}\partial_{l'} g_F(\langle \{f_i\}_{i=1}^K,\gamma\rangle) \langle (\nabla_v f_l)(\nabla_v f_{l'}),\gamma\rangle \\
&+\sum_{l=1}^K \partial_l g_F(\langle \{f_i\}_{i=1}^K,\gamma\rangle)\Bigg(\left\langle \frac{\brotkiste}{\beta}\Delta_v f_l-\brotkiste v\nabla_v f_l+v\nabla_x f_l,\gamma\right\rangle\nonumber\\
&-\sum_{\{(x,v),(x',v')\}\subset\gamma} \nabla\phi(x-x')(\nabla_v f_l(x,v)-\nabla_v f_l(x',v'))\Bigg)\nonumber
\end{align}
where $F$ is as above, $\gamma\in\Gamma^v$ and $\langle \{f_i\}_{i=1}^K,\cdot\rangle:=(\langle f_1,\cdot,\rangle,\cdots,\langle f_K,\cdot\rangle)$.\\
We therefore call any (possible infinite particle) $\Gamma^v$-valued process solving the martingale problem for $L$ on $\mathcal FC_b^\infty(\mathcal D_s,\Gamma^v)$ a martingale solution of (\ref{eqn:langevin}) (on configuration space).\shortspacing
Due to the degeneracy in the position coordinates of the generator $L$ as given above, there is no hope to apply the theory of symmetric or sectorial Dirichlet forms to obtain an existence result (as is done in the case of the stochastic gradient dynamics in \cite{Os96}, \cite{Y96}, \cite{AKR98b}). In finite dimensions, i.e.~in the case of finite particle Langevin dynamics, one can easily verify that the corresponding generator is non-sectorial. One might think of using the theory of generalized Dirichlet forms (cf.~\cite{St99a}) instead in order to construct the dynamics directly on configuration space (or the space of multiple configurations). But to do so, one needs to find a domain of essential m-dissipativity of $L$ in $L^2(\Gamma^v,\mu)$ for a suitable measure $\mu$. Even in finite dimension this is in general at least a non-trivial problem (but see \cite{CG06} for the case of finite particle dynamics corresponding to $H^{1,\infty}$-potentials).\shortspacing
As starting point for the construction of the infinite particle dynamics we use finite volume $N$-particle Langevin dynamics constructed in \cite{CG07a}. We approximate $\R^d$ by cubes $\Lambda_{\lambda_n}=(-\lambda_n,\lambda_n]^d$, $n\in\N$, where $\lambda_n\uparrow\infty$ as $n\to\infty$, and choose a sequence $(N_n)_{n\in\N}$ of natural numbers such that $\lim_{n\to\infty} \frac{N_n}{(2\lambda_n)^d}=\rho<\infty$. \shortspacing
In order to prove tightness of the sequence of the dynamics of $N_n$ particles in $\Lambda_{\lambda_n}$, $n\in\N$, we establish a (uniform improved) Ruelle bound for the correlation functions of their invariant initial distributions, the finite volume canonical Gibbs measures with periodic boundary condition. In \cite{Ru70} one finds the (original) proof for such a bound, which works at least (cf.~the proof of \cite[Corollary 5.3]{Ru70}) for empty boundary condition, but only in the grand canonical setting. In \cite{GKR04} a Ruelle bound for canonical correlation functions with empty boundary condition is shown by an adaption of Ruelles proof using an estimate for the partition functions from \cite{DM67}. In the situation of the dynamics in \cite{CG07a} the boundary of $(-\lambda_n,\lambda_n]$ is assumed to be periodic, such that effectively one has to consider the canonical correlation functions with periodic boundary condition. Though these functions may be written down similar to the empty boundary case using summations $\hat\phi_{\lambda_n}$ (cf.~(\ref{eqn:hatphilambda}) below) of the potential, these sums are not lower regular uniformly in $n$. But this would be necessary to apply the proof from \cite{GKR04} (essentially) directly. However, this problem is solved by another modification of this proof (basically by adding a third case to the case differentiation of Ruelles proof, cf.~Remark \ref{rem:tollesachesoeineruellebound} below), allowing us to use (uniformly lower regular) cutoffs of the $\hat\phi_{\lambda_n}$.\shortspacing
Having shown tightness of the approximating laws and therefore the existence of weak accumulation points, we next need to prove that these accumulation points solve (\ref{eqn:langevin}) in the sense of the martingale problem (as explained above). The main problem here is to approximate $LF$ as in \eqref{eqn:generatorL} uniformly on the side of the approximations as well as on the side of the limit by bounded continuous random variables. We prove that this is indeed possible, when the approximating measures fulfill uniformly a Ruelle bound. Section \ref{sub:weaklimits} contains results on such approximations which we consider to be useful in general when dealing with limits of stochastic dynamics on configuration space. Though using some of the arguments from the proofs contained in \cite{GKR04} these results can be used to generalize the construction of stochastic gradient dynamics given there to the case of potentials which are only weakly differentiable in $\R^d\setminus\{0\}$ instead of $C^1(\R^d\setminus\{0\})$. For details, see Remark \ref{rem:GKR04} below.\shortspacing
In \cite{GKR04}, under an additional assumption, also convergence of the corresponding $L^2$ semigroups is shown with the help of Mosco convergence of the associated Dirichlet forms. This yields convergence of the semigroups and hence the Markov property as well as convergence of the sequence of approximating laws. In our situation, we do not have symmetric Dirichlet forms corresponding to the approximating processes (which are not reversible). However, one may apply results from \cite{Tö06} to obtain convergence of the semigroups in this situation. In the present situation this approach depends again heavily on finding a suitable domain of essential m-dissipativity for the limiting operator. In this article we refrain from further pursuing this question.\shortspacing
Finally, by using a method from \cite{Ge95}, where equivalence of the microcanonical and the grand canonical ensemble are shown in the periodic boundary situation, we transport this result to the case of the canonical ensemble. This shows that the invariant measure of the dynamics constructed in this paper is a grand canonical Gibbs measure. The considerations in \cite{Ge95} work for any temperature/activity and therefore this result is not limited to the high temperature/low activity regime. The corresponding result in \cite[Section 6]{GKR04} is restricted to this regime. This may be considered to be an advantage of starting with a periodic setting.\shortspacing
Let us briefly summarize the core results of this paper:
\begin{itemize}
\item Derivation of an (improved) Ruelle bound for finite volume canonical correlation functions with periodic boundary condition. This bound holds for sufficiently large volume and is uniform for bounded particle densities. (Theorem \ref{thm:RB}, Corollary \ref{cor:iRB}.)
\item Tightness of the laws $P^{(n)}$ of $N_n$-particle Langevin dynamics in cubes $(-\lambda_n,\lambda_n]^d$ with periodic boundary condition for a wide class of symmetric pair potentials which are weakly differentiable in $\R^d\setminus\{0\}$. Here we assume that $\lambda_n\uparrow\infty$ and $\frac{N_n}{(2\lambda_n)^d}$ converges to some $\rho\in [0,\infty)$ as as $n\to\infty$. (Theorem \ref{thm:tightness}.)
\item Identification of accumulation points $P$ of $(P^{(n)})_{n\in\N}$ as above as martingale solutions of the Langevin equation on configuration space. (Theorem \ref{thm:martingaleproblemlemma}.)
\item Identification of the limit of finite volume canonical Gibbs measures with periodic boundary condition (i.e.~the initial and invariant distribution of $P$ as above) as grand canonical Gibbs measure. (Theorem \ref{thm:georgii}.) (We should mention that the hard work was done by Georgii and by Georgii and Zessin in \cite{GZ93}, \cite{Ge94}, \cite{Ge95}, where the corresponding result for limits of microcanonical Gibbs measures is shown.)
\end{itemize}

The above results apply to any dimension $d\geq 1$. The Ruelle bound and the result on equivalence of ensembles are true for any repulsive, tempered, bounded below potentials (see conditions (RP), (T), (BB) in Section \ref{sub:conditions}). The results on the dynamics require the weak differentiability condition (WD) formulated in Section \ref{sub:addconditions} and additionally, as a restriction coming from the approximation with periodic dynamics, one needs to control the forces at large distances with condition (IDF). However, this condition may be rather seen as a theoretical restriction (cf.~Remark \ref{rem:bedingungIDF}(i), and also Remark \ref{rem:bedingungIDF}(iii)). 

We begin our considerations by defining a Polish metric on the configuration space $\Gamma^v$ similar to \cite{KK04} and \cite{GKR04}.

\end{section}
\setcounter{equation}{0}

\begin{section}{A Polish metric on $\Gamma^v$}\label{sec:metric}

A natural topology for the space $\Gamma^v$ defined in (\ref{eqn:Gammav}) is the topology $\tau$ generated by the continuous functions with spatially bounded support, i.e.~by mappings $\langle f,\cdot\rangle$ with $f\in C_{sbs}(\R^d\times\R^d)$. In particular, using the (vague topology, i.e.~the) topology generated by $C_0(\R^d\times\R^d)$ functions instead, a sequence of configurations would be able to converge to the empty configuration just by convergence of the \emph{marks} to infinity.\shortspacing
In this section we define a Polish metric on $\Gamma^v$ which generates $\tau$. We use a construction similar to the one for unmarked simple configurations given in \cite{KK04} and \cite{GKR04}. Below we consider $\Gamma^v$ as a subset of the set $\mathcal M^v$ of Radon measures on $\R^d\times\R^d$ and $\Gamma$ as a subset of the set $\mathcal M$ of Radon measures on $\R^d$ (in the sense that a set of points in $\R^d\times\R^d$ (resp.~$\R^d$) is identified with the sum of Dirac measures in these points). The notation $\langle \cdot,\cdot\rangle$ is then extended to the dualization between continuous compactly supported functions and Radon measures.\\[2.4ex]
It is well known that the vague topology on $\mathcal M^v$ is generated by the metric $d_{\mathcal M^v}$, defined by
$$
d_{\mathcal M^v}(\mu,\nu):=\sum_{k=1}^\infty 2^{-k} \frac{\vert \langle f_k,\mu\rangle-\langle f_k,\nu\rangle\vert}{1+\vert \langle f_k,\mu\rangle-\langle f_k,\nu\rangle\vert},\quad\mu,\nu\in \mathcal M^v,
$$
where $f_k$, $k\in\N$, are suitable elements of $C_0^2(\R^d\times\R^d)$ (cf. e.g.~(the proof of) \cite[A7.7]{Kal76}). Let $(g_k)_{k\in\N}$ be a sequence in $C_0^2(\R^d)$ such that $d_{\mathcal M}(\mu,\nu):=\sum_{k=1}^\infty 2^{-k} \frac{\vert \langle g_k,\mu\rangle-\langle g_k,\nu\rangle\vert}{1+\vert \langle g_k,\mu\rangle-\langle g_k,\nu\rangle\vert}$, $\mu,\nu\in\mathcal M$, generates the vague topology on $\mathcal M$. For any two $\mu, \nu\in \mathcal M^v$ assigning finite mass to any $\Lambda\times\R^d$, $\Lambda\subset\R^d$ compact, we may define
\begin{equation}\label{eqn:apfle}
d_\star(\mu,\nu):=d_{\mathcal M^v}(\mu,\nu)+d_{\mathcal M}(\pr_x\mu,\pr_x\nu)
\end{equation}
where $\pr_x\mu\in \mathcal M$ denotes the image measure of $\mu\in \mathcal M^v$ w.r.t.~the projection to the first $d$ coordinates. We obtain the following lemma.

\begin{lemma}\label{lem:dstarcont}
The metric $d_\star: \Gamma^v\times\Gamma^v\to [0,\infty)$ generates the topology $\tau$. Moreover $\pr_x: (\Gamma^v,d_\star)\to (\Gamma,d_{\mathcal M})$ is continuous.
\end{lemma}

\begin{proof}
Continuity of $\pr_x$ follows from the definition.\\
The topology generated by $d_\star$ is coarser than $\tau$, since $\langle g_k,\pr_x(\cdot)\rangle$, $\langle f_k,\cdot\rangle$, $k\in\N$, are continuous w.r.t.~$\tau$. Conversely, let $(\gamma_n)_n\subset \Gamma^v$ converge to $\gamma\in \Gamma^v$ w.r.t.~$d_\star$ and let $f\in C_{sbs}(\R^d\times\R^d)$. Choose $\Lambda\subset\R^d$ bounded and such that $\supp(f)\subset\Lambda\times\R^d$ and $\pr_x\gamma(\partial\Lambda)=0$. ($\partial A$ denotes the boundary of a set $A\subset \R^d$ or $\R^d\times\R^d$.) By vague convergence of $\gamma_n$ towards $\gamma$ and of $\pr_x(\gamma_n)$ towards $\pr_x(\gamma)$ one obtains
$$
\gamma(\Lambda\times\Delta)=\lim_n \gamma_n(\Lambda\times\Delta)\quad\mbox{and}\quad\gamma(\Lambda\times\R^d)=\lim_n\gamma_n (\Lambda\times\R^d)
$$
for bounded $\Delta\subset\R^d$ such that $\gamma(\partial(\Lambda\times\Delta))=0$. By choosing $\Delta$ large enough such that $
\gamma(\Lambda\times\Delta^c)=0$ one finds that $\gamma_n(\Lambda\times\Delta^c)=0$ for sufficiently large $n$. Therefore, we find that for large $n$ we have $\langle f,\gamma_n\rangle=\langle f\cdot 1_{\Lambda\times\Delta},\gamma_n\rangle\to \langle f,\gamma\rangle$ as $n\to\infty$. Here and in the sequel $1_A$ always denotes the indicator function of a set $A$. (The domain of $1_A$ usually follows from the context.) Since $(\Gamma^v,d_\star)$ is as a metric space first countable, we have established continuity of the identity mapping $(\Gamma^v,d_\star)\to (\Gamma^v,\tau)$, and the lemma is shown.
\end{proof}

\begin{remark}\label{rem:metric}
Note that the above argument proving that convergence in $d_\star$ implies continuity of $\langle f,\cdot\rangle$ is in particular valid for unbounded continuous functions having spatially bounded support.
\end{remark}

However, $d_\star$ is far from being a complete metric on $\Gamma^v$. Firstly, consider the sequence $(\delta_{(x,v_n)})_{n\in\N}$ of Dirac measures, where $v_n\to\infty$. Such a sequence is a Cauchy sequence w.r.t.~$d_\star$, but it does not converge. Secondly, nothing prevents positions of particles from converging to each other. We use the idea from \cite{KK04} to solve these problems. Let $(I_k)_{k\in\N}$ be a sequence of $C^1$ functions on $\R^d$ such that $1_{\{\vert\cdot\vert\leq k\}}\leq I_k\leq 1_{\{\vert\cdot\vert\leq k+1\}}$ and choose a function $h: \R^d\to (0,1]$ such that $h\in C^1(\R^d)\cap L^1(\R^d)$. Moreover, let $\Phi: (0,\infty)\to [0,\infty)$ be a continuous decreasing function such that $\lim_{t\to 0} \Phi(t)=\infty$. Then the space $\Gamma$ of simple unmarked configurations is a complete (separable) metric space when equipped with the metric
\begin{equation}\label{eqn:Sfunction}
d^{\Phi,h}({\hat\gamma},{\hat\gamma}'):=d_{\mathcal M}({\hat\gamma},{\hat\gamma}')+\sum_{k=1}^\infty 2^{-k}r_k \frac{\vert S^{\Phi,hI_k}({\hat\gamma})-S^{\Phi,hI_k}({\hat\gamma}')\vert}{1+\vert S^{\Phi,hI_k}({\hat\gamma})-S^{\Phi,hI_k}({\hat\gamma}')\vert}\quad\mbox{for ${\hat\gamma},{\hat\gamma'}\in \Gamma$,}
\end{equation}
where for nonnegative $f\in C^1(\R^d)$ and ${\hat\gamma}\in \Gamma$ we set
$$
S^{\Phi,f}({\hat\gamma}):=\sum_{\{x,y\}\subset \hat\gamma} e^{\Phi(\vert x-y\vert)}f(x)f(y)
$$
and $(r_k)_{k\geq 0}$ is any bounded sequence of positive numbers (cf.~\cite[Theorem 3.5]{KK04}). (The topology and the completeness of the metric are, of course, invariant w.r.t.~the weights $(r_k)_{k\in\N}$, as long as they are positive and bounded.) Moreover, in \cite[Theorem 3.5]{KK04}, it is shown that the metric $d^{\Phi,h}$ generates the vague topology on $\Gamma$. This construction solves the problem of avoiding convergence to multiple configurations.\shortspacing
It remains to keep mass away from $v=\infty$. Let $a: [0,\infty)\to [0,\infty)$ be an increasing surjective $C^2$ function and define $\chi_k(x,v):=a(v)(hI_k)(x)$, $x,v\in\R^d$. We define for $\gamma, \gamma'\in \Gamma^v$
$$
d^{\Phi,a,h}(\gamma,\gamma'):=d_{\mathcal M^v}(\gamma,\gamma')+d^{\Phi,h}(\pr_x(\gamma),\pr_x(\gamma'))+\sum_{k=1}^\infty q_k 2^{-k} \frac{\vert\langle \chi_k,\gamma\rangle-\langle \chi_k,\gamma'\rangle\vert}{1+\vert\langle \chi_k,\gamma\rangle-\langle \chi_k,\gamma'\rangle\vert}
$$
with $(q_k)_{k\in\N}$ also being a bounded sequence of positive numbers. We obtain the following result. 
\begin{lemma}\label{lem:themetric}
$d^{\Phi,a,h}$ generates the topology $\tau$ on $\Gamma^v$ and $(\Gamma^v,d^{\Phi,a,h})$ is complete.
\end{lemma}
\begin{proof}
Convergence w.r.t.~$d^{\Phi,a,h}$ implies convergence w.r.t.~$d_\star$, so we have to prove the converse. Since convergence w.r.t.~$d_{\mathcal M}$ is equivalent to convergence w.r.t.~$d^{\Phi,h}$, for the first assertion it remains to check that $\langle \chi_k,\cdot\rangle$ is continuous w.r.t.~$d_\star$ for each $k\in\N$. This follows from Remark \ref{rem:metric}.\shortspacing
To prove completeness, let $(\gamma_n)_n\subset \Gamma^v$ be a Cauchy sequence w.r.t.~$d^{\Phi,a,h}$. Then we already know by completeness of $(\mathcal M^v,d_{\mathcal M^v})$ and $(\Gamma,d^{\Phi,h})$ that there exists $\gamma\in \mathcal M^v$ and ${\hat\gamma}\in \Gamma$ such that $\gamma_n\to \gamma$ and $\pr_x(\gamma_n)\to \hat{\gamma}$ vaguely as $n\to\infty$. We have to prove that $\gamma\in \Gamma^v$ and $\pr_x(\gamma)={\hat\gamma}$, so $\gamma_n\to\gamma$ w.r.t.~$d_\star$, hence w.r.t.~$d^{\Phi,a,h}$.\shortspacing
Since the $\N_0$-valued measures in $\mathcal M^v$ form a closed subset w.r.t.~vague convergence (cf. \cite[A7.4]{Kal76}), we know that $\gamma$ is $\N_0$-valued and thus it is a sum of Dirac measures (cf.~\cite[Lemma 2.1]{Kal76}). Here and below we set $\N_0:=\N\cup\{0\}$. Being a Cauchy sequence implies being a bounded sequence, so for each $k\in\N$ we have that $(\langle \chi_k,\gamma_n\rangle)_n$ is a bounded sequence. This implies that there exists a compact set $\Delta_k\subset \R^d$ such that for all $n\in\N$ it holds $\gamma_n([-k,k]^d\times\Delta_k^c)=\gamma([-k,k]^d\times\Delta_k^c)=0$ and we can also assume that $\gamma([-k,k]^d\times \partial \Delta_k)=0$. Let $\Lambda\subset \R^d$ be any open relatively compact set such that $\gamma(\partial(\Lambda\times\R^d))=0$. Then there exists $k\in\N$ such that $\Lambda\subset [-k,k]^d$ and thus for large $n$
$$
\gamma(\Lambda\times\R^d)=\gamma(\Lambda\times\Delta_k)=\gamma_n(\Lambda\times\Delta_k)=\gamma_n(\Lambda\times\R^d)=\hat\gamma(\Lambda).
$$
Using a base of the topology of $\R^d$ consisting of sets $\Lambda$ as above, one concludes that $\gamma\in\Gamma^v$ and $\pr_x\gamma=\hat\gamma$.
\end{proof}
\ \\
Finally, we define some compact functions on $\Gamma^v$, i.e.~functions having compact sublevel sets. On $\Gamma$, such functions are e.g.~given by $S^{\Phi,h}$, defined as in (\ref{eqn:Sfunction}) for every $\gamma\in\Gamma$ for which the sum converges (cf.~\cite[p.~782]{KK04}). (Note that $h$ does not have compact support.) We define (with $a,\Phi,h$ as above) for $\gamma\in \Gamma^v$
$$
S^{\Phi,a,h}(\gamma):=S^{\Phi,h}(\pr_x(\gamma))+\sum_{(x,v)\in \gamma} a(v)h(x).
$$
\begin{lemma}\label{lem:compactfct}
The sets $M_K:=\{\gamma\in \Gamma^v| S^{\Phi,a,h}(\gamma)\leq K\}$, $K\in\R$, are compact.
\end{lemma}
\begin{proof}
$S^{\Phi,a,h}$ is the sum of two increasing limits of continuous functions: continuity of $S^{\Phi,hI_k}$ is shown in \cite[Lemma 3.4]{KK04} and continuity of $\langle\chi_k,\cdot\rangle$ follows from Remark \ref{rem:metric}. So $S^{\Phi,a,h}$ is lower semicontinuous, which implies that the $M_K$ are closed.\shortspacing
Let $(\gamma_n)_n\subset M_K$. Then by compactness of $\{S^{\Phi,h}\leq K\}$ in $\Gamma$ the sequence $(\pr_x \gamma_n)_n$ has a convergent subsequence. We denote its limit by ${\hat\gamma}$ and we assume that $(\pr_x\gamma_n)_n$ is already this subsequence. Let $\Lambda\subset \R^d$ be compact. We know that by definition of $M_K$ and $S^{\Phi,a,h}$ it holds $\gamma_n\cap (\Lambda\times\Delta)=\gamma_n\cap (\Lambda\times\R^d)$ for some compact $\Delta\subset \R^d$ and for all $n$. Moreover, assuming that $\hat\gamma(\partial\Lambda)=0$, for large $n$ it holds $\gamma_n(\Lambda\times\R^d)=\hat\gamma(\Lambda)<\infty$. But these two properties of the $\gamma_n$ already imply relative compactness of $(\gamma_n\cap (\Lambda\times\R^d))_{n\in\N}$ w.r.t.~vague topology in $\Lambda\times\R^d$. Using a diagonal argument and a sequence $(\Lambda_k)_{k\in\N}$ of compact sets such that $\bigcup_k\Lambda_k=\R^d$ and $\hat\gamma(\partial\Lambda_k)=0$ for all $k\in\N$, we find that $(\gamma_n)_n$ is relatively compact in $\mathcal M^v$ w.r.t.~vague topology. So as in the proof of Lemma \ref{lem:themetric} we can show that any accumulation point $\gamma$ fulfills $\gamma\in\Gamma^v$ and $\pr_x(\gamma)=\hat{\gamma}$, which proves the lemma.
\end{proof}
In fact, since in many of the considerations below the velocities do not play an interesting role, we can often restrict to the unmarked configurations. Therefore, the functions $S^{\Phi,a,h}$ are only included for completeness as well as Lemma \ref{lem:compactfct} above. 
\end{section}
\setcounter{equation}{0}

\begin{section}[Ruelle bound]{Ruelle bound in the finite volume canonical case with periodic boundary condition}\label{sec:RB} 

In this Section we derive the Ruelle bound for correlation functions corresponding to finite volume canonical Gibbs measures with periodic boundary condition. We first state and discuss conditions on the potential which are similar to those in \cite[Section 3]{GKR04} in Section \ref{sub:conditions} and investigate properties of the periodic sum of the potential in Section \ref{sub:properties}. In particular, we prove that the important superstability property holds uniformly for these sums as well as temperedness and lower regularity in a sense sufficient for our purposes. We then go on with the proof of the Ruelle bound in the periodic case in Section \ref{sub:Ruellebound}. Finally, in Section \ref{sub:weaklimits} we show that a uniform Ruelle bound extends to weak limits of measures, and prove some approximation results which we need for the proof of Theorem \ref{thm:martingaleproblemlemma} below. Though all considerations are stated for the configurational case (not including velocities) they also extend to the case of ``full'' measures (with independent Gaussian distributed velocities). For details on this fact, see also Section \ref{sub:weaklimits}.

\begin{subsection}{Conditions on the potential}\label{sub:conditions}

Throughout Section \ref{sec:RB} we assume that the (symmetric) pair potential $\phi: \R^d\to\R\cup\{\infty\}$ is measurable and fulfills the assumptions $(BB)$, $(RP)$, $(T)$ which are given below. By $\vert\cdot\vert$ we denote the maximum norm in $\R^k$, $k\in\N$, i.e.~$\vert (y_1,\cdots,y_k)\vert:=\max_{1\leq i\leq k} \vert y_i\vert$, $(y_1,\cdots,y_k)\in\R^k$.
\begin{enumerate}
\item[(BB)] (\emph{bounded below}) There exists $M<\infty$ such that $\phi(x)\geq -M$ for all $x\in \R^d$.
\item[(RP)] (\emph{repulsion}) There exist $R_1>0$ and a decreasing continuous function $\Phi: (0,\infty)\to [0,\infty)$ with $\lim_{t\to 0}\Phi(t)t^d=\infty$ such that
$$
\phi(x)\geq \Phi(\vert x\vert)\quad \mbox{for $\vert x\vert\leq R_1$}.
$$
Furthermore, $\phi$ is bounded from above on $\{x\in \R^d| r\leq \vert x\vert<\infty\}$ for all $r>0$.
\item[(T)] (\emph{temperedness}) There exist $G,R_2<\infty$ and $\varepsilon>0$ such that
$$
\vert \phi(x)\vert\leq G\vert x\vert^{-d-\varepsilon}\quad \mbox{for $\vert x\vert\geq R_2$}
$$
\end{enumerate}

Note that the second condition in (RP) implies that we may (and therefore we will) set $R_1=R_2=:R$. Moreover, $R$ may be chosen arbitrarily small (changing, of course, the other constants).\shortspacing
For later use in Section \ref{sec:construction} we need more regularity of the function $\Phi$, so we prove the following lemma.
\begin{lemma}\label{lem:Theothercapitalphi}
Let $\Phi: (0,\infty)\to [0,\infty)$ be continuous, decreasing and such that $\Phi(t)t^d\to \infty$ as $t\to 0$. Then there exists $\hat{\Phi}: (0,\infty)\to [0,\infty)$ such that $\hat{\Phi}\leq \Phi$, $\hat{\Phi}(t) t^d\to \infty$ as $t\to 0$ and such that moreover $\hat{\Phi}$ is continuously differentiable and $e^{-a\hat{\Phi}}\hat\Phi'$ is bounded for any $a>0$.
\end{lemma}
\begin{proof}
Choose $s_1\in (0,1]$ such that $\Phi(s)\geq s^{-d}$ for all $s\in (0,s_1]$. When $s_k$ is chosen for some $k\in\N$, we choose $s_{k+1}<s_k\wedge \frac{1}{k+1}$ such that $\Phi(s)\geq (k+1)s^{-d}$ for all $s\in (0,s_{k+1}]$. We define a function $\hat{\Phi}_1$ in the following way:
$$
\hat{\Phi}_1(s):=0\quad\mbox{for $s\in [s_1,\infty)$}
$$
and
$$
\hat{\Phi}_1(s):=\hat{\Phi}_1(s_k)+k(s^{-d}-s_k^{-d})\quad\mbox{for $s\in [s_{k+1},s_k]$.}
$$
$\hat{\Phi}_1$ is decreasing and continuous. By induction one shows that $\hat{\Phi}_1(s_k)\leq ks_k^{-d}$ for all $k\in\N$ and therefore $\hat{\Phi}_1(s)\leq ks^{-d}$ for all $s\in [s_{k+1},s_k]$, $k\in\N$. Therefore $\hat{\Phi}_1\leq \Phi$. Another induction shows that for any $k\in\N$ we have $\hat{\Phi}_1(s)\geq ks^{-d}-ks_k^{-d}$ for \emph{all} $s\leq s_k$ and one concludes that $\hat{\Phi}_1(s)s^d\to \infty$ as $s\to 0$. Computing the derivative of $\hat{\Phi}_1$ in the sets $(s_{k+1},s_k)$, we find that
$$
0>\hat{\Phi}_1'(s)=-dks^{-d-1}\geq -ds^{-d-2}\quad \mbox{whenever $s\in (s_{k+1},s_k)$,}
$$
since $k\leq \frac{1}{s_k}$ for all $k\in\N$. So the absolute value of $\hat{\Phi}'_1$ grows polynomially with $s^{-1}$. Therefore $\hat{\Phi}_1$ fulfills all assertions with the exception that it is not differentiable at the points $s_k$, $k\in\N$.\shortspacing
Similarly to $\hat{\Phi}_1$ we define the function $\hat{\Phi}_2\leq \hat{\Phi}_1$ using the sequence $(s'_k)_{k\in\N}$, defined by $s'_k:=s_{k+1}$ instead of $(s_k)_k$. Then $\hat{\Phi}_2$ has the same properties as $\hat{\Phi}_1$ and the derivatives are such that there exists a continuous function $\theta:(0,\infty)\to (-\infty,0]$ such that $s^{d+1}\hat{\Phi}_1'(s)\leq \theta(s)\leq s^{d+1}\hat{\Phi}_2'(s)$ on $(0,\infty)\setminus\{s_k| k\in\N\}$. Integrating we obtain a function $\hat{\Phi}$, defined by $\hat{\Phi}(s):=\int_{s_1}^s \theta(t)t^{-d-1}\,dt$, $s\in (0,\infty)$, fulfilling the assertions.
\end{proof}

Let $\Lambda\subset \R^d$. By $\Gamma_\Lambda$ we denote the set of locally finite simple configurations in $\Lambda$ (i.e.~locally finite subsets). In the sequel we will often denote finite or periodic configurations by $Z$ (or similar notations) instead of $\gamma$, such that the notation looks a bit more similar to the one in \cite[Section 3]{GKR04}, \cite{Ru70}. For a finite configuration $Z\in \Gamma_{\R^d}=\Gamma$ we define the configurational energy
\begin{equation}\label{eqn:confen}
U_\phi(Z):=\sum_{\{x,y\}\subset Z} \phi(x-y)
\end{equation}
and for $Z',Z''\in \Gamma_{\R^d}$ being disjoint finite configurations we define the interaction energy
$$
W_\phi (Z',Z''):=U_\phi(Z'\cup Z'')-U_\phi(Z')-U_\phi(Z'')=\sum_{x\in Z',y\in Z''} \phi(x-y).
$$
It is well known (cf. \cite[Proposition 1.4]{Ru70}) that the assumptions (RP), (T) and (BB) imply
\begin{enumerate}
\item[(SS)] \emph{(superstability)} There exist $A>0$, $B\geq 0$ such that, if $Z$ is a finite configuration in $\R^d$, then
$$
U_\phi(Z)\geq A\sum_{r\in\Z^d} \sharp(Z\cap Q_1(r))^2-B\sharp Z.
$$
\item[(LR)] \emph{(lower regularity)} There exists a decreasing mapping $\Psi: \N_0\to [0,\infty)$ such that $\sum_{r\in\Z^d}\Psi(\vert r\vert)<\infty$ and for disjoint finite configurations $Z,Z'$ it holds
\begin{equation}\label{eqn:LReqn}
W_\phi(Z,Z')\geq -\sum_{r,r'\in\Z^d} \Psi(\vert r-r'\vert) \sharp(Z\cap Q_1(r))\sharp(Z'\cap Q_1(r)),
\end{equation}
\end{enumerate}
where $Q_1(r):=\left\{(x_1,\cdots,x_d)\in \R^d\Big |r_i-\frac 12\leq x_i<r_i+\frac 12\right\}$ for $r=(r_1,\cdots,r_d)\in \Z^d$.\\
In the case of pair interactions corresponding to a symmetric potential which is bounded from below (i.e.~the case we consider here), (LR) as given above is equivalent to (LR) as given in \cite{Ru70} and also to (LR) with (\ref{eqn:LReqn}) replaced by $W_\phi(\{x\},\{y\})\geq -\Psi(\vert r-r'\vert)$ for all $x\in Q_1(r), y\in Q_1(r')$, $r,r'\in \Z^d$, $x\neq y$.

\end{subsection}

\begin{subsection}{Potentials fulfilling (RP), (T), (BB) in periodic domains}\label{sub:properties}

For $\lambda>0$ we define $\Lambda_\lambda:=(-\lambda,\lambda]^d$. If $Z\in \Gamma_{\Lambda_\lambda}$, we define $\tilde{Z}\in \Gamma_{\R^d}$ to be the configuration resulting from $2\lambda$-periodic continuation of $Z$ to $\R^d$. A configuration $Z\in \Gamma_{\Lambda_\lambda}$ is said to have distances $<\lambda$, if it holds $\{((x_1,\cdots,x_d),(y_1,\cdots,y_d))\in Z\times Z| x_i-y_i=\lambda \mbox{ for some $1\leq i\leq d$}\}=\emptyset$. Note that when we consider $\Lambda_\lambda$ to have a periodic boundary, $\lambda$ is the maximal possible distance between two particles in $\Lambda_\lambda$. Usually (in the sense of canonical Gibbs measures in continuous systems) a configuration \emph{has} distances $<\lambda$.\shortspacing
In the case of periodic boundary condition we have to deal with the configurational energy of a finite configuration $Z\in \Gamma_{\Lambda_\lambda}$ with periodic boundary condition, which we define by
\begin{equation}\label{eqn:periodicenergy}
\widetilde{U}_{\phi,\lambda}(Z):=\sum_{\{x,y\}\subset Z} \sum_{r\in\Z^d} \phi(x-y+2\lambda r).
\end{equation}
\begin{remark}
Note that in this definition the interaction between one particle and its copies is ignored. This does not have consequences for the results derived below. In fact, the corresponding canonical Gibbs measures and their correlation functions are exactly the same as if these interactions were included.
\end{remark}
Temperedness of the potential $\phi$ ensures that the above definition makes sense as well as the following. We define for $\lambda>0$, $y\in \R^d$
$$
\phi_\lambda(y):=1_{(-\lambda,\lambda)^d}(y)\sum_{r\in \Z^d}\phi(y+2\lambda r).
$$
We use $\phi_\lambda$ in order to express $\widetilde{U}_{\phi,\lambda}$ in terms of a finite configuration (cf.~Lemma \ref{lem:mysterious} below). Possibly one would at first sight prefer to use the indicator function $1_{(-2\lambda,2\lambda)}$ instead of $1_{(-\lambda,\lambda)}$ to simplify this, but see Remark \ref{rem:nouniformbla} below.
\begin{lemma}\label{lem:mysterious}
There is a set $S\subset \Lambda_{2\lambda}\setminus \Lambda_\lambda$, such that for $Z\in \Gamma_{\Lambda_\lambda}$ having distances $<\lambda$ it holds
$$
\widetilde{U}_{\phi,\lambda}(Z)=W_{\phi_{\lambda}}(Z,\widetilde{Z}\cap S)+U_{\phi_\lambda}(Z)
$$
\end{lemma}
\begin{remark}
In order to avoid lenghty descriptions of shape of the set $S$, the assertion of Lemma \ref{lem:mysterious} looks more mysterious than necessary (see the proof below).
\end{remark}
\begin{proof}
First note that by assumption for any $x,y\in \widetilde{Z}$ the statement $x-y\in \Lambda_\lambda$ is equivalent to $\vert x-y\vert<\lambda$, which is symmetric in $x$, $y$. It holds
\begin{align}\label{eqn:billig1}
\widetilde{U}_{\phi,\lambda}(Z)&=\sum_{\{x,y\}\subset Z}\sum_{r\in \Z^d} \varphi(x-y+2\lambda r)\\
&= U_{\phi_\lambda}(Z)+\sum_{\stackrel{\{x,y\}\subset Z}{x-y\notin \Lambda_\lambda}}\sum_{r\in\Z^d} \phi(x-y+2\lambda r)\nonumber
\end{align}
We consider the set $\bauernkeks:=\{\{-r,r\}| r\in \Z^d, \vert r\vert=1\}$ and choose an arbitrary mapping $\chi: \bauernkeks\to \{r\in\Z^d| \vert r\vert=1\}$ such that $\chi(\{-r,r\})\in \{-r,r\}$ for any $r\in \Z^d$, $\vert r\vert=1$, i.e.~$\chi$ selects only one representant of each antipodal pair $\{-r,r\}$ from $\bauernkeks$. We define $S:=\bigcup_{r\in \chi(\bauernkeks)}(\Lambda_\lambda+2\lambda r)\cap \Lambda_{2\lambda}$.\shortspacing
Define $\mathcal X_1:=\{\{x,y\}\subset Z| \vert x-y\vert>\lambda\}$ and $\mathcal X_2:=\{(x,y)| x\in Z, y\in \widetilde{Z}\cap S, \vert x-y\vert<\lambda\}$. We define $\theta: \mathcal X_1\to \mathcal X_2$ in the following way. For $\{x,y\}\in \mathcal X_1$ there exists (uniquely) an $r_{\{x,y\}}\in \chi(\bauernkeks)$ such that $y-x\in \Lambda_\lambda-2\lambda r_{\{x,y\}}$ (w.l.o.g., possibly after interchanging $x$ and $y$). Then we set $\theta(\{x,y\}):=(x,y+2\lambda r_{\{x,y\}})$, which is in $\mathcal X_2$. Then $\theta$ is a bijection fulfilling $\sum_{r\in\Z^d}\phi(x-y+2\lambda r)=\phi_\lambda(x'-y')$ for any $\theta(\{x,y\})=(x',y')$, $\{x,y\}\in \mathcal X_1$. This and (\ref{eqn:billig1}) imply the assertion.
\end{proof}
\begin{remark}\label{rem:nouniformbla}
$\widetilde{U}_{\phi,\lambda}(Z)$, $Z\in \Gamma_{\Lambda_\lambda}$ can be easier expressed in terms of $\hat{\phi}_\lambda$, which we define by
\begin{equation}\label{eqn:hatphilambda}
\hat{\phi}_\lambda (y):=1_{(-2\lambda,2\lambda)^d}(y)\sum_{r\in\Z^d} \phi(y+2\lambda r)
\end{equation}
but below we prove properties of $\phi_\lambda$ which cannot be obtained for $\hat{\phi}_\lambda$. In particular, the latter potentials are not uniformly lower regular (or tempered).
\end{remark}

Let us now focus on properties of $\phi_\lambda$, $\lambda>0$, and the total energy $\widetilde{U}_{\phi,\lambda}$ with periodic boundary condition. We first observe that $\phi_\lambda$ fulfill uniformly in $\lambda\geq \lambda_0>0$ the conditions we imposed on $\phi$.

\begin{lemma}\label{lem:uniform1}
Let $\lambda_0>0$. There exist $\widetilde{R}$, $\widetilde{M}$, $\widetilde{G}$ in $\R^+$ and a decreasing continuous function $\widetilde{\Phi}: (0,\infty)\to [0,\infty)$ fulfilling $\lim_{s\to 0} \widetilde{\Phi}(s)s^d=\infty$ (which, as is possible by Lemma \ref{lem:Theothercapitalphi}, shall be continuously differentiable and such that $e^{-a\widetilde{\Phi}}\widetilde{\Phi}'$ is bounded for any $a>0$), such that
\begin{enumerate}
\item For all $\lambda\geq \lambda_0$ it holds $\phi_\lambda\geq -\widetilde{M}$.
\item For all $\lambda\geq \lambda_0$ it holds
$$
\vert \phi_\lambda(x)\vert\leq \widetilde{G}\vert x\vert^{-d-\varepsilon}\quad \mbox{whenever $\vert x\vert\geq \widetilde{R}$}.
$$
\item For all $\lambda\geq \lambda_0$ it holds
$$
\phi_\lambda(x)\geq \widetilde{\Phi}(\vert x\vert)\quad \mbox{whenever $\vert x\vert \leq \widetilde{R}$}.
$$
\item For all $c>0$ it holds
$$
\sup_{\lambda\geq \lambda_0}\sup_{\vert x\vert\geq c}\vert \phi_\lambda(x)\vert<\infty
$$
\end{enumerate}
\end{lemma}
\begin{proof}
We may w.l.o.g.~assume that in (RP), (T) it holds $R_1=R_2=:R<\lambda_0$. Temporarily we choose $\widetilde{R}:=R$.\\
For $y\in (-\lambda,\lambda)^d$ and $r\in\Z^d\setminus\{0\}$ it holds $\vert y+2\lambda r\vert\geq 2\lambda -\vert y\vert\geq \lambda\geq R$, so
\begin{align*}
\phi_\lambda(y)&=\sum_{r\in\Z^d} \phi(y+2\lambda r)\geq -M-G\sum_{r\in\Z^d\setminus\{0\}}\vert y+2\lambda r\vert^{-d-\varepsilon}\\
&\geq -M-G\sum_{r\in\Z^d\setminus\{0\}} (\vert 2\lambda r\vert-\vert y\vert)^{-d-\varepsilon}\geq -M-G\sum_{r\in\Z^d\setminus\{0\}}\vert \lambda r\vert^{-d-\varepsilon}\\
&\geq -M-G\lambda_0^{-d-\varepsilon} \sum_{r\in\Z^d\setminus\{0\}}\vert r\vert^{-d-\varepsilon}
\end{align*}
and the r.h.s.~is a constant larger than $-\infty$, which proves (i).\shortspacing
The same argument shows that for $y\in (-\lambda,\lambda)^d$, $\lambda\geq \lambda_0$, it holds
\begin{equation}\label{eqn:distphiphilambda}
\vert \phi_\lambda(y)-\phi(y)\vert \leq G \lambda^{-d-\varepsilon}\sum_{r\in\Z^d\setminus\{0\}}\vert r\vert^{-d-\varepsilon}\leq G \lambda_0^{-d-\varepsilon}\sum_{r\in\Z^d\setminus\{0\}}\vert r\vert^{-d-\varepsilon}.
\end{equation}
This proves (iv).\\
To show (ii), we define $\widetilde{G}_1:=G\left(1+\sum_{r\in\Z^d\setminus\{0\}}\vert r\vert^{-d-\varepsilon}\right)$. Let $\lambda\in (\lambda_0,\infty)$ and $x\in\R^d$, $\vert x\vert\geq \widetilde{R}$. If $\vert x\vert\geq \lambda$, then $\phi_\lambda(x)=0$ and there is nothing to prove. Therefore, let $\vert x\vert\leq \lambda$. It holds $\vert x+2\lambda r\vert\geq \vert \lambda r\vert\geq \lambda>\vert x\vert\geq R$ for all $r\in\Z^d\setminus\{0\}$. Hence
$$
\vert \phi_\lambda(x)\vert \leq \sum_{r\in\Z^d} \vert \phi(x+2\lambda r)\vert \leq G\vert x\vert^{-d-\varepsilon}+G\sum_{r\in\Z^d\setminus\{0\}} \vert \lambda r\vert^{-d-\varepsilon}\leq \widetilde{G}_1\vert x\vert^{-d-\varepsilon}
$$
proving (ii).\shortspacing
Finally, (\ref{eqn:distphiphilambda}) implies (iii) with $\widetilde{\Phi}:=\Phi-G\lambda_0^{-d-\varepsilon}\sum_{r\in\Z^d\setminus\{0\}}\vert r\vert^{-d-\varepsilon}$. Since this function $\widetilde{\Phi}$ might become negative away from $0$, we may have to choose $\widetilde{R}$ a bit smaller. By (iv) we see that then (ii) still holds with $\widetilde{G}_1$ replaced by some possibly larger constant $\widetilde{G}$.
\end{proof}

In Lemma \ref{lem:uniformprops} below the above result is used to prove that the $\phi_\lambda$ are superstable and lower regular \emph{uniformly} in $\lambda\geq \lambda_0$ and that moreover also the energy functions $\widetilde{U}_{\phi,\lambda}$ of configurations in $\Lambda_\lambda$ with periodic boundary are uniformly superstable. For obtaining the latter result we first need two simple technical lemmas.\\
For $\delta>0$ and $r\in\Z^d$ we set $Q_\delta(r):=\{(x_1,\cdots,x_d)\in\R^d| \delta(r_l-\frac12)\leq x_l<\delta(r_l+\frac12)\mbox{ for all $1\leq l\leq d$}\}$ and we define $\mathcal Q_\delta:=\{Q_\delta(r)| r\in\Z^d\}$.
\begin{lemma}\label{lem:simple1}
Let $\delta_1,\delta_2>0$ such that $\delta_2\geq \delta_1\geq \delta_2/2$. Then for all finite configurations $Z\in\Gamma_{\R^d}$ it holds
$$
6^{-d}\sum_{Q\in \mathcal Q_{\delta_1}} \sharp(Z\cap Q)^2\leq \sum_{Q'\in Q_{\delta_2}} \sharp(Z\cap Q')^2\leq 6^{d}\sum_{Q\in \mathcal Q_{\delta_1}} \sharp(Z\cap Q)^2
$$
\end{lemma}
\begin{proof}
This follows from the facts that for any $Q\in \mathcal Q_{\delta_1}$ it holds $\sharp\{Q'\in \mathcal Q_{\delta_2}| Q\cap Q'\neq \emptyset\}\leq 2^d$ and for any $Q'\in \mathcal Q_{\delta_2}$ it holds $\sharp\{Q\in Q_{\delta_1}| Q\cap Q'\neq \emptyset\}\leq 3^d$.
\end{proof}

\begin{lemma}\label{lem:simple2}
Let $\lambda_0>0$. There exists a constant $c_{\lambda_0}$ such that for all $\lambda\geq \lambda_0$, $Z\in \Gamma_{\Lambda_\lambda}$, $k\in\N$ it holds
$$
(2k+1)^dc_{\lambda_0}^{-1}\sum_{r\in\Z^d} \sharp(Z\cap Q_1(r))^2\leq \sum_{r\in\Z^d} \sharp(Z_k\cap Q_1(r))^2\leq (2k+1)^dc_{\lambda_0}\sum_{r\in\Z^d} \sharp(Z\cap Q_1(r))^2
$$
where $Z_k:=\bigcup_{r\in\Z^d, \vert r\vert\leq k}(Z+2\lambda r)$.
\end{lemma}
\begin{proof}
We have $(2k+1)^d \sum_{r\in\Z^d} (Z\cap Q_{2\lambda}(r))^2=\sum_{r\in\Z^d} (Z_k\cap Q_{2\lambda}(r))^2$ and the same holds when $2\lambda$ is replaced by some $\delta\in [\min\{2\lambda_0,1/3\},1]$ such that $2\lambda$ is an odd multiple of $\delta$. So by repeated application of Lemma \ref{lem:simple1} we obtain the assertion.
\end{proof}
\begin{lemma}\label{lem:uniformprops}
Let $\lambda_0>0$. It holds
\begin{enumerate}
\item There exists a decreasing mapping $\widetilde\Psi: \N_0\to [0,\infty)$ fulfilling $\sum_{r\in\Z^d}\widetilde\Psi(\vert r\vert)<\infty$ such that if $Z,Z'$ are disjoint finite configurations and $\lambda\geq \lambda_0$ it holds
$$
W_{\phi_\lambda}(Z,Z')\geq -\sum_{r,r'\in\Z^d} \widetilde\Psi(\vert r-r'\vert)\sharp(Z\cap Q_1(r))\sharp(Z'\cap Q_1(r')).
$$
\item There are constants $\widetilde{A}>0$, $\widetilde{B}\geq 0$ such that for all $\lambda\geq \lambda_0$ and all finite configurations $Z$, it holds
$$
U_{\phi_\lambda}(Z)\geq \widetilde{A}\sum_{r\in\Z^d} \sharp(Z\cap Q_1(r))^2-\widetilde{B}\sharp Z
$$
Moreover, after possibly enlarging $\widetilde{A}$, for all $\lambda\geq \lambda_0$ and $Z\in\Gamma_{\Lambda_\lambda}$ having distances $<\lambda$ it holds
$$
\widetilde{U}_{\phi,\lambda}(Z)\geq \widetilde{A}\sum_{r\in\Z^d} \sharp(Z\cap Q_1(r))^2-\widetilde{B}\sharp Z
$$
\end{enumerate}
\end{lemma}
\begin{proof}
We define a potential $\underline{\phi}: \R^d\to \R\cup\{\infty\}$ by
$$
\underline{\phi}(x):=\left\{\begin{array}{ll} \widetilde{\Phi}(\vert x\vert)& \mbox{if $\vert x\vert\leq \widetilde{R}$},\\-\widetilde{G}\vert x\vert^{-d-\varepsilon}&\mbox{if $\widetilde{R}\leq \vert x\vert$.}\end{array}\right.
$$
with $\widetilde{R}$, $\widetilde{G}$, $\widetilde{\Phi}$ as in Lemma \ref{lem:uniform1}. This potential fulfills (RP), (BB) and (T) and is therefore superstable and lower regular by \cite[Proposition 1.4]{Ru70}. Since $\phi_\lambda\geq \underline{\phi}$ for all $\lambda\geq \lambda_0$ this already implies (i) and the first assertion in (ii).\shortspacing
For $r\in\Z^d$ we set $\Lambda_{\lambda,r}:=\Lambda_\lambda+2\lambda r$. Let $k\in\N$ be a natural number and define $Z_k:=\bigcup_{r\in\Z^d, \vert r\vert\leq k} (Z+2\lambda r)$. It holds
$$
U_{\phi_\lambda}(Z_k)=\sum_{r\in\Z^d,\vert r\vert\leq k} \left(U_{\phi_\lambda}(\widetilde{Z}\cap \Lambda_{\lambda,r})+\sum_{r'\in\chi(\bauernkeks)} W_{\phi_\lambda}(\widetilde{Z}\cap \Lambda_{\lambda,r},\widetilde{Z}\cap \Lambda_{\lambda,r+r'})\eta(r,r',k)\right)
$$
where $\chi$ is defined as in the proof of Lemma \ref{lem:mysterious} and $\eta(r,r',k)=1$ for $\vert r+r'\vert\leq k$ and $0$ else. It holds $\eta(r,r',k)=1$ for $\vert r\vert\leq k-1$, thus by Lemma \ref{lem:mysterious}
\begin{align*}
U_{\phi_{\lambda}}(Z_k)=&\sum_{r\in\Z^d,\vert r\vert< k} \widetilde{U}_{\phi,\lambda}(Z)\\
&+\sum_{r\in\Z^d,\vert r\vert=k} \left(U_{\phi_\lambda}(\widetilde{Z}\cap \Lambda_{\lambda,r})+\sum_{r'\in\chi(\bauernkeks)} W_{\phi_\lambda}(\widetilde{Z}\cap \Lambda_{\lambda,r},\widetilde{Z}\cap \Lambda_{\lambda,r+r'})\eta(r,r',k)\right).
\end{align*}
But for $W_{\phi_\lambda}(\widetilde{Z}\cap \Lambda_{\lambda,r},\widetilde{Z}\cap \Lambda_{\lambda,r+r'})\eta(r,r',k)$, $r'\in\chi(\bauernkeks)$ there are only finitely many possible finite values, independently of $k$, hence there exists $C<\infty$ such that
$$
\vert U_{\phi,\lambda}(Z_k)-(2k-1)^d \widetilde{U}_{\phi,\lambda}(Z)\vert\leq C k^{d-1}
$$
proving that
$$
\widetilde{U}_{\phi,\lambda}(Z)=\lim_{k\to\infty} \frac{1}{(2k-1)^d} U_{\phi_\lambda}(Z_k)
$$
By the first assertion in (ii) and by Lemma \ref{lem:simple2} we conclude that
$$
\widetilde U_{\phi,\lambda}(Z)\geq c_{\lambda_0}\widetilde{A}\sum_{r\in\Z^d} \sharp(Z\cap Q_1(r))^2-\widetilde B\sharp Z
$$
\end{proof}

\end{subsection}

\begin{subsection}{Ruelle bound for canonical correlation functions with periodic boundary condition}\label{sub:Ruellebound}

Before going into the proof of the Ruelle bound we note a property of the canonical partition functions with periodic boundary stated in Lemma \ref{lem:minlos} below. Its proof is a slight adaptation of the proof of \cite[Lemma 3']{DM67} to the periodic boundary case (with external potential equal to $0$). Note that the result of the following lemma in particular holds for the type of potentials we consider in this section. Its assumptions are obviously weaker than (RP), (T), (BB).
\begin{lemma}\label{lem:minlos}
Let $\phi: \R^d\to\R\cup\{\infty\}$ be measurable, symmetric, bounded from below and such that for any $a>0$ it holds $C_a:=\int_{\{\vert x\vert\geq a\}}\vert \phi(x)\vert\,dx<\infty$. We define $\widetilde{U}_{\phi,\lambda}:=U_{\hat\phi_\lambda}$, where $\hat\phi_\lambda:=\sum_{r\in\Z^d}\phi(\cdot+2\lambda r)$ is defined as limit in $L^1_{\textnormal{loc}}((-2\lambda,2\lambda)\setminus\{0\};dx)$.\\
Consider for $\lambda>0$, $N\in\N_0$, $\beta>0$
$$
Z_{\lambda}^{N,\beta}:=\int_{\Lambda_\lambda^N}e^{-\beta \widetilde{U}_{\phi,\lambda}(\{x_1,\cdots,x_N\})}\,dx_1\cdots dx_N,
$$
which are ($N!$ times) the canonical partition functions with periodic boundary condition. Set $Z_{\lambda}^{0,\beta}:=1$. Let $S\subset [0,\infty)\times(0,\infty)$ be any compact subset. There exists a constant $k_{\phi,S}\geq 1$ such that for any $N\in\N_0,\lambda>0,\beta>0$ fulfilling $\left(\frac{N}{(2\lambda)^d},\beta\right)\in S$ it holds
$$
Z_{\lambda}^{N,\beta}\leq \frac{k_{\phi,S}}{(2\lambda)^{d}}{Z_{\lambda}^{N+1,\beta}}.
$$
\end{lemma}
\begin{proof}
Set $\rho_{\max}:=\sup\{\rho| \exists \beta\mbox{ such that }(\rho,\beta)\in S\}$ and choose $a>0$ small enough such that the volume $V_a=(2a)^d$ of a $\vert\cdot\vert$-ball with radius $a$ fulfills $V_a \rho_{\max}\leq \frac{1}{2}$. Then $N V_a\leq \frac{1}{2} (2\lambda)^d$. Fix $Z=(x_1,\cdots,x_N)\in \Lambda_{\lambda}^N$ and consider the set $\Lambda_{\lambda}^a:=\Lambda_\lambda\setminus\{\overline{B}_a(x_1)\cup\cdots\cup \overline{B}_a(x_N)\}$, where $\overline{B}_a(x):=\bigcup_{r\in\Z^d,\vert r\vert\leq 1}\{y\in \R^d\,|\,\vert y-(x+2\lambda r)\vert\leq a\}$, $x\in\Lambda_\lambda$. It holds
\begin{align*}
\int_{\Lambda_{\lambda}^a} \left\vert \sum_{i=1}^N \hat\phi_\lambda(\xi-x_i)\right\vert\,d\xi
&\leq \sum_{i=1}^N \int_{\Lambda_{\lambda}\setminus U_a(0)}\vert \hat\phi_\lambda(\xi)\vert\,d\xi\\
&\leq N \int_{\R^d\setminus U_a(0)} \vert \phi(x)\vert \,dx=N C_a\leq \rho_{\max} (2\lambda)^d C_a
\end{align*}
Consequently, $\left \{\xi\in \Lambda_\lambda\Big| \sum_{i=1}^N \hat\phi_\lambda(\xi-x_i)\leq 4\rho_{\max} C_a\right\}$ has volume of at least $\vol(\Lambda_\lambda^a)-\frac{(2\lambda)^d}4\geq  \frac{(2\lambda)^d}4$. (Here and in the sequel $\vol(\cdot)$ shall denote Lebesgue measure.) Hence
\begin{align*}
Z_{\lambda}^{N+1,\beta}&=\int_{\Lambda_\lambda^N} e^{-\beta \widetilde{U}_{\phi,\lambda}(\{x_1,\cdots,x_N\})}\int_{\Lambda_\lambda} e^{-\beta\sum_{i=1}^N \hat\phi_\lambda(x_i-\xi)}\,d\xi\,dx_1\cdots dx_N\\
& \geq \frac{(2\lambda)^d}{4} e^{-4 \beta\rho_{\max}C_a} Z_\lambda^{N,\beta}
\end{align*}
so the assertion holds with $k_{\phi,S}=4 e^{4\beta\rho_{\max}C_a}$.
\end{proof}
\ \\
Now, fix $\lambda_0>0$, $\beta>0$ and $\rho_{\max}>0$. $\rho_{\max}$ will be used below as a bound for the particle density. We choose sequences $(\phi_j)_{j\in \N}$, $(V_j)_{j\in\N}$ and $(l_j)_{j\in\N}$ and numbers $P\in\N$, $\alpha>0$ as in \cite[Section 2]{Ru70} corresponding to $\widetilde\Psi$, $\widetilde{A}$, $\widetilde{B}$ as in Lemma \ref{lem:uniformprops}. For $k:=k_{[0,\rho_{\max}]\times \{\beta\}}$ as in Lemma \ref{lem:minlos} we define $\gamma(\rho_{\max},\beta):=\frac{1}{\widetilde A}(\widetilde B+\beta^{-1}\ln(k))$.\\
We define $Q(j):=[-l_j-0.5,l_j+0.5]^d$, $j\in\N$, then $V_j$ is the volume of $Q(j)$ w.r.t.~Lebesgue measure.\shortspacing
The following is somehow obvious, but important.
\begin{lemma}\label{lem:ovbip}
For any $g>0$ there exists $\lambda_1(g)>0$ such that for all $\lambda\geq \lambda_1(g)$ it holds
\begin{equation}\label{eqn:obvbutimp}
(2\lambda)^d\,g<\psi_j V_j
\end{equation}
for some $j\in\N$ such that $Q(j+1)\subset \Lambda_{\lambda/2}$. Here $c_{\lambda_0}$ is as in Lemma \ref{lem:simple2}.
\end{lemma}
\begin{proof}
For any $\lambda>0$ large enough we can fix $j_\lambda$ such that $Q(j_\lambda+1)\subsetneq \Lambda_{\lambda/2}\subset Q(j_\lambda+2)$. Then by the definition of $V_j, l_j$ (cf.~\cite{Ru70})
$$
V_{j_\lambda}(1+3\alpha)^{2d}\geq V_{j_\lambda+2}\geq \lambda^d
$$
Thus (\ref{eqn:obvbutimp}) holds as soon as $\psi_{j_\lambda}>2^d(1+3\alpha)^{2d}g$. Hence our assertion follows from the fact that $j_{\lambda}\to\infty$ and consequently $\psi_{j_\lambda}\to\infty$ as $\lambda\to\infty$.
\end{proof}

We define $\lambda_*:=\max\left\{\lambda_0,\lambda_1\left(\gamma(\rho_{\max},\beta)\,c_{\lambda_0} \rho_{\max} 3^d\right)\right\}$, where $c_{\lambda_0}$ is as in Lemma \ref{lem:simple2} and $\lambda_1(\cdot)$ is as in the above Lemma. As in \cite{Ru70} we write $[j]:=\{r\in\Z^d| \vert r\vert\leq l_j\}$.
\begin{lemma}\label{lem:I,II,III}
Let $\lambda\geq \lambda_*$, and let $Z\in \Gamma_{\Lambda_\lambda}$ be such that $Z$ has distances $<\lambda$ and fulfills $\sharp Z\leq \rho_{\max}(2\lambda)^d$. Let $\overline{Z}:=\widetilde{Z}\cap (S\cup \Lambda_\lambda)$, where $S$ is as in Lemma \ref{lem:mysterious}. Then one of the following statements is valid:
\begin{enumerate}
\item[(I)] For all $j\geq P$ it holds
$$
\sum_{r\in [j]}\sharp(\overline{Z}\cap Q_1(r))^2\leq \psi_j V_j.
$$
\item[(II)] It holds
$$
\beta \widetilde{U}_{\phi,\lambda}(Z)\geq \ln(k)\sharp Z.
$$
\item[(III)] There exists a largest $q\geq P$ fulfilling
$$
\sum_{r\in [q]} \sharp (\overline{Z}\cap Q_1(r))^2\geq \psi_q V_q
$$
and it additionally holds $Q(q+1)\subset \Lambda_{\lambda/2}$.
\end{enumerate}
\end{lemma}

\begin{proof}
Let us at first consider the situation where $\sum_{r\in\Z^d}\sharp(Z\cap Q_1(r))^2\geq \gamma(\rho_{\max},\beta)\sharp Z$. Using Lemma \ref{lem:uniformprops}(ii) we find that
$$
\beta \widetilde{U}_{\phi,\lambda}(Z)\geq \beta\left({\widetilde{A}} \gamma(\rho_{\max},\beta) \sharp Z-\widetilde B \sharp Z\right)=\ln(k)\sharp Z,
$$
i.e.~(II) holds. Hence we may assume for the rest of the proof that
$$
\sum_{r\in\Z^d} \sharp(Z\cap Q_1(r))^2\leq \gamma(\rho_{\max},\beta)\sharp Z.
$$
Using Lemma \ref{lem:simple2}, the notations given there and the definition of $\lambda_*$ we find that
\begin{align*}
\sum_{r\in Z^d}\sharp(\overline{Z}\cap Q_1(r))^2&\leq \sum_{r\in\Z^d}\sharp (Z_1\cap Q_1(r))^2\leq \gamma(\rho_{\max},\beta)c_{\lambda_0} 3^d\sharp Z\\
&\leq c_{\lambda_0} \gamma(\rho_{\max},\beta)\,\rho_{\max} 3^d (2\lambda)^d<\psi_{j_0}V_{j_0}
\end{align*}
for some $j_0\in\N$ fulfilling $Q(j_0+1)\subset \Lambda_{\lambda/2}$ by Lemma \ref{lem:ovbip}. Consequently, for all $j\geq j_0$ it holds
\begin{equation}\label{eqn:qkleinmacher}
\sum_{r\in [j]}\sharp(\overline{Z}\cap Q_1(r))^2<V_{j_0}\psi_{j_0}\leq \psi_j V_j
\end{equation}
Now, if (I) is not valid, the existence of a largest $q\geq P$ such that $\sum_{r\in [q]}\sharp(\overline{Z}\cap Q_1(r))^2\geq \psi_q V_q$ is clear. But from (\ref{eqn:qkleinmacher}) we find that this number $q$ fulfills also the second condition in (III).
\end{proof}

Let us have another look at the energy in case (III). Set $C:=\frac{\widetilde{A}}{4}(1+3\alpha)^{-d-1}$ (this is the constant $C$ from \cite[Proposition 2.5]{Ru70}).
\begin{lemma}\label{lem:energyIII}
Let $\lambda\geq \lambda_*$. There exists a constant $\kappa$, not depending on $Z$ and $\lambda$, such that the following holds: If in Lemma \ref{lem:I,II,III} statement (III) is valid, then
$$
-\widetilde{U}_{\phi,\lambda}(Z)\leq -\widetilde{U}_{\phi,\lambda}(Z\cap Q(q+1)^c)-\frac{\widetilde A}{4}\sum_{r\in [q+1]} \sharp (Z\cap Q_1(r))^2-C \psi_{q+1} V_{q+1}+\kappa\sharp (Z\cap Q(q+1)).
$$
Moreover, there is another constant $\kappa'$ such that in the same situation
$$
-\widetilde{U}_{\phi,\lambda}(Z)\leq -\widetilde{U}_{\phi,\lambda}(Z\cap Q(q+1)^c)-(C \psi_{q+1} -\kappa')V_{q+1}-\ln(k)\sharp (Z\cap Q(q+1)).
$$
\end{lemma}
\begin{proof}
Let $\overline{Z}$ be defined as in Lemma \ref{lem:I,II,III} and define $\overline{Z}^{(q+1)}:=\widetilde{Z^{(q+1)}}\cap (S\cup \Lambda_\lambda)$, where $Z^{(q+1)}:=Z\cap Q(q+1)$ and $S$ is as in Lemma \ref{lem:mysterious}. It holds
\begin{align*}
-\widetilde{U}_{\phi,\lambda}(Z)=&-\widetilde{U}_{\phi,\lambda}(Z\cap Q(q+1)^c)-U_{\phi_\lambda}(Z^{(q+1)})\\
&- W_{\phi_\lambda}(Z^{(q+1)},\overline{Z}\setminus Q(q+1))-W_{\phi_\lambda} (Z\cap Q(q+1)^c,\overline{Z}^{(q+1)}\setminus (Z^{(q+1)})).
\end{align*}
Using \cite[Proposition 2.5a]{Ru70} we find that the first assertion is shown as soon as we can prove that
$$
-W_{\phi_\lambda}(Z\cap Q(q+1)^c,\overline{Z}^{(q+1)}\setminus (Z^{(q+1)}))\leq \kappa\sharp(Z^{(q+1)}).
$$
But this can be seen using Lemma \ref{lem:uniformprops}: Note that $\sharp(Z\cap Q(q+1)^c)\leq \sharp Z$ and $\sharp (\overline{Z}^{(q+1)}\setminus (Z^{(q+1)}))= \frac{3^d-1}{2} \sharp Z^{(q+1)}$. We obtain by the uniform lower regularity (Lemma \ref{lem:uniformprops}(i))
\begin{eqnarray*}
\lefteqn{-W_{\phi_\lambda}(Z\cap Q(q+1)^c,\overline{Z}_{q+1}\setminus (Z\cap Q(q+1)))}\\
& &\leq \frac{3^d-1}{2} \sharp Z\,\sharp(Z^{(q+1)})\widetilde\Psi\left(\left\lfloor \frac{3\lambda}2\right\rfloor- \left\lceil {\lambda}\right\rceil  \right)\\
& &\leq \frac{3^d-1}{2} \rho_{\max} \sharp(Z^{(q+1)}) (2\lambda)^d \widetilde\Psi\left(\left\lfloor \frac{3\lambda}2\right\rfloor- \left\lceil {\lambda}\right\rceil  \right).
\end{eqnarray*}
By the summability property of $\widetilde\Psi$ we know that $\lambda^d \widetilde\Psi\left(\left\lfloor \frac{3\lambda}2\right\rfloor- \left\lceil {\lambda}\right\rceil  \right)$ is bounded independently of $\lambda$. (It even tends to $0$ as $\lambda\to\infty$). Hence the first assertion follows.\shortspacing
The second assertion is seen from the first one, from the fact that there exists $\kappa'>0$ such that for any $l\in \N_0$ it holds
$$
-\frac{\widetilde A}{4}l^2+(\ln(k)+\kappa)l\leq \kappa'
$$
and from $V_{q+1}=\sharp [q+1]$.
\end{proof}

\begin{remark}\label{rem:tollesachesoeineruellebound}
Note that for the proofs of Ruelle bounds in \cite{Ru70} and \cite{GKR04} it is only necessary to consider the cases (I) and (III) as in Lemma \ref{lem:I,II,III}. In case (III) the restriction $Q(q+1)\subset \Lambda_{\lambda/2}$ does not occur there. For the periodic boundary case, however, a restriction on $q$ like this is essential in order to estimate the interaction term $-W_{\phi_\lambda}(Z\cap Q(q+1)^c,\overline{Z}^{(q+1)}\setminus (Z^{(q+1)}))$ in the proof of Lemma \ref{lem:energyIII}. For this reason (II) is considered as a separate case: When one chooses $\lambda_*$ large enough and for some configuration (III) holds with $q$ being too large, the total periodic energy of the configuration is large enough to be estimated from below in a suitable way. The meaning of this estimate and the other estimates in Lemmas \ref{lem:I,II,III} and \ref{lem:energyIII} becomes clear in the proof of Theorem \ref{thm:RB} below (which works as in \cite{GKR04} or \cite{Ru70}).
\end{remark}

\noindent We are now prepared to prove the main result of this section.
\begin{theorem}\label{thm:RB}
Let $\phi$ be a pair potential fulfilling (RP), (BB), (T) given in Section \ref{sub:conditions} and let $\rho_{\max}>0$, $\beta>0$.\\
Then there exists a constant $\xi>0$ and some $\lambda_*>0$ such that the following holds:\\
For all $\lambda\geq \lambda_*$ and $n\in \N_0$, $N\in\N$ fulfilling $n\leq N\leq \rho_{\max} (2\lambda)^d$ the canonical correlation function with periodic boundary, given by
$$
k^{(n,N)}_{\lambda}(x_1,\cdots,x_n):=\frac{N!}{(N-n)!} \frac{1}{Z_\lambda^{N,\beta}} \int_{\Lambda_\lambda^{(N-n)}} e^{-\beta \widetilde{U}_{\phi,\lambda}(\{x_1,\cdots,x_N\})}dx_{n+1}\cdots dx_N,
$$
$x_1,\cdots,x_n\in \Lambda_\lambda$, is bounded by $\xi^n$.
\end{theorem}
\begin{remark}
The above definition is supposed to imply that one sets $k^{(0,N)}_\lambda=1$ for any $N\in\N$, $\lambda>0$. Moreover $k^{(N,N)}(x_1,\cdots,x_N)=N!\frac{1}{Z_\lambda^{N,\beta}} e^{-\beta\widetilde{U}_{\phi,\lambda}(\{x_1,\cdots,x_N\})}$, $x_1,\cdots,x_N\in\Lambda_\lambda$.
\end{remark}

\begin{proof}
Choose $\lambda_*$, $C$, $k$ etc.~as above, let $D<\infty$ be as in \cite[Proof of Proposition 2.6]{Ru70}. Set
$$
\xi:=\max\left\{\rho_{\max}\left(1+k e^{-\beta D}+\sum_{q\geq P}e^{-(\beta C\psi_{q+1}-\beta\kappa'-\rho_{\max})V_{q+1}}\right),1\right\},
$$
which is $<\infty$, since $\psi_{q+1}\to\infty$ as $q\to\infty$ and $V_{q+1}$ grows at least as fast as $q$. The proof is done by induction on $n$. For $n=0$ the assertion is trivially fulfilled.\shortspacing
By Lebesgue's dominated convergence theorem we may w.l.o.g.~assume that $\{x_1,\cdots,x_n\}$ has distances $<\lambda$. Moreover, we may assume that $x_1=0$. Let $S^{I}$, $S^{II}$ and $S^{III}$ be the subsets of tupels $(x_{n+1},\cdots,x_N)\in \Lambda_\lambda^{N-n}$ such that $Z:=\{x_1,\cdots,x_N\}$ has distances $<\lambda$ and satisfies (I), (II), (III) in Lemma \ref{lem:I,II,III}, respectively. Denote by $S_{q,l}^{III}$ the subset of $S^{III}$ such that $q$ is as in Lemma \ref{lem:I,II,III}(III) and $l=\sharp(\{x_{n+1},\cdots,x_N\}\cap Q(q+1))$.\shortspacing
Let $(x_{n+1},\cdots,x_N)\in S^{I}$. Then as in \cite[Proof of Proposition 2.6]{Ru70} we find that
$$
W_{\phi_\lambda}(\{x_1\},\{x_2,\cdots,x_N\})\leq D.
$$
Hence, since $x_1=0$ we have 
\begin{align*}
\widetilde{U}_{\phi,\lambda}(\{x_1,\cdots,x_N\})&=\widetilde{U}_{\phi,\lambda}(\{x_2,\cdots,x_N\})+W_{\phi_\lambda}(\{x_1\},\{x_2,\cdots,x_N\})\\
&\leq \widetilde{U}_{\phi,\lambda}(\{x_2,\cdots,x_N\})+D.
\end{align*}
Thus
\begin{eqnarray}\label{eqn:rb1}
\lefteqn{\frac{N!}{(N-n)!}\frac{1}{Z_\lambda^{N,\beta}} \int_{S^I}e^{-\beta \widetilde{U}_{\phi,\lambda}(\{x_1,\cdots,x_N\})}dx_{n+1}\cdots dx_N}\\
& & \leq e^{-\beta D} \frac{N!}{(N-n)!} \frac{1}{Z_\lambda^{N,\beta}}\int_{\Lambda_\lambda^{N-n}}e^{-\beta \widetilde{U}_{\phi,\lambda}(\{x_2,\cdots,x_N)\}}dx_{n+1}\cdots dx_N\nonumber\\
& & \leq e^{-\beta D} N \frac{k}{(2\lambda)^d} k_\lambda^{(n-1,N-1)}(x_2,\cdots,x_n)\leq e^{-\beta D} k \rho_{\max} \xi^{n-1}.\nonumber
\end{eqnarray}
by Lemma \ref{lem:minlos}.\shortspacing
Now let us consider the configurations in $S^{II}$. Here Lemma \ref{lem:I,II,III} and \ref{lem:minlos} yield
\begin{eqnarray}\label{eqn:rb2}
\lefteqn{\frac{N!}{(N-n)!} \frac{1}{Z_\lambda^{N,\beta}}\int_{S^{II}} e^{-\beta \widetilde{U}_{\phi,\lambda}(\{x_1,\cdots,x_N\})}dx_{n+1}\cdots dx_N}\\
& & \leq N^n \frac{1}{Z_\lambda^{N,\beta}} \left((2\lambda)^d\right)^{N-n}k^{-N}\leq N^n \left(\frac{k}{(2\lambda)^d}\right)^N\left((2\lambda)^d\right)^{N-n}k^{-N}\nonumber\\
& & \leq \rho_{\max}^n\leq \rho_{\max} \xi^{n-1}.\nonumber
\end{eqnarray}
We finally turn to $S_{q,l}^{III}$, $q\geq P$, $0\leq l\leq N-n$. Denote $N(q):=\sharp(\{x_1,\cdots,x_n\}\cap Q(q+1))\geq 1$ and assume w.l.o.g.~that $x_1,\cdots,x_{N(q)}\in Q(q+1)$. We set $\chi_q:=e^{-\beta(C\psi_{q+1}-\kappa')V_{q+1}}$. Lemma \ref{lem:energyIII} shows that
\begin{align*}
\lefteqn{\frac{N!}{(N-n)!} \frac{1}{Z_\lambda^{N,\beta}} \int_{S_{q,l}^{III}}e^{-\beta \widetilde{U}_{\phi,\lambda}(\{x_1,\cdots,x_N\})}\, dx_{n+1}\cdots dx_N}\\
& \,\,\leq \chi_q \frac{N!}{(N-n)!} \frac{1}{Z_\lambda^{N,\beta}} \frac{V_{q+1}^l}{k^{N(q)+l}}\left(\hspace{-0.25cm}\begin{array}{c}N-n\\N-n-l\end{array}\hspace{-0.25cm}\right)\int_{\Lambda_\lambda^{N-n-l}} e^{-\beta \widetilde{U}_{\phi,\lambda}(\{x_{N(q)+1},\cdots,x_{N-l}\})}dx_{n+1}\cdots dx_{N-l}\\
& \,\,=\chi_q \frac{N!}{(N-N(q)-l)!l!}\frac{V_{q+1}^l}{k^{N(q)+l}} \frac{Z_{\lambda}^{(N-N(q)-l)}}{Z_\lambda^{N,\beta}}k_\lambda^{(N-l-N(q),n-N(q))}(x_{N(q)+1},\cdots,x_n)\\
& \,\,\leq \chi_q \frac{N^{N(q)+l}}{l!}\frac{V_{q+1}^l}{k^{N(q)+l}}\left(\frac{k}{(2\lambda )^d}\right)^{N(q)+l} \xi^{n-N(q)}\\
& \,\,\leq \chi_q \rho_{\max}^{N(q)}\xi^{n-N(q)}\frac{\rho_{\max}^lV_{q+1}^l}{l!}\leq \chi_q \rho_{\max} \xi^{n-1} \frac{(\rho_{\max} V_{q+1})^l}{l!}.
\end{align*}
Summing over $q\geq P$ and $l$ we obtain
\begin{multline}\label{eqn:rb3}
\frac{N!}{(N-n)!} \frac{1}{Z_\lambda^{N,\beta}} \int_{S^{III}}e^{-\beta \widetilde{U}_{\phi,\lambda}(\{x_1,\cdots,x_N\})}dx_{n+1}\cdots dx_N\\
\leq \xi^{n-1}\rho_{\max} \sum_{q\geq P}e^{-(\beta C\psi_{q+1}-\beta\kappa'-\rho_{\max})V_{q+1}}
\end{multline}
The assertion is implied by our choice of $\xi$, (\ref{eqn:rb1}), (\ref{eqn:rb2}), (\ref{eqn:rb3}) and the fact that the set of tupels $(x_{n+1},\cdots,x_N)\in\Lambda_\lambda^{N-n}$ such that $\{x_1,\cdots,x_N\}$ has distances $<\lambda$ has full Lebesgue measure in $\Lambda_\lambda^{N-n}$.
\end{proof}
\ \\
As in \cite[Theorem 3.2]{GKR04} in the case of empty boundary condition one also obtains an improved Ruelle bound.
\begin{corollary}\label{cor:iRB}
Under the assumptions of Theorem \ref{thm:RB}, there exists a constant $\zeta\geq \xi$ such that for all $\lambda\geq \lambda_*$ and $n\in \N_0$, $N\in\N$ fulfilling $n\leq N\leq \rho_{\max}(2\lambda)^d$ it holds
\begin{equation}\label{eqn:imprRBinf}
k_\lambda^{(n,N)}(x_1,\cdots,x_n)\leq \zeta^n \inf_{1\leq i\leq n} e^{-\beta\sum_{j\neq i}\hat\phi_{\lambda}(x_i-x_j)}
\end{equation}
for all $x_1,\cdots,x_n\in \Lambda_\lambda$. It follows also that
\begin{align*}
k_\lambda^{(n,N)}(x_1,\cdots,x_n)&\leq \zeta^n e^{-\frac{2\beta}{n}\sum_{\{i,j\}\subset \{1,\cdots,N\}}\hat\phi_{\lambda}(x_i-x_j)}\\
&= \zeta^n e^{-\frac{2\beta}{n}\widetilde{U}_{\phi,\lambda}(\{x_1,\cdots,x_n\})}\nonumber
\end{align*}
\end{corollary}
\begin{proof}
The proof is a slight modification of the proof of the second assertion of \cite[Theorem 3.2]{GKR04}.\\
Since the canonical ensemble w.r.t.~$\phi$ with periodic boundary condition is the same as the canonical ensemble w.r.t.~$\hat\phi_{\lambda}$ having empty boundary condition, we have the following Kirkwood-Salsburg type equation:
\begin{eqnarray*}
\lefteqn{k_{\lambda}^{(n,N)}(x_1,\cdots,x_n)}\\
& &=N\frac{Z_\lambda^{N-1,\beta}}{Z_\lambda^{N,\beta}} \exp \left(-\beta \sum_{2\leq i\leq N}\hat\phi_{\lambda}(x_1-x_j)\right)\Bigg(k_\lambda^{(n-1,N-1)}(x_2,\cdots,x_N)\\
& &+\sum_{l=1}^{N-n} \frac{1}{l!}\int_{\Lambda_\lambda^l} k_\lambda^{(n+l-1,N-1)}(x_2,\cdots,x_n,y_1,\cdots,y_l)\prod_{i=1}^l \left(e^{-\beta \hat\phi_{\lambda}(x_1-y_i)}-1\right)dy_1\cdots dy_l\Bigg)
\end{eqnarray*}
(cf. \cite{GKR04} or \cite[Equation (38.16)]{Hill56}). We may assume that any tupels occurring in this formula have distances $<\lambda$ and by translation invariance we are allowed to assume that $x_1=0$. Then under the integral sign we may replace $\hat{\phi}_\lambda$ by $\phi_\lambda$. But $\phi_\lambda$ fulfills (T) and (BB) uniformly in $\lambda\geq \lambda_0$. Thus
$$
I_\lambda:=\int_{\R^d} \vert e^{-\beta\phi_\lambda(y)}-1\vert\,dy<\infty,
$$
is bounded independently of $\lambda\geq \lambda_0$. Therefore, by Lemma \ref{lem:minlos} and Theorem \ref{thm:RB} we have
\begin{eqnarray*}
\lefteqn{k_\lambda^{(n,N)}(x_1,\cdots,x_n)\leq \exp\left(-\beta\sum_{2\leq i\leq n} \hat{\phi}_\lambda(x_1-x_i)\right) k\rho_{\max} \left(\xi^{n-1}+\sum_{l=1}^{N-n} \frac{\xi^{n+l-1}I^k}{l!}\right)}\\
& & \leq \exp\left(-\beta \sum_{2\leq i\leq n}\hat{\phi}_\lambda(x_1-x_i)\right) \xi^{n-1} k\rho_{\max} e^{\xi I}\leq \exp\left(-\beta \sum_{2\leq i\leq n}\hat{\phi}_\lambda(x_1-x_i)\right) \zeta^n
\end{eqnarray*}
where $\zeta:=\max\{\xi,k\rho_{\max}e^{\xi I}\}$. Symmetry of the correlation function implies the assertions.
\end{proof}

\end{subsection}

\begin{subsection}{Weak limits of measures and Ruelle bounds}\label{sub:weaklimits}

In this section we prove that a uniform (improved) Ruelle bound is transported to weak limits $\mu_n\to\mu$ of measures on $\Gamma$. Moreover, we prove that for a large class of functions $f: \Gamma_0\to\R$ defined on the space $\Gamma_0:=\{\hat{\gamma}\subset \R^d| \sharp\hat\gamma<\infty\}$ of finite configurations it holds $\mu_n(Kf)\to\mu(Kf)$ and that one may find bounded continuous local functions approximating $Kf$ uniformly in $L^1(\Gamma;\mu_n)$, $n\in\N$, and in $L^1(\Gamma;\mu)$. Here $Kf:\Gamma\to\R$ denotes the $K$-transform of $f$, given by $Kf(\hat\gamma):=\sum_{\hat\eta\subset\hat\gamma,\hat\eta\in\Gamma_0}f(\hat\eta)$. For further information on this mapping see \cite{Ku99}. These results then also hold for $\Gamma_0$, $\Gamma$ replaced by the velocity marked spaces $\Gamma^v_0$, $\Gamma^v$, when one assumes that the velocities are independently Gaussian distributed and do also not depend on the configuration. This is (basically) seen with the help of Lemma \ref{lem:vmeasure} below.\shortspacing
Let us at first collect some more notations (cf.~\cite{Ku99}). By $\Gamma_\Lambda$ we denote the subset of $\Gamma$ consisting of configurations contained in $\Lambda\subset\R^d$. Let now $\Lambda\subset\R^d$ be open. The projection $p_\Lambda: \Gamma\to \Gamma_\Lambda$ mapping $\gamma\in\Gamma$ to $\gamma\cap\Lambda\in \Gamma_\Lambda$ is continuous, when $\Gamma_\Lambda$ and $\Gamma$ are equipped with the vague topology, which we will always assume below. This means, we equip $\Gamma_\Lambda$ (resp.~$\Gamma$) with the vague topology on the set of Radon measures on $\Lambda$ (resp. $\R^d$). Moreover, these spaces shall be equipped with the corresponding Borel $\sigma$-fields.\\
We denote by $\Gamma_n\subset \Gamma_0$ the set of $n$-point configurations and by $\Gamma_{\Lambda,n}\subset\Gamma_\Lambda$ the set of $n$-point configurations contained in $\Lambda$, $\Lambda\subset\R^d$ open, bounded. For measurable and topological structures on these spaces we refer to \cite{Ku99} and to the considerations around Lemma \ref{lem:nocheinlemma} below. We denote the Borel $\sigma$-field on $\Gamma_0$ by $\mathcal B(\Gamma_0)$. Let $\Lambda\subset\R^d$ be open and bounded. In the sequel we use the fact that when we consider $\Gamma_{\Lambda}\subset\Gamma_0$ as a topological, hence as a measurable space, the corresponding measurable structure coincides with the one on $\Gamma_\Lambda$ induced by the vague topology (cf.~\cite[Remark 2.1.2]{Ku99}). This implies that $p_\Lambda: \Gamma\to\Gamma_{\Lambda}\subset \Gamma_0$ is measurable.\\[1ex]
The Lebesgue-Poisson measure $\lambda$ on $\Gamma_0$ is defined by
$$
\lambda(A):=\sum_{n=0}^\infty \frac{1}{n!}\int_{(\R^d)^n} 1_A(\{x_1,\cdots,x_n\})\,dx_1\cdots dx_n\quad\mbox{for $A\in \mathcal B(\Gamma_0)$}.
$$
A measure $\mu$ on $\Gamma$ is said to be locally absolutely continuous w.r.t.~Lebesgue-Poisson measure if for each open bounded $\Lambda\subset\R^d$ the image measure $\mu\circ p_\Lambda^{-1}$ is absolutely continuous w.r.t.~the restriction of $\lambda$ to $\Gamma_\Lambda$, considered as a subset of $\Gamma_0$. For any such probability measure $\mu$ one defines the correlation functional $\rho_\mu: \Gamma_0\to\R$ by
$$
\rho_\mu(\hat\gamma'):=\int_{\Gamma_{\Lambda}} \frac{d(\mu\circ p_\Lambda^{-1})}{d\lambda}(\hat\gamma\cup \hat\gamma')d\lambda(\hat\gamma)\quad\mbox{when $\hat\gamma'\in \Gamma_\Lambda$, $\Lambda\subset\R^d$ open, bounded}.
$$
In the same manner we define $\Gamma_\Lambda^v$, $p_\Lambda^v$, etc., but we replace the vague topology by the topology generated by bounded continuous functions with spatially bounded support. We define $\lambda^v$ to be the Lebesgue-Poisson measure corresponding to the intensity measure $\frac{1}{\sqrt{2\pi/\beta}^d}e^{-\beta v^2/2}d(x,v)$ (cf.~\cite[Chapter 3.1.3]{Ku99}), i.e.
$$
\lambda^v(A):=\sum_{n=0}^\infty \frac{1}{n!}\frac{1}{\sqrt{2\pi/\beta}^{nd}}\int_{(\R^d\times\R^d)^n} 1_A(\{(x_1,v_1),\cdots,(x_n,v_n)\}) e^{-\beta (v_1^2+\cdots+v_n^2)/2}\,dx_1\cdots dv_n
$$
for $A\in \mathcal B(\Gamma_0^v)$. We also define for a function $f: \Gamma_0^v\to\R$ similarly to the unmarked case $Kf(\gamma):=\sum_{\eta\subset\gamma, \eta\in\Gamma_0^v}f(\eta)$.
\shortspacing
The difference between the velocity marked situation and the unmarked situation is negligible for the sort of results we derive below, if the velocities are assumed to be distributed independent and Gaussian. This is (mainly) seen by Lemma \ref{lem:vmeasure} below. We call a measurable function $F:\Gamma^v\to\R$ (resp.~$F: \Gamma\to\R$) a cylinder function, if for some bounded measurable $\Lambda\subset\R^d$ it holds $F=F\circ\pr_\Lambda^v$ (resp.~$F=F\circ\pr_\Lambda$). We need one preliminary observation:
\begin{lemma}\label{lem:conv}
A sequence $(\nu_n)_{n\in\N}$ of probability measures on $\Gamma^v$ converges weakly to a probability measure $\nu$ if $\nu_n(F)\to\nu(F)$ as $n\to\infty$ holds for all bounded continuous cylinder functions $F: \Gamma^v\to\R$.\\
A similar statement holds for probability measures on $\Gamma$.
\end{lemma}
\begin{proof}
Let $f_k$, $g_k$, $k\in\N$ be as in the definition of the metric $d_\star$ in (\ref{eqn:apfle}). Let $\mathcal O_{\R^d}$ be a countable base of the topology of $\R^d$. Set $\widetilde {\mathcal O}_{\Gamma^v}:=\{\langle f_k,\cdot\rangle^{-1}(U)|U\in\mathcal O_{\R^d}, k\in\N\}\cup \{\langle g_k,\pr_x\cdot\rangle^{-1}(U)|U\in\mathcal O_{\R^d}, k\in\N\}$. The set $\mathcal O_{\Gamma^v}$ of finite intersections of sets from $\widetilde {\mathcal O}_{\Gamma^v}$ forms a countable base of the topology of $\Gamma^v$ consisting of cylinder sets (i.e.~sets whose indicator functions are cylinder functions). Now one verifies that the indicator function of each element of $\widetilde {\mathcal O}_{\Gamma^v}$ is the monotone limit of bounded continuous cylinder functions. This extends to sets from $\mathcal O_{\Gamma^v}$, and to countable unions of such sets, i.e.~to all open sets in $\Gamma^v$. From this one can derive that $\liminf_n \mu_n(O)\geq\mu(O)$ holds for each open set $O\subset\Gamma^v$ implying weak convergence.\\
The second assertion is shown analogously.
\end{proof}

\begin{lemma}\label{lem:vmeasure}
Let $\mu$ be a probability measure on $\Gamma$ which is locally absolutely continuous w.r.t.~Lebesgue-Poisson measure $\lambda$. Then there exists a unique measure $\mu^v$ on $\Gamma^v$, defined via 
\begin{equation}\label{eqn:vdensity}
\frac{d(\mu^v\circ(\pr_\Lambda^v)^{-1})}{d\lambda^v}(\{(x_1,v_1),\cdots,(x_k,v_k)\})=\frac{d(\mu\circ\pr_\Lambda^{-1})}{d\lambda}(\{x_1,\cdots,x_k\})
\end{equation}
for any $\{(x_1,v_1),\cdots,(x_k,v_k)\}\in \Gamma^v_\Lambda$, $k\in\N_0$, $\Lambda\subset\R^d$ open, bounded. \\[0.4ex]
Moreover, for measures $\mu$, $\mu_n$, $n\in\N$, on $\Gamma$ which are locally absolutely continuous w.r.t.~Lebesgue-Poisson measure one obtains 
\begin{enumerate}
\item $\mu_n\to \mu$ weakly iff $\mu_n^v\to \mu^v$ weakly.
\item $(\mu_n)_n$ is tight iff $(\mu_n^v)_n$ is tight.
\item For any nonnegative measurable $f: \Gamma_0^v\to\R_0^+$ it holds
$$
\mu^v(Kf)=\mu(Kf_*)
$$
where 
\begin{multline*}
\quad\quad\quad\quad f_*(\{(x_1,\cdots,x_k\}):=\frac{1}{\sqrt{2\pi/\beta}^{kd}}\int_{\R^{kd}} f(\{(x_1,v_1),\cdots,(x_k,v_k)\})\times\\ e^{-(\beta/2)(v_1^2+\cdots+v_k^2)}\,dv_1\cdots dv_k
\end{multline*}
for $\{x_1,\cdots,x_k\}\in \Gamma_0$, $k\in\N_0$. Moreover $\mu_n(Kf_*)\to \mu(Kf_*)$ as $n\to\infty$ iff $\mu_n^v(Kf)\to \mu^v(Kf)$ as $n\to\infty$.
\item For the correlation functionals it holds $\rho^v_{\mu^v}({\gamma})=\rho_\mu(\pr_x {\gamma})$, when we define $\rho^v_{\mu^v}$ analogously to $\rho_\mu$.
\end{enumerate}
\end{lemma}
\begin{proof}
Existence and uniqueness are seen using Kolmogorov's theorem (cf.~\cite[Section 3.1.3]{Ku99}).\\
We now prove (i): Note that $\lambda=\lambda^v\circ\pr_x^{-1}$. Applying the uniqueness part of Kolmogorov's theorem to the unmarked case, using (\ref{eqn:vdensity}) and noting that $\pr_x\circ\pr^v_\Lambda=\pr_\Lambda\circ\pr_x$ for all $\Lambda\subset\R^d$ open, bounded, we find that $\mu=\mu^v\circ\pr_x^{-1}$, $\mu_n=\mu_n^v\circ\pr_x^{-1}$, $n\in\N$. The fact that $\mu_n^v\to\mu^v$ implies $\mu_n\to\mu$ is now seen by continuity of the projection $\pr_x: \Gamma^v\to\Gamma$.\\
Conversely, assume that $\mu_n\to\mu$ weakly. We use Lemma \ref{lem:conv}. Let $F: \Gamma^v\to\R$ be a bounded continuous cylinder function. So, there exists $\Lambda\subset\R^d$ bounded, measurable such that $F=F\circ\pr_{\Lambda}$. Define the bounded function $F_*: \Gamma\to\R$ by
$$
F_*(\eta\cup \{x_1,\cdots,x_k\}):=\frac{1}{\sqrt{2\pi/\beta}^{kd}}\int_{\R^{kd}} F(\{(x_1,v_1)\cdots,(x_k,v_k)\}){e^{-\beta \sum_{i=1}^k v_i^2/2}}\,dv_1\cdots dv_k
$$
for $\eta\in \Gamma_{\Lambda^c}$ and $x_1,\cdots,x_k\in\Lambda$ pairwise distinct, $k\in\N$. (Note that by continuity the integrand is measurable. Measurability of $F_*: \Gamma\to\R$ follows e.g.~from its continuity, which is shown below.) This definition is independent of $\Lambda$ as long as $F\circ\pr_\Lambda^v=F$ and from the definition of $\mu^v$ one finds that $\mu(F_*)=\mu\circ\pr_\Lambda^{-1}(F_*)=\mu^v\circ (\pr_\Lambda^v)^{-1}(F)=\mu^v(F)$ and also $\mu_n(F_*)=\mu^v_n(F)$. So, using Lemma \ref{lem:conv} it remains to prove that $F_*: \Gamma\to\R$ is continuous. This is not immediate, since we are dealing with topologies on configuration space here. So, let $\hat\gamma^n\to\hat\gamma$ in $\Gamma$. Choose $\Lambda$ as above such that $\Lambda$ is open and $\partial\Lambda\cap\hat\gamma=\emptyset$. Let $\Lambda\cap \hat\gamma:=\{x_1,\cdots,x_k\}$, $k\in\N$. By a construction as in \cite[Proof of Proposition 4.1.5]{Ku99} one finds that for large $n$ it holds $\hat\gamma^n\cap \Lambda=\{x_1^{n},\cdots,x_k^{n}\}$ such that $x_i^n\to x_i$ as $n\to\infty$ for all $i=1,\cdots,k$. Finally, we note that this implies that $\{(x_1^n,v_1),\cdots,(x_k^n,v_k)\}\to \{(x_1,v_1),\cdots,(x_k,v_k)\}$ in $\Gamma^v$ as $n\to\infty$ for any $v_1,\cdots,v_k\in\R^d$, thus $F_*(\hat\gamma^{n})\to F_*(\hat\gamma)$ follows from Lebesgue's dominated convergence theorem, concluding the proof of (i).\shortspacing
(ii) follows from (i) and the fact that we are dealing with Polish spaces here, so tightness and relative compactness are equivalent by Prokhorov's theorem.\shortspacing
(iii): If $f$ has support in $\Gamma_{\Lambda,m}$ for some open bounded $\Lambda$ and some $m\in\N$, the first statement is seen by a calculation as in the proof of (i). By monotone convergence the assertion extends to $f: \Gamma_m\to\R_0^+$ and also to general $f: \Gamma_0\to\R_0^+$. The second statement follows from the first one.\shortspacing
(iv) is seen from the definitions.
\end{proof}

If a probability measure $\mu$ on $\Gamma$ is locally absolutely continuous w.r.t.~Lebesgue-Poisson measure and its correlation functional fulfills
\begin{equation}\label{eqn:RBound}
\rho_{\mu}(\eta)\leq \xi^{\sharp\eta}\quad\mbox{for all $\eta\in \Gamma_0$,}
\end{equation}
it is said to fulfill a Ruelle bound. Note that by \cite[Proposition 4.2.2]{Ku99} (or by Proposition \ref{prop:tolltolltoll} below) any measure fulfilling a Ruelle bound possesses finite local moments, i.e.
$$
\mu(\sharp(\cdot\cap \Lambda)^m)<\infty
$$
holds for any relatively compact $\Lambda\subset\R^d$ and $m\in\N$.\\[0.5ex]
For the unmarked situation the following lemma is a special case of \cite[Theorem 4.2.11]{Ku99}. In the velocity marked case it can be shown analogously or using Lemma \ref{lem:vmeasure}(iii) above. \\For a measure $\mu$ on $\Gamma_0$ (resp.~$\Gamma_0^v$) and a nonnegative measurable function $f: \Gamma_0\to\R$ (resp.~$f: \Gamma_0^v\to\R$) we denote by $f(\cdot)\mu$ or $f\mu$ the measure having density $f$ w.r.t.~$\mu$.
\begin{lemma}\label{lem:vonkuna}
Let $\mu$ be a measure on $\Gamma$ being locally absolutely continuous w.r.t.~Lebesgue-Poisson measure. Let $f\in L^1(\Gamma_0;\rho_\mu(\cdot)\lambda)$. \\Then $Kf\in L^1(\Gamma;\mu)$, $\Vert Kf\Vert_{L^1(\Gamma;\mu)}\leq \Vert  K\vert f\vert\,\Vert_{L^1(\Gamma;\mu)}\leq \Vert f\Vert_{L^1(\Gamma_0;\rho_\mu(\cdot)\lambda)}$ and
$$
\int_{\Gamma_0} f(\eta)\rho_\mu(\eta)d\lambda(\eta)=\int_{\Gamma} (Kf)(\gamma)d\mu(\gamma).
$$
In particular, the sum defining $Kf$ converges $\mu$-a.s.~absolutely. The same holds with $\mu$, $\rho_\mu$, $\lambda$, $\Gamma_0$, $\Gamma$ replaced by $\mu^v$, $\rho^v_{\mu^v}$, $\lambda^v$, $\Gamma_0^v$, $\Gamma^v$, respectively.
\end{lemma}
\ \\
When one assumes some more integrability of $f$, the above integrability result may be extended also to powers of $K$-transforms, since they can be expressed as sums of $K$-transforms of products:
\begin{proposition}\label{prop:tolltolltoll}
Let $\mu$ be a probability measure on $\Gamma$ which is locally absolutely continuous w.r.t.~Lebesgue-Poisson measure. Let $K\in\N$ and $f: \Gamma_m\to\R$ (or equivalently $f: (\R^d)^m\to\R$ symmetric) be measurable. Define for $M\leq mK$
\begin{multline*}
\mathcal Y^m_{M,K}:=\{(\alpha_1,\cdots,\alpha_{mK})\in \{1,\cdots,M\}^{2K}| \alpha_{lm+1},\cdots,\alpha_{lm+m}\mbox{ are pairwise distinct}\\ \mbox{for all $l=0,\cdots,K-1$}, \sharp\{\alpha_1,\cdots,\alpha_{mK}\}=M\}.
\end{multline*}
(The last condition in this definition ensures that each element of $\{1,\cdots,M\}$ appears at least once.)\\
Assume that for any $M\leq mK$ and $(\alpha_1,\cdots,\alpha_{mK})\in \mathcal Y^m_{M,K}$ it holds
\begin{equation}\label{eqn:wuest}
\left(\int_{(\R^d)^M} \rho_\mu(\{\xi_1,\cdots,\xi_M\}) \prod_{l=0}^{K-1} \vert f(\xi_{\alpha_{lm+1}},\cdots,\xi_{\alpha_{lm+m}})\vert d\xi_1\cdots d\xi_M\right)<\infty
\end{equation}
Then
$$
\int_{\Gamma} \vert Kf\vert^K d\mu<\infty
$$
and this integral may be estimated by a (finite) linear combination of the integrals in (\ref{eqn:wuest}) with coefficients only depending on $K$ and $M$, not on $\mu$.\\
A similar statement holds in the velocity marked case with independent Gaussian velocities. In this case, in (\ref{eqn:wuest}) one also integrates over the velocities, but w.r.t.~the corresponding Gaussian measure instead of Lebesgue measure.
\end{proposition}
\begin{proof}
W.l.o.g.~we assume $f$ to be nonnegative. For $\hat\gamma\in\Gamma$ let $(z_i)_{i\in\N}$ be a an enumeration of the elements of $\hat\gamma$. It holds
\begin{eqnarray*}
\lefteqn{(m!)^K\left\vert\sum_{\hat\eta\subset\hat\gamma,\sharp \hat\eta=m}f(\hat\eta)\right\vert^K}\\
& & =\sum_{\stackrel{(x_1,\cdots,x_{mK})\in \gamma^{mK}}{\sharp\{x_{lm+1},\cdots,x_{lm+m}\}=m\forall l}}\prod_{l=0}^{K-1} f(x_{lm+1},\cdots,x_{lm+m})\\
& & =\sum_{M=m}^{mK} \sum_{(y_1,\cdots,y_M)\in\hat\gamma^{M}_*} \sum_{(\alpha_1,\cdots,\alpha_{mK})\in \mathcal Y_{K,M}} \prod_{l=0}^{K-1} f(y_{\alpha_{lm+1}},\cdots,y_{\alpha_{lm+m}})\\
& & =\sum_{M=m}^{mK} \frac{1}{M!}\sum_{(y_1,\cdots,y_M)\in\hat\gamma^{M}_*}\sum_{\sigma} \sum_{(\alpha_1,\cdots,\alpha_{mK})\in \mathcal Y_{K,M}} \prod_{l=0}^{K-1} f(y_{\sigma\alpha_{lm+1}},\cdots,y_{\sigma\alpha_{lm+m}}).
\end{eqnarray*}
where $\sum_{\sigma}$ extends over all permutations $\sigma$ of $\{1,\cdots,M\}$ and $\hat\gamma^M_*$ denotes the set of $M$-tupels $(z_{i_1},\cdots,z_{i_M})$ such that $i_1<\cdots<i_M$. The last equality is due to the fact that the last sum is a symmetric expression in $y_1,\cdots,y_M$. We obtain
\begin{eqnarray}\label{eqn:dfuhid}
\lefteqn{(m!)^K\left\vert\sum_{\hat\eta\subset\hat\gamma,\sharp \hat\eta=m}f(\hat\eta)\right\vert^K}\\
& & =\sum_{M=m}^{mK} \frac{1}{M!} \sum_{(\alpha_1,\cdots,\alpha_{mK})\in \mathcal Y_{K,M}} \sum_{\{y_1,\cdots,y_M\}\subset \hat\gamma} F_{(\alpha_1,\cdots,\alpha_{mK},M)}(y_1,\cdots,y_M)\nonumber
\end{eqnarray}
where $F_{(\alpha_1,\cdots,\alpha_{mK},M)}(y_1,\cdots,y_M):=\sum_{\sigma} \prod_{l=0}^{K-1} f(y_{\sigma\alpha_{lm+1}},\cdots,y_{\sigma\alpha_{lm+m}})$ defines a symmetric function. So the last sum in (\ref{eqn:dfuhid}) is the $K$-transform of a symmetric function. Applying Lemma \ref{lem:vonkuna} we obtain the assertion.\\
For the velocity marked case the proof is completely analogous.
\end{proof}
\begin{remark}\label{rem:moments}
\begin{enumerate}
\item In fact, we only use the above proposition for $m=1$ and for $m=2$. If $m=1$ and $\rho_\mu$ fulfills (\ref{eqn:RBound}), the situation becomes considerably easy, since (\ref{eqn:wuest}) is implied by
$$
f\in L^1(\R^d;dx)\cap L^K(\R^d;dx)
$$
and $\mu(\vert Kf\vert^K)$ can be estimated in terms of $\Vert f\Vert_{L^1(\R^d;dx)}$, $\Vert f\Vert_{L^K(\R^d;dx)}$ and $\xi$.
In the velocity marked case the situation for $m=1$ is analogous (with $\R^d$ replaced by $\R^d\times\R^d$ and $dx$ replaced by $\frac{1}{\sqrt{2\pi/\beta}^d}e^{-\beta v^2/2}d(x,v)$).
\item Note that if $(\mu_n)_{n\in\N}$ are as in Proposition \ref{prop:tolltolltoll} and fulfill a Ruelle bound uniformly in $n$, one finds that the resulting estimate for $\Vert Kf\Vert_{L^K(\Gamma;\mu_n)}$ is uniform w.r.t.~$n$.
\end{enumerate}
\end{remark}

Before going on we need some information on the topological and measurable structure of $\Gamma_0$. $\Gamma_0=\bigcup_{m=0}^\infty \Gamma_m$ is equipped with the topology of disjoint union. Therefore, $\mathcal B(\Gamma_0)$ is generated by open sets $U\subset\Gamma_m$, $m\in\N$, which are bounded in the sense that $U\subset\Gamma_{\Lambda,m}$ for some open bounded $\Lambda\subset\R^d$. The topology on $\Gamma_m$, $m\in\N$, is defined as the quotient topology w.r.t.~the mapping $\sym_m: (\R^d)^m\setminus D\to\Gamma_m$, where $D_m=\{(x_1,\cdots,x_m)\in(\R^d)^m| x_i=x_j\mbox{ for some $i\neq j$}\}$.\\
We find that the set $\mathcal U_m$ of open bounded sets in $\Gamma_m$ is closed w.r.t.~finite intersections and that there exists a sequence $(U_k)_{k\in\N}\subset\mathcal U_m$ increasing to $\Gamma_m$. Therefore, the collection $\mathcal U:=\bigcup_m \mathcal U_m$ can be used to prove equality of measures on $\Gamma_0$. We use this fact in the proof Lemma of \ref{lem:originalRB} below. Moreover, any element of $\mathcal U$ is the limit of a monotonically increasing sequence of bounded continuous functions on $\Gamma_0$:
\begin{lemma}\label{lem:nocheinlemma}
Let $U\subset\Gamma_0$ be open and bounded, i.e.~there exists $M\in\N_0$ and an open bounded subset $\Lambda\subset\R^d$ such that $U\subset\bigcup_{m=0}^M \Gamma_{\Lambda,m}$. Then there exists a sequence $(f_k)_{k\in\N}$ of bounded continuous functions on $\Gamma_0$ increasing to $1_U$.
\end{lemma}
\begin{proof}
Assume w.l.o.g.~that $U\subset \Gamma_{\Lambda,m}$, $m\in\N$. Since $\sym_m^{-1}(U)\subset\R^d$ is open we may choose bounded continuous functions $\tilde{f}_k: \R^d\to\R$, $k\in\N$, with bounded support increasing to $1_{\sym_m^{-1}(U)}$. By mixing over the permutations of the arguments we may assume that the $\tilde{f}_k$ are symmetric. Now define $f_k(\{x_1,\cdots,x_m\}):=\tilde f_k(x_1,\cdots,x_m)$ for $\{x_1,\cdots,x_m\}\in \Gamma_m$, $k\in\N$. The desired properties of the sequence $(f_k)_{k\in\N}$ follow immediately.
\end{proof}

We now prove the result mentioned at the beginning for the case of the original (in contrast to ``improved'') Ruelle bound.
\begin{lemma}\label{lem:originalRB}
Let $(\mu_n)_n$ be a sequence of probability measures on $\Gamma$ such that each $\mu_n$ is locally absolutely continuous w.r.t.~Lebesgue-Poisson measure and such that moreover the correlation functionals fulfill a Ruelle bound (\ref{eqn:RBound}) uniformly in $n$. Let $\mu_n\to \mu$ weakly as $n\to\infty$.\\
Then the following holds:
\begin{enumerate}
\item For any $f\in L^1(\Gamma_0;\xi^{\sharp\cdot} \lambda)$ it holds $\mu_n(Kf)\to\mu(Kf)$ as $n\to\infty$. Moreover, there exists a sequence $(G_k)_{k\in\N}$ of bounded continuous cylinder functions $G_k: \Gamma\to\R$ such that $G_k\to Kf$ as $k\to\infty$ in $L^1(\Gamma;\mu)$ and $L^1(\Gamma;\mu_n)$ uniformly in $n\in\N$.
\item $\mu$ is locally absolutely continuous w.r.t.~Lebesgue-Poisson measure and its correlation functional fulfills the same Ruelle bound as the $\mu_n$.
\item The sequence $\left(\frac{\rho_{\mu_n}}{\xi^{\vert\cdot\vert}}\right)_{n\in\N}$ converges in weak-* sense to $\frac{\rho_\mu}{\xi^{\vert\cdot\vert}}$ in $L^\infty(\Gamma_0;\lambda)$ (seen as dual of $L^1(\Gamma_0;\lambda)$). 
\end{enumerate}
Similar statements hold for $\mu^v$, $\mu^v_n$, $n\in\N$.
\end{lemma}
\begin{proof}
Let $f:\Gamma_0\to\R$ be any nonnegative bounded continuous function having local support, i.e.~there exists $\Lambda\subset\R^d$ bounded such that $f(\hat\gamma)=0$ for all $\hat\gamma\in\Gamma_0\setminus\Gamma_{\Lambda}$. Then the mappings $Kf\wedge r: \Gamma\to\R$, $r>0$, are bounded and continuous (cf.~\cite[Proposition 4.1.5(v)]{Ku99}). Consequently, $\mu_n(Kf\wedge r)\to \mu(Kf\wedge r)$ as $n\to\infty$. By Proposition \ref{prop:tolltolltoll} and Remark \ref{rem:moments}(i) we find that $\mu_n(Kf-Kf\wedge r)\leq \frac{\mu_n((Kf)^2)}{r}\to 0$ as $r\to\infty$ uniformly in $n$. Moreover, for each $r>0$ it holds
$$
\mu(Kf\wedge r)=\lim_{n\to\infty}\mu_n(Kf\wedge r)\leq \sup_{n\in\N} \mu_n(Kf)<\infty
$$
by Lemma \ref{lem:vonkuna}. So, the monotone convergence theorem implies $Kf\in L^1(\Gamma;\mu)$ and $Kf\wedge r\to Kf$ in $L^1(\Gamma;\mu)$ as $r\to \infty$. Therefore (i) holds for $f$ as described above. We continue the proof of (i) after showing (ii) and (iii).\shortspacing
Relative compactness of $\left(\frac{\rho_{\mu_n}}{\xi^{\sharp\cdot}}\right)_{n\in\N}$ w.r.t.~weak-$*$ topology follows already from boundedness and the Banach-Alaoglu theorem. Let $\tilde\rho$ be an accumulation point and set $\rho:=\tilde\rho \xi^{\sharp\cdot}$. (This convenient method for obtaining a limiting correlation functional is taken from \cite[Theorem 5.5]{Ru70}, \cite[Theorem 2.7.12]{Ku99}, where it was used to prove the existence of a grand canonical Gibbs measure.) We now prove that $\rho(\cdot)\lambda$ coincides with the correlation measure of $\mu$ (cf.\cite[Section 4.2]{Ku99}). Once this is shown, we find by \cite[Proposition 4.2.2, Proposition 4.2.16]{Ku99} (the conditions given there are fulfilled by $\rho(\cdot)\lambda$) that $\mu$ is locally absolutely continuous w.r.t.~Lebesgue-Poisson measure and by \cite[Proposition 4.2.14]{Ku99} we see that $\rho$ is indeed the correlation functional for $\mu$. This implies (ii), and since it implies that there is at most one accumulation point $\tilde\rho$, (iii) also follows.\\
Let $U\in \mathcal U_m$, $m\in\N_0$. Choose a sequence $(f_k)_{k\in\N}$ increasing to $1_U$ as in Lemma \ref{lem:nocheinlemma}. Then by integrability of $1_U$ w.r.t.~$\xi^{\sharp\cdot}\lambda$ and Lemma \ref{lem:vonkuna} we find that $Kf_k\to K1_U$ in $L^1(\Gamma;\mu_n)$ as $k\to\infty$ uniformly in $n$. As above, using the monotone convergence theorem we obtain that $Kf_k\to K1_U$ also in $L^1(\Gamma;\mu)$. Therefore, 
\begin{equation}\label{eqn:Rom}
\mu_n(K1_U)\to \mu(K1_U)
\end{equation}
as $n\to\infty$.\\
We choose a subsequence $(\rho_{\mu_{n_k}})_{k\in\N}$ such that $\lim_{k\to\infty}\frac{\rho_{\mu_{n_k}}}{\xi^{\sharp\cdot}}=\tilde\rho$ in $L^\infty(\Gamma^0;\lambda)$ w.r.t.~weak-$*$ topology. Then
\begin{equation}\label{eqn:Rom2}
\int_{\Gamma_0} \rho_{\mu_{n_k}}1_U d\lambda=\int_{\Gamma_0} \frac{\rho_{\mu_{n_k}}}{\xi^{\sharp\cdot}}\xi^{\sharp\cdot} 1_U d\lambda\to \int_{\Gamma_0} \rho 1_U d\lambda
\end{equation}
as $n\to\infty$. By Lemma \ref{lem:vonkuna} the left-hand sides of (\ref{eqn:Rom}) and (\ref{eqn:Rom2}) coincide, hence we conclude equality of the right hand sides for all $U\in\mathcal U$. This implies by \cite[Theorem 10.3]{Bi79}, \cite[Definition 4.2.1]{Ku99} and the considerations preceding Lemma \ref{lem:nocheinlemma} that the correlation measure of $\mu$ is indeed given by $\rho(\cdot)\lambda$.\shortspacing
We complete the proof of (i). Let $f\in L^1(\Gamma_0;\xi^{\sharp\cdot}\lambda)$. We may w.l.o.g.~assume that $f$ is nonnegative. Choose a sequence $(f_k)_{k\in\N}$ of bounded continuous functions having local support converging to $f$ in $L^1(\Gamma_0;\xi^{\sharp\cdot}\lambda)$. Now, since the same Ruelle bound holds uniformly for $\mu_n$, $n\in\N$, and also for $\mu$, (i) follows from Lemma \ref{lem:vonkuna}: It holds $Kf_k\to Kf$ in $L^1(\Gamma;\mu_n)$ uniformly in $n$ and in $L^1(\Gamma;\mu)$.\shortspacing
In the velocity marked case (ii) now follows using the definition of $\mu^v$ and Lemma \ref{lem:vmeasure}(iv). (iii) is also seen using this lemma: For $f\in L^1(\Gamma_0^v,\lambda^v)$ it holds with $f_*$ defined as in Lemma \ref{lem:vmeasure}(iii)
$$
\int_{\Gamma^v} \frac{\rho_{\mu_n}^v}{\xi^{\sharp \cdot}} f d\lambda^v=\int_{\Gamma_0^v} \frac{\rho_{\mu_n}}{\xi^{\sharp \cdot}} f_* d\lambda \to \int_{\Gamma_0^v} \frac{\rho_{\mu}}{\xi^{\sharp \cdot}} f_* d\lambda)=\int_{\Gamma_0^v} \frac{\rho_{\mu}^v}{\xi^{\sharp \cdot}} f d\lambda^v
$$
as $n\to\infty$. (i) is shown for the velocity marked case analogously as for the unmarked case.
\end{proof}

We now extend the results from Lemma \ref{lem:originalRB} for the case of the improved Ruelle bound.
\begin{lemma}\label{lem:improvedRB}
Let $(\mu_n)_n$ be a sequence of probability measures on $\Gamma$, which are locally absolutely continuous w.r.t.~Lebesgue-Poisson measure and converge weakly to $\mu$. Let $\zeta\geq 1$ and $(\tilde h_n)_{n\in\N}\subset L^\infty(\Gamma_0;\lambda)$ be (uniformly bounded and) weak-$*$ convergent to some $\tilde h\in L^\infty(\Gamma_0;\lambda)$. Assume that
$$
\rho_{\mu_n}(\hat\eta)\leq \zeta^{\sharp\hat\eta}\tilde h_n(\hat\eta)
$$
is valid for all $\hat\eta\in\Gamma_0$ and $n\in\N$. Then the following holds:
\begin{enumerate}
\item $\rho_\mu$ fulfills the analogous bound with $\tilde h_n$ replaced by $\tilde h$.
\item Assume in addition that there exists a function $\overline{h}$ such that $\tilde{h}_n, \tilde{h}\leq \overline{h}$. For any measurable function $f:\Gamma_0\to\R$ which is integrable w.r.t.~$\zeta^{\sharp\cdot}\overline{h}(\cdot)\lambda$ it holds $\mu_n(Kf)\to\mu(Kf)$. Moreover, there exists a sequence of bounded continuous cylinder functions $(G_k)_{k\geq 0}$ such that $G_k\to Kf$ as $n\to\infty$ uniformly in $L^1(\Gamma;\mu_n)$, $n\in\N$, and in $L^1(\Gamma;\mu)$.
\end{enumerate}
Similar statements hold for $\mu^v$, $\mu^v_n$, $n\in\N$.
\end{lemma}
\begin{proof}
Since the $h_n$, $n\in\N$, are uniformly bounded and $\rho_\mu(\emptyset)=1$, the $\rho_\mu$ fulfill a uniform Ruelle bound $\rho_\mu\leq \tilde\zeta^{\sharp\cdot}$ with $\tilde\zeta\geq \zeta$ w.l.o.g. So $\left(\frac{\rho_{\mu_n}}{\tilde\zeta^{\sharp\cdot}}\right)_n$ converges in weak-$*$ sense to $\frac{\rho_\mu}{\tilde\zeta^{\sharp\cdot}}$ by Lemma \ref{lem:originalRB}(iii). Thus
$$
\int_A \left(\frac{\rho_{\mu}}{\tilde\zeta^{\sharp\cdot}}\right)d\lambda=\lim_{n\to\infty} \int_A \left(\frac{\rho_{\mu_n}}{\tilde\zeta^{\sharp\cdot}}\right)d\lambda\leq \lim_{n\to\infty} \int_A \tilde{h}_n\frac{\zeta^{\sharp\cdot}}{\tilde\zeta^{\sharp\cdot}}\,d\lambda=\int_A \tilde{h}\frac{\zeta^{\sharp\cdot}}{\tilde\zeta^{\sharp\cdot}}\,d\lambda
$$
holds for any set $A\subset \Gamma_{\Lambda,m}$ for some open relative compact $\Lambda\subset\R^d$ and some $m\in\N$. This already implies (i).\shortspacing
We now prove (ii). Let $(f_k)_{k\in\N}\subset L^1(\Gamma_0;\zeta^{\sharp\cdot}\lambda)$ be such that $f_k\to f$ in $L^1(\Gamma_0;\overline{h}(\cdot)\xi^{\sharp\cdot}\lambda)$. Due to Lemma \ref{lem:vonkuna} it holds
$$
\Vert Kf-Kf_k\Vert_{L^1(\Gamma;\mu_n)}\leq \Vert f-f_k\Vert_{L^1(\Gamma_0;\rho_{\mu_n}(\cdot)\lambda)}\leq \Vert f-f_k\Vert_{L^1(\Gamma_0;\overline{h}(\cdot)\lambda)},
$$
which converges to $0$ as $k\to\infty$ uniformly in $n$. Analogously we see that $Kf_k\to Kf$ in $L^1(\Gamma;\mu)$. Now (ii) follows from Lemma \ref{lem:originalRB}(i).\\
In the velocity marked case (i) is directly seen by Lemma \ref{lem:vmeasure}(iv) and (ii) is derived analogously to the unmarked case.
\end{proof}

We now focus on a special class of measures, the canonical Gibbs measures. For any open bounded set $\Lambda\subset\R^d$, $N\in\N$, $\beta>0$ and a symmetric potential $\phi$ (which we assume to be bounded below and finite a.e.) one defines the canonical Gibbs measure $\mu_{\Lambda,N}^{\phi,\beta}$ by
\begin{equation}\label{eqn:canGM}
\mu_{\Lambda,N}^{\phi,\beta}(A):=\frac{1}{Z_{\Lambda,N}^{\phi,\beta}}\int_{\Lambda^N} 1_A(x_1,\cdots,x_N)e^{-\beta U_\phi(x_1,\cdots,x_N)}\,dx_1\cdots dx_N,
\end{equation}
$A\subset \Lambda^{N}$ measurable, where $Z_{\Lambda,N}^{\phi,\beta}$ is the normalization constant. Define a mapping $\sym_{\Lambda,N}: \Lambda^N\to \Gamma$ by $\sym_{\Lambda,N}(x_1,\cdots,x_N):=\{x_1,\cdots,x_N\}$, $x_1,\cdots,x_N\in \Lambda$. Then the image measure $\mu_{\Lambda,N}^{\phi,\beta}\circ\sym_{\Lambda,N}^{-1}$ defines the corresponding distribution of $N$-point configurations. (Note that $\mu_{\Lambda,N}^{\phi,\beta}$-a.s.~$\sym_{\Lambda,N}$ has values in $\Gamma_{\Lambda,N}$, i.e.~one a.s.~obtains $N$-point configurations.)\shortspacing
We formulate the tightness result from \cite[Lemma 5.2]{GKR04} more generally, such that it also admits the perodic boundary case, in which, as the particle number $N$ and the volume $\Lambda$ of the system, also the potential $\phi$ varies.
\begin{lemma}\label{lem:tightnesslemma}
Let $(\phi_n)_{n\in\N}$ be a sequence of symmetric pair interactions fulfilling (RP), (BB) uniformly in $n$. Moreover let $(N_n)_{n\in\N}\subset\N$ and $(\Lambda_n)_{n\in\N}$ be a sequence of open relatively compact subsets of $\R^d$. Assume that $\sup_n\frac{N_n}{\vol(\Lambda_n)}<\infty$. Set $\mu_n:=\mu_{N_n,\Lambda_n}^{\phi_n,\beta}\circ\sym_{\Lambda_n,N_n}^{-1}$. If the correlation functionals $\rho_{\mu_n}$ of $\mu_n$, $n\in\N$, fulfill the improved Ruelle bound
$$
\rho_{\mu_n}(\eta)\leq \zeta^{\sharp\eta}e^{-\frac{2\beta}{\sharp \eta} \sum_{\{x,y\}\subset\eta}\phi_n(x-y)}\quad \mbox{for all $\eta\in\Gamma_0$}
$$
uniformly in $n$, then the sequence $(\mu_n)_{n\in\N}$ is tight. As a consequence, the same holds for the sequence $(\mu^v_n)_{n\in\N}$.
\end{lemma}
\begin{proof}
The proof is essentially the same as the proof of \cite[Lemma 5.2]{GKR04}. We use the compact functions $S^{\beta\Phi/3,h}$, where $\Phi$ is chosen as in (RP) and $h$ is as in Section \ref{sec:metric}. In order to prove that $\sup_{n\in\N}\Vert S^{\beta\Phi/3,h}\Vert_{L^2(\Gamma^v;\mu_n)}<\infty$ using Proposition \ref{prop:tolltolltoll}, one has to estimate integrals of the form
$$
\int_{(\R^d)^M} \vert f\vert(\xi_A,\xi_B)\vert f\vert (\xi_C,\xi_D)e^{-\frac{2}{M}\sum_{1\leq i<j\leq M} \beta\phi_n(\xi_i-\xi_j)}d\xi_1\cdots d\xi_M
$$
where $f(y,y')=e^{(\beta/3)\Phi(\vert y-y'\vert)}h(y)h(y')$, $y,y'\in\R^d$, $\{A,B,C,D\}=\{1,\cdots,M\}$, $M\in \{2,3,4\}$, $A\neq B$ and $C\neq D$. Since $e^{(\beta/3)(\Phi(\vert \xi_A-\xi_B\vert)+\Phi(\vert \xi_C-\xi_D\vert))}e^{-\frac{2}{M}\sum_{1\leq i<j\leq M} \beta \phi_n(\xi_i-\xi_j)}$ is bounded for any such $M,A,B,C,D$, it remains to show that
$$
\int_{(\R^d)^M} \vert h(\xi_A)h(\xi_B)h(\xi_C)h(\xi_D)\vert d\xi_1\cdots d\xi_M<\infty.
$$
But this follows since $h\in L^1(\R^d)\cap L^\infty(\R^d)$. The last assertion then follows by Lemma \ref{lem:vmeasure}(ii).
\end{proof}

\begin{remark}\label{rem:periodictightness}
\begin{enumerate}
\item Let $\phi$ fulfill (RP), (BB), (T) as in Section \ref{sub:conditions} and let $\Lambda_n:=\Lambda_{\lambda_n}=(-\lambda_n,\lambda_n]^d$, $n\in\N$. Then by Lemma \ref{lem:uniform1} and Theorem \ref{thm:RB} the conditions of Lemma \ref{lem:tightnesslemma} are fulfilled for any sequence $(N_n)_{n\in\N}$, $(\lambda_n)_{n\in\N}$ fulfilling $\lim_{n\to\infty}\lambda_n=\infty$ and $\lim_{n\to\infty}\frac{N_n}{(2\lambda_n)}=\rho\in [0,\infty)$ with $\phi_n:=\hat\phi_{\lambda_n}$, defined as in (\ref{eqn:hatphilambda}). Hence the corresponding sequence $(\mu_n)_{n\in\N}$ is tight.
\item In the periodic case (cf.~(i)) one might rather consider $\mu_{\Lambda_n,N_n}^{\phi_n,\beta}\circ\per_{\Lambda_n,N_n}^{-1}$ instead of $\mu_{\Lambda_n,N_n}^{\phi_n,\beta}\circ\sym_{\Lambda_n,N_n}^{-1}$, where $\per_{\Lambda_n,N_n}(x_1,\cdots,x_{N_n}):=\bigcup_{r\in\Z^d}\{x_1+2\lambda_n r,\cdots,x_{N_n}+2\lambda_n r\}$. \\
But since for any cylinder function $F:\Gamma\to\R$ it holds $\mu_{\Lambda_n,N_n}^{\phi_n,\beta}\circ\per_{\Lambda_n,N_n}^{-1}(F)=\mu_{\Lambda_n,N_n}^{\phi_n,\beta}\circ\sym_{\Lambda_n,N_n}^{-1}(F)$, Lemma \ref{lem:conv} implies that weak convergence properties of these sequences are equivalent.
\end{enumerate}
\end{remark}

Finally, in order to apply the result from Lemma \ref{lem:improvedRB} to the case of periodic boundary condition, we make the following remark.

\begin{remark}\label{rem:overlinephi}
Consider again the situation from Remark \ref{rem:periodictightness}(i). For $\psi: \R^d\to\R$ we define $b_\psi: \Gamma_0\to \R$ by $b_\psi(\eta):=e^{-\frac{2\beta}{n}\sum_{\{x,y\}\subset\eta}\psi(x-y)}$ (or $b_\psi(\eta):=\inf_{x\in\eta} e^{-\beta\sum_{y\in\eta\setminus\{x\}}\psi(x-y)}$), $\eta\in\Gamma_0$. Setting $\tilde h_n:=1_{\Lambda_{\lambda_n}} b_{\hat\phi_{\lambda_n}}$ and $\tilde h:=b_\phi$ we find by uniform stability of the periodic interaction energy of configurations in $\Lambda_{\lambda_n}$ (cf.~Lemma \ref{lem:uniformprops}(ii)) that the $\tilde h_n$ are uniformly bounded. From (\ref{eqn:distphiphilambda}) we find that $\phi_{\lambda_n}\to\phi$ pointwise and hence also $\hat\phi_{\lambda_n}\to\phi$ pointwise, which implies that $\tilde{h}_n\to\tilde{h}$ pointwise. Together with uniform boundedness we obtain weak-$*$ convergence in $L^\infty(\Gamma_0;\lambda)$. Hence Lemma \ref{lem:improvedRB}(i) can be applied and $\mu$ fulfills the improved Ruelle bound for $\phi$.\\
We now choose a function $\overline{h}$ fulfilling the assertion of Lemma \ref{lem:improvedRB}(ii) and which is useful for the considerations in Section \ref{sub:MP} below. By (\ref{eqn:distphiphilambda}) there exists $m>0$ such that
$$
\vert \hat\phi_{\lambda_n}(y)-\phi(y)\vert=\vert \phi_{\lambda_n}(y)-\phi(y)\vert\leq m\quad\mbox{for $n\in\N$, $\vert y\vert< \lambda_n$}
$$
and
$$
\inf_{z\in\R^d}\hat\phi_{\lambda_n}(z)=\inf_{z\in\R^d}\phi_{\lambda_n}(z)\leq -m-M
$$
for all $n\in\N$, where $-M$ denotes a lower bound of $\phi$.
Hence, setting
$$
\overline{\phi}(y):=\left\{\begin{array}{ll} \phi(y)-m&\mbox{if $\vert y\vert<\lambda_1$}\\ -m-M &\mbox{else}\end{array}\right.,
$$
we obtain $\hat\phi_{\lambda_n}\geq \overline{\phi}$ for all $n\in\N$ and $\phi\geq \overline{\phi}$, which implies $\tilde{h}_n,\tilde{h}\leq b_{\overline{\phi}}=:\overline{h}$, $n\in\N$. Thus the conclusion of Lemma \ref{lem:improvedRB}(ii) is valid in this case. Moreover, $\overline{\phi}+m+M+M' \geq \phi$, where $M':=\sup_{\vert x\vert\geq \lambda_1}\phi(x)<\infty$, so $\Vert\cdot\Vert_{L^p(\R^d;e^{-\beta\overline{\phi}}dx)}$ and $\Vert\cdot\Vert_{L^p(\R^d;e^{-\beta\phi}dx)}$ are equivalent norms for $p\geq 1$.
\end{remark}

\end{subsection}
\end{section}
\setcounter{equation}{0}

\begin{section}{$N/V$-limit of Langevin dynamics}\label{sec:construction}

We now derive the main result of this article. Starting with $N$-particle Langevin dynamics on cuboid domains (Section \ref{sub:fddyn}), we go on by proving tightness of the corresponding laws on $\Gamma^v$ (Section \ref{sub:tightness}) and finally prove (Section \ref{sub:MP}) that any weak accumulation point of these laws solves (\ref{eqn:langevin}) weakly in the sense specified in Section \ref{sec:introduction}.\shortspacing
Throughout this section we fix an inverse temperature $\beta>0$ and we assume that any function $g: A\to\R$ defined on some subset $A\subset\Gamma_0$ (resp.~$\Gamma_0^v$) is extended to the whole of $\Gamma_0$ (resp.~$\Gamma_0^v$) by being set to $0$ on the complement of $A$ (cf.~Section \ref{sub:weaklimits} for the definition of $\Gamma_0$, $\Gamma_0^v$).

\begin{subsection}{Additional conditions on the potential}\label{sub:addconditions}

Let $\phi$ be a (symmetric) pair potential fulfilling the conditions (RP), (BB), (T) given in Section \ref{sub:conditions}. Consider the following additional conditions on $\phi$:
\begin{enumerate}
\item[(WD)] (\emph{weak differentiability}) $\phi$ is continuous in $\R^d\setminus\{0\}$, $\phi$ is weakly differentiable on this set and $\nabla\phi$ is bounded on each of the sets $\{x\in \R^d| \vert x\vert>r\}$, $r>0$. Moreover, $\nabla\phi\in L^1(\R^d;e^{-\beta\phi}dx)\cap L^3(\R^d;e^{-\beta\phi}dx)$.
\item[(IDF)] (\emph{integrably decreasing forces}) $\phi$ is weakly differentiable in $\R^d\setminus\{0\}$ and there exist $R_3>0$ and a decreasing function $\bratwurst: [R_3,\infty)\to [0,\infty)$ such that
$$
\vert \nabla\phi(x)\vert\leq \bratwurst(\vert x\vert)\quad\mbox{ for all $x\in\R^d$, $\vert x\vert \geq R_3$}
$$
and $\int_{[R_3,\infty)} r^{d-1}\bratwurst(r)\,dr<\infty$.
\end{enumerate}
\begin{remark}\label{rem:bedingungIDF}
\begin{enumerate}
\item Both assumptions we suppose to be quite natural and sufficiently weak, allowing e.g.~the Lennard-Jones potential or any other potential fulfilling (WD) and being such that $\vert \nabla\phi(x)\vert$ decreases when $\vert x\vert\to\infty$, $x\in\R^d$. 
\item In order to do the construction using a limit of dynamics corresponding to $\phi$ with periodic boundary, we need uniform $L^3$-integrability of the $\nabla\hat\phi_{\lambda}$ w.r.t.~$e^{-\beta \hat\phi_\lambda}$ (cf.~(\ref{eqn:hatphilambda})), at least for a sequence $\lambda_n$ tending to $\infty$ as $n\to\infty$. When (WD) holds, this is an assumption on the behavior of $\nabla\phi$ \emph{away} from the origin. Condition (IDF) yields an appropriate behavior, as we prove in the following lemma. Though it might be not optimal, it is sufficient for our purposes.
\item We suppose that one does not need (IDF) to construct a martingale solution of (\ref{eqn:langevin}). The construction for a potential $\phi$ not fulfilling (IDF) might be done by constructing first the dynamics for smooth cut-offs of $\phi$ by approximation with periodic potentials and then approximating $\phi$ by the cut-offs. However, we do not enter into a detailed consideration about this question here.
\end{enumerate}
\end{remark}

\begin{lemma}\label{lem:intofhatphilambda}
\begin{enumerate}
\item Let $\phi$ fulfill (RP), (BB), (T), (WD). Then for any $\lambda>0$ the function $\hat\phi_\lambda$ is weakly differentiable in $(-2\lambda,2\lambda)^d\setminus\{0\}$ and for any $\lambda_0>0$ it holds $\sup_{\lambda\geq\lambda_0} \Vert \nabla\hat\phi_\lambda\Vert_{L^1(\Lambda_{2\lambda};e^{-\beta\hat\phi_\lambda}dx)}<\infty$.
\item If $\phi$ additionally fulfills (IDF), then $\sup_{\lambda\geq \lambda_0}\Vert \hat\phi_\lambda\Vert_{L^3(\Lambda_{2\lambda};e^{-\beta\hat\phi_\lambda}dx)}<\infty$ holds for any $\lambda_0>0$.
\end{enumerate}
\end{lemma}
\begin{proof}
The functions $\hat{\phi}_\lambda$, $\lambda>0$, are cutoffs of $2\lambda$-periodic functions, hence all the assertions on integrals over $\Lambda_{2\lambda}$ can be reduced to assertions on integrals over $\Lambda_\lambda$. E.g., we have $\Vert \hat\phi_\lambda\Vert^p_{L^p(\Lambda_{2\lambda};e^{-\beta\hat\phi_\lambda}dx)}= 2^d\Vert \hat\phi_\lambda\Vert^p_{L^p(\Lambda_{\lambda};e^{-\beta\hat\phi_\lambda}dx)}$, $p\geq 1$.\shortspacing
(i): For the first assertion it suffices to prove that $\hat\phi_\lambda$ is the $L^1(\Lambda_{2\lambda};e^{-\beta\hat\phi_\lambda}dx)$-limit of a convergent sequence of weakly differentiable (in $(-2\lambda,2\lambda)^d\setminus\{0\}$) functions w.r.t.~the norm $\Vert\cdot\Vert_{W^{1,1}(\Lambda_{2\lambda},e^{-\beta\hat\phi_\lambda}dx)}:=\Vert \cdot\Vert_{L^1(\Lambda_{2\lambda};e^{-\beta\hat\phi_\lambda}dx)}+\Vert \nabla\cdot\Vert_{L^1(\Lambda_{2\lambda};e^{-\beta\hat\phi_\lambda}dx)}$. Define $\hat\phi_{\lambda,k}:=\sum_{r\in\Z^d, \vert r\vert\leq k} \phi(\cdot+2\lambda r)$, $k\in\N$. Then $\hat\phi_{\lambda,k}\to \hat\phi_\lambda$ as $k\to\infty$ in $L^1(\Lambda_{2\lambda};e^{-\beta\hat\phi_\lambda})$ and moreover for $k,l\in\N$, $k\geq l$, it holds
\begin{eqnarray*}
\lefteqn{\Vert \nabla\hat\phi_{\lambda,k}-\nabla\hat\phi_{\lambda,l}\Vert_{L^1(\Lambda_{2\lambda};e^{-\beta \hat\phi_\lambda})}}\\
&\leq& 2^d \int_{\lambda (2l+1)\leq \vert x\vert} \vert \nabla\phi(x)\vert e^{\beta (m+M)}\,dx\to 0
\end{eqnarray*}
as $l\to\infty$. Here $m$ is as in Remark \ref{rem:overlinephi}. This shows the first assertion.\\
Since $m$ can be chosen independent of $\lambda\geq \lambda_0$, we find by an easy argument similar to the above calculation that $\Vert \nabla\hat\phi_\lambda\Vert_{L^1(\Lambda_{2\lambda};e^{-\beta\hat\phi_\lambda}dx)}\leq 2^d e^{\beta (m+M+M')}\Vert \nabla\phi\Vert_{L^1(\R^d;e^{-\beta\phi}dx)}$, where $M':=\sup_{\vert x\vert\geq \lambda_0} \phi(x)$.\shortspacing
We now prove (ii). We may assume that $R_3=R_1=R_2=:R\leq \lambda_0$. By the considerations preceding the proof of (i) we have to estimate
$$
\int_{\Lambda_\lambda} \left\vert \sum_{r\in\Z^d} \nabla\phi(x+2\lambda r)\right\vert^3e^{-\beta\hat\phi_\lambda(x)}\,dx\leq \int_{\Lambda_\lambda} \left(\vert\nabla\phi(x)\vert+\sum_{0\neq r\in\Z^d}\bratwurst(\vert x+2\lambda r\vert ) \right)^3e^{-\beta\hat\phi_\lambda(x)}\,dx
$$
independently of $\lambda\geq \lambda_0$. Due to $L^3$-integrability of $\nabla\phi$ w.r.t.~$e^{-\beta\phi}dx$, hence w.r.t. $e^{-\beta\hat\phi_\lambda}dx$ uniformly in $\lambda\geq \lambda_0$, it suffices to show that $(2\lambda)^d\sup_{x\in\Lambda_\lambda}\left(\sum_{r\neq 0}\bratwurst(\vert x+2\lambda r\vert)\right)^3$ is bounded independently of $\lambda\geq\lambda_0$.\\
This follows from monotonicity and integrability of $\bratwurst$: For $\lambda\geq \lambda_0$, $x\in\Lambda_\lambda$ it holds:
\begin{equation*}
\sum_{\stackrel{r=(r_1,\cdots,r_d)\in\Z^d}{r_i\geq 1 \forall i}} \bratwurst(\vert x+2\lambda r\vert)\leq \sum_{\stackrel{r=(r_1,\cdots,r_d)\in\Z^d}{r_i\geq 0 \forall i}} \frac{\int_{\Lambda_\lambda} \bratwurst(\vert y+2\lambda r\vert)\,dy}{(2\lambda)^d}\leq \frac{C\int_{[0,\infty)}\theta(t)t^{d-1}dt}{(2\lambda)^d}
\end{equation*}
for some $C<\infty$ independent of $\lambda$. Here we extend $\theta$ to $[0,\infty)$ by setting $\theta(t):=\theta(R_3)$ for $t\leq R_3$. Other parts of the sum $\sum_{0\neq r\in\Z^d}$ are treated in an analogous way.
\end{proof}

\begin{remark}
In order to have some more concrete legitimation for the introduction of the additional condition (IDF), we consider the following example: Set $d=1$ and consider $\phi$ according to (RP), (BB), (WD) such that it holds $\nabla\phi(x)=1$ whenever $2k+\frac{1}{2}-\frac1{\vert k+1\vert^2}\leq \vert x\vert\leq 2k+\frac{1}{2}+\frac{1}{\vert k+1\vert^2}$ for some $k\in\Z$ and $\nabla\phi(x)\geq 0$ when $2k+\frac{1}{4}\leq \vert x\vert\leq 2k+\frac{3}{4}$. Then $\phi$ can be such that $\nabla\phi\in L^1\cap L^3$ and (T) is fulfilled, but $\sum_{r\in\Z} \nabla\phi(\cdot+2 r)$ behaves like $\frac{1}{\sqrt{\cdot-\frac{1}{2}}}$ around $\frac{1}{2}$, so one does not obtain $L^2$ integrability of $\nabla\hat\phi_1$.
\end{remark}

\begin{lemma}\label{lem:momentsofhatphilambda}
Let $\phi$ fulfill (RP), (T), (BB), (WD), and assume that for some $\lambda_0>0$ it additionally holds $\sup_{\lambda\geq \lambda_0}\Vert \nabla\hat\phi_{\lambda}\Vert_{L^p(\Lambda_\lambda;e^{-\hat\phi_\lambda}dx)}<\infty$ for $p=1,2,3$. Then it holds for $i=1,2,3$
$$
\sup_{\lambda\geq \lambda_0} \Vert \vert\nabla\hat\phi_{\lambda} \vert e^{-\frac{\beta}{3}\hat\phi_\lambda}\Vert_{L^i(\Lambda_{2\lambda};dx)}= 2^d\sup_{\lambda\geq \lambda_0} \Vert \vert\nabla\hat\phi_{\lambda} \vert e^{-\frac{\beta}{3}\hat\phi_\lambda}\Vert_{L^i(\Lambda_{\lambda};dx)}<\infty
$$
\end{lemma}
\begin{proof}
The equality is clear, cf.~the proof of Lemma \ref{lem:intofhatphilambda}.\shortspacing
Choose any $a>0$. The functions $\hat\phi_\lambda$ are bounded in the set $\{x\in\Lambda_\lambda| \vert x\vert \geq a\}$ and the bound is uniform in $\lambda\geq \lambda_0$ (cf.~Lemma \ref{lem:uniform1}). Hence there exists $D>0$ such that $e^{-\frac{i\beta}{3} \hat\phi_\lambda}\leq D e^{-\beta\hat\phi_\lambda}$ on this set for $\lambda\geq \lambda_0$. We compute for $\lambda\geq \lambda_0$
\begin{eqnarray*}
\lefteqn{\Vert \vert\nabla\hat\phi_{\lambda} \vert e^{-\frac{\beta}{3}\hat\phi_\lambda}\Vert_{L^i(\Lambda_{\lambda};dx)}}\\
& &\leq D^{1/i}\Vert \nabla\hat\phi_{\lambda} \Vert_{L^i(\Lambda_\lambda\cap \{\vert\cdot\vert>a\};e^{-\beta\hat\phi_\lambda}dx)}
+ \Vert \vert\nabla\hat\phi_{\lambda} \vert^i e^{-\frac{\beta i}{3}\hat\phi_\lambda}\Vert^{1/i}_{L^1(\Lambda_{\lambda}\cap\{\vert\cdot\vert\leq a\};dx)}\\
& &\leq D^{1/i}\Vert \nabla\hat\phi_{\lambda} \Vert_{L^i(\Lambda_\lambda;e^{-\beta\hat\phi_\lambda}dx)}
+(2a)^{\frac{(3-i)d}{3i}}\Vert \nabla\hat\phi_{\lambda} \Vert_{L^3(\Lambda_\lambda;e^{-\beta\hat\phi_\lambda}dx)}
\end{eqnarray*}
by the H\"older inequality. The assertion follows.
\end{proof}

\end{subsection}

\begin{subsection}{The finite particle dynamics on $\Gamma^v$}\label{sub:fddyn}

Let $\phi$ fulfill (RP), (T), (BB), (WD) and (IDF) and let $N\in\N$, $\lambda>0$. The state space for the $N$-particle dynamics is given by $E_\lambda^N$, where $E_\lambda:=M_\lambda\times\R^d$, $M_\lambda$ being the manifold resulting from glueing the opposite surfaces of $\Lambda_\lambda=(-\lambda,\lambda]^d$ together. We define the $N$-particle potential $\Psi_{\lambda,N}$ by $\Psi_{\lambda,N}(x_1,\cdots,x_N):=\sum_{i< j}\hat\phi_\lambda(x_i-x_j)$, $(x_1,\cdots,x_N)\in M_{\lambda}^N$. The potential $\Psi_{\lambda,N}$ fulfills the assumptions of \cite[Theorem 2.1]{CG07a} (cf.~\cite[Example 2.3]{CG07a}). Thus there is a law $P_{\lambda,N}$ on $C([0,\infty),E_\lambda^N)$ such that the corresponding process is a Markov process solving (cf.~\cite[Lemma 3.12(ii)]{CG07a}) the martingale problem for the $L^2(E_\lambda^N;\mu_{\lambda,N})$-closure $(\overline{L_{\lambda,N}},D(\overline{L_{\lambda,N}}))$ of the generator $(L_{\lambda,N},C_0^\infty(E_\lambda^N))$, given by
\begin{equation}\label{eqn:finvolgenerator}
L_{\lambda,N}=\frac{\brotkiste}{\beta}\Delta_v-\brotkiste v\nabla_v+v\nabla_x-(\nabla\Psi_{\lambda,N})\nabla_v.
\end{equation}
Here $\mu_{\lambda,N}$, the invariant initial distribution of the process, is given by
$$
\mu_{\lambda,N}(A)=\frac{1}{Z}\int_{A} e^{-\beta\Psi_{\lambda,N}(x_1,\cdots,x_N)}e^{-\frac{\beta}{2}(v_1^2+\cdots+v_N^2)}\,dx_1\,dv_1\cdots dx_N\,dv_N,
$$
where $A$ is a Borel subset of $E_\lambda^N$ and $Z$ is a normalization constant. So, $\mu_{\lambda,N}$ is the canonical Gibbs measure corresponding to $\Psi_{\lambda,N}$. We do not claim $(\overline{L_{\lambda,N}},D(\overline{L_{\lambda,N}}))$ to be essentially maximal dissipative nor do we need such a property in the sequel.\shortspacing
We do essentially not distinguish $E_\lambda^N$ and $M_\lambda^N\times \R^{dN}$: An element $(x_1,v_1,\cdots,x_N,v_N)$ of $E_\lambda^N$ we sometimes denote by $(x,v)$, $x=(x_1,\cdots,x_N)$, $v=(v_1,\cdots,v_N)$. $\Delta_v$, $\nabla_v$ denote the Laplacian and the gradient resp.~in $v$-direction, $v$ denotes multiplication by the vector $v$, $v\nabla_v:=\sum_{i=1}^N v_i\nabla_{v_i}$ etc.\\
For later use we prove a lemma concerning the domain of $\overline{L_{\lambda,N}}$. We do not claim that it is stated in maximal generality, in particular as far as it concerns the first assertion.

\begin{lemma}\label{lem:functionsinDL}
\begin{enumerate}
\item Let $f\in C(E_\lambda^N)$ be such that it possesses continuous partial derivatives up to order $2$ in all $v$-directions and continuous partial derivatives of order $1$ in all $x$-directions. Assume moreover that $f$ and all mentioned partial derivatives are bounded in absolute value by a multiple of $(x,v)\mapsto(1+\vert v\vert)^k$ for some $k\in\N$. Then $f\in D(\overline{L_{\lambda,N}})$ and $\overline{L_{\lambda,N}}f$ is given as in (\ref{eqn:finvolgenerator}).
\item Let $f: M_{\lambda}^N\setminus D_{\lambda,N}\to\R$, where $D_{\lambda,N}:=\{(x_1,\cdots,x_N)\in M_\lambda^N| \exists i,j\in\{1,\cdots,N\}: x_i=x_j, i\neq j\}$. Assume that $f$ is once continuously differentiable and that $f,\nabla_x f\in L^2(M_\lambda^N;e^{-\beta\Psi_{\lambda,N}}dx)$. Then when $f$ is considered as a function on $E_\lambda^N$ which is constant in $v$-directions, it holds $f\in D(\overline{L_{\lambda,N}})$ and $\overline{L_{\lambda,N}}f$ is given as in (\ref{eqn:finvolgenerator}), i.e.~$\overline{L_{\lambda,N}}f=v\nabla_x f$.
\end{enumerate}
\end{lemma}
\begin{proof}
(i): First assume that $f$ has compact support. Approximate $f$ uniformly by $C_0^\infty$-functions $f_k$, $k\in\N$, such that also the mentioned partial derivatives of $f$ are uniformly approximated by the respective partial derivatives of $f_k$. (Take, for example, convolutions with a suitable approximate identity.) Then one obtains $L^2$ convergence of $f_n$ to $f$ and also $L^2$-convergence of $L_{\lambda,N}f_k$ towards $\frac{\brotkiste}{\beta}\Delta_v f-\brotkiste v\nabla_v f+v\nabla_x f-(\nabla\Psi_{\lambda,N})\nabla_v f$. (Note that all partial derivatives of $\Psi_{\lambda,N}$ are square integrable w.r.t.~$\mu_{\lambda,N}$.) This proves the assertion for compactly supported $f$.\\
For the general case in (i) we do another approximation of $f$ by multiplying $f$ with smooth compactly supported functions $\eta_k: E_\lambda^N \to \R$, $k\in\N$, which fulfill $1_{\{\vert \cdot\vert\leq k\}}(v)\leq \eta_k(x,v)\leq 1_{\{\vert \cdot\vert\leq k+2\}}(v)$, $(x,v)\in E_\lambda^N$, and are such that their first and second partial derivatives are bounded in absolute value by $1$.\shortspacing
(ii): For functions $f$ having compact support in the open set $M_{\lambda}^N\setminus D_{\Lambda,N}=\{x\in M_{\lambda}^N| \Psi_{\lambda,N}(x)<\infty\}$ the assertion is implied by (i). When $f$ is bounded, one can do an approximation as follows: Let $\chi_k: \R\to[0,1]$, $k\in\N$, be smooth functions such that $\chi_k(t)=1$ for all $t\in [-k,k]$, $\chi_k(t)=0$ for all $t\in \R\setminus [-k-2,k+2]$ and the first derivative of $\chi_k$ is bounded in absolute value by $1$. Define $f_k: M_\lambda^N\to\R$ by $f_k(x):=f(x)(\chi_k\circ\Psi_{\lambda,N})(x)$, $x\in M_\lambda^N$. Then $f_k\to f$ as $k\to\infty$ in $L^2(E_\lambda^N;\mu_{\lambda,N})$ and 
$$
\overline{L_{\lambda,N}}f_k=v\nabla_x f_k=(\chi_k\circ \Psi_{\lambda,N}) v\nabla_x f+(v\nabla_x\Psi_{\lambda,N})(\chi_k'\circ \Psi_{\lambda,N})f\to v\nabla_x f
$$
as $k\to\infty$ in $L^2(E_\lambda^N;\mu_{\lambda,N})$. Here we again used the fact that $\nabla\Psi_{\lambda,N}\in L^2(E_\lambda^N;\mu_{\lambda,N})$.\\
Finally we consider the case where $f$ is unbounded. Choose another sequence of smooth functions $\kappa_k: \R\to \R$, $k\in\N$, such that $\kappa_k(t)=t$ for all $t\in [-k,k]$, $\kappa_k$ is increasing with derivative bounded by $1$ and $\kappa_k$ is constant on $\R\setminus [-k-1,k+1]$. Define $f_k:=\kappa_k\circ f$. Then $f_k\to f$ as $k\to\infty$ in $L^2(E_\lambda^N;\mu_{\lambda,N})$ and
$$
\overline{L_{\lambda,N}}f_k=(\kappa_k'\circ f)v\nabla_x f\to v\nabla_x f
$$
in $L^2(E_\lambda^N;\mu_{\lambda,N})$ as $k\to\infty$, so the assertion is shown.
\end{proof}

In order to simplify notation, in the sequel we do not distinguish $M_\lambda$ and $\Lambda_\lambda$ as well as $E_\lambda$ and $\Lambda_\lambda\times\R^d$. Moreover, since confusion would not be dangerous, we do not use different notations for $\Lambda_\lambda^N$ and the set $\Lambda_\lambda^N\setminus D_{\lambda,N}$. (The diagonal $D_{\lambda,N}$, defined in Lemma \ref{lem:functionsinDL}(ii), is not hit $P_{\lambda,N}$-a.s.~and hence may be omitted).\shortspacing
We consider the mapping $\per_{\lambda,N}: E_\lambda^N\to \Gamma^v$ defined by 
$$
\per_{\lambda,N}(x_1,\cdots,v_N):=\bigcup_{r\in\Z^d} \{(x_1+2\lambda r,v_1),\cdots,(x_N+2\lambda r,v_N)\}.
$$ 
Furthermore, we define the mapping $\per_{\lambda,N}^{\otimes [0,\infty)}: C([0,\infty),E_\lambda^N)\to C([0,\infty),\Gamma^v)$ by assigning to a path $((x_1(t),\cdots,v_N(t)))_{t\geq 0}$ the path $\left(\per_{\lambda,N} (x_1(t),\cdots,v_N(t))\right)_{t\geq 0}$ of images w.r.t.~$\per_{\lambda,N}$.\\
Both mappings are well-defined except on the diagonal (which is negligible w.r.t.~both $\mu_{\lambda,N}$ and $P_{\lambda,N}$) and measurable. We set $\mu^{(\lambda,N)}:=\mu_{\lambda,N}\circ\per_{\lambda,N}^{-1}$ and define $P^{(\lambda,N)}:=P_{\lambda,N}\circ \left(\per_{\lambda,N}^{\otimes [0,\infty)}\right)^{-1}$. These probability laws are the starting point for the construction of an infinite particle Langevin dynamics as a weak limit.\shortspacing
Sometimes we also use the mappings $\sym_{\lambda,N}: E_\lambda^N\to \Gamma^v$, defined by $\sym_{\lambda,N}(x_1,\cdots,v_N):=\{(x_1,v_1),\cdots,(x_N,v_N)\}$. This is done for technical reasons: The measures $\mu^{(\lambda,N)}$ are not locally absolutely continuous w.r.t.~Lebesgue-Poisson measure and in particular do not fulfill a Ruelle bound. Therefore, in order to apply the results from Section \ref{sub:weaklimits} we have to use $\mu_{\lambda,N}\circ\sym_{\lambda,N}^{-1}$ instead.
\begin{remark}
Note that one at best faces some technical difficulties trying a construction by dynamics given through $P_{\lambda,N}\circ\left(\sym_{\lambda,N}^{\otimes [0,\infty)}\right)^{-1}$, where one defines $\sym_{\lambda,N}^{\otimes [0,\infty)}$ analogously to $\per_{\lambda,N}^{\otimes [0,\infty)}$. The paths corresponding to these laws are not even right continuous. In contrast, the laws $P^{(\lambda,N)}$ describe diffusions.
\end{remark}

\end{subsection}

\begin{subsection}{Tightness}\label{sub:tightness}

In this section we prove, under the conditions and using the notations of Section \ref{sub:fddyn}, tightness of any sequence $(P^{(\lambda_n,N_n)})_{n\in\N}$ such that $\lambda_n\uparrow\infty$ and $\frac{N_n}{(2\lambda_n)^d}\to \rho\in [0,\infty)$ as $n\to\infty$. In the sequel we abbreviate subscripts $\lambda_n,N_n$ by $n$, i.e.~$P_n:=P_{\lambda_n,N_n}$, $\sym_n:=\sym_{\lambda_n,N_n}$ etc. Paths in $C([0,\infty),\Gamma^v)$ will below always be denoted by $({\gamma}_t)_{t\geq 0}$. Clearly, we may assume that $\lambda_1$ is large enough for Theorem \ref{thm:RB} (and Corollary \ref{cor:iRB}) to apply.\shortspacing
Tightness of the sequence of distributions $P^{(n)}\circ \gamma_t^{-1}(=\mu^{(n)})$, $t\geq 0$, is seen from Remark \ref{rem:periodictightness} and Lemma \ref{lem:vmeasure}(ii). So we go on by estimating moments of $d^{\Phi,a,h}(\gamma_t,\gamma_s)$, $t,s\geq 0$, with $d^{\Phi,a,h}$ as defined in Section \ref{sec:metric}. We follow an idea from \cite{HS78} and use semimartingale decompositions of the summands contained in $d^{\Phi,a,h}(\gamma_t,\gamma_s)$. Before we do so, we need some preparations.\shortspacing
The following lemma might also be stated more generally. However, we restrict to what we are about to use later.

\begin{lemma}\label{lem:tness1lemma}
Let $f: \R^d\times\R^d\to\R$ have bounded spatial support and being once continuously differentiable in the $x$-directions and twice continuously differentiable in the $v$-directions. Assume moreover that all these derivatives are bounded. $f$ itself may be unbounded. Set $F:=\langle f,\cdot\rangle=Kf$.\\
Then $F\circ \per_n$ is an element of the domain $D(\overline{L_n})$ and it holds
$$
\sup_n \mu_n(\vert \overline{L_n}(F\circ \per_n)\vert^3)<\infty
$$
\end{lemma}
\begin{proof}
We may w.l.o.g. assume that the spatial support of $f$ is contained in an open cube of side length less than $2\lambda_1$. (Otherwise we use an appropriate partition of unity corresponding to a suitable locally finite open cover of $\R^d$ to decompose $f$ (cf.~the proof of \ref{lem:Phitightness} below).) Moreover, we may w.l.o.g.~assume that the spatial support of $f$ is relatively compact in $(-\lambda_1,\lambda_1)^d$, since $\overline{L_n}$ commutes with simultaneous spatial translations of all particles in (the manifold) $E_{\lambda_n}$ and $\mu_n$ is invariant w.r.t.~these translations. So we may replace $\per_n$ by $\sym_n$.\shortspacing
The first assertion is seen from Lemma \ref{lem:functionsinDL}(i). (Note that $f(x,v)$, $(x,v)\in \R^d\times\R^d$ grows at most linearly as $\vert v\vert\to\infty$.)\shortspacing
It holds for $n\in\N$
$$
\overline{L_n}(F\circ\sym_n)=Kg_1\circ \sym_n-Kg_2^{n}\circ \sym_n,
$$
where $g_1: \Gamma_1\to\R$ is given by $g_1(\{(x,v)\}):=\frac{\brotkiste}{\beta}\Delta_v f(x,v)-\brotkiste v\nabla_v f(x,v)+v\nabla_x f(x,v)$ and $g_2^n: \Gamma_2\to\R$ is given by $g_2^n(\{(x,v),(x',v')\}):=\nabla\hat\phi_{\lambda_n}(x-x')(\nabla_v f(x,v)-\nabla_v f(x',v'))$, $(x,v),(x',v')\in E_{\lambda_n}$.\shortspacing
Let us first prove
\begin{equation}\label{eqn:tightness1}
\sup_n \mu_n(\vert Kg_1\circ \sym_n\vert^3)<\infty.
\end{equation}
Since $g_1$ has (seen as a function defined on $\R^d\times\R^d$) bounded spatial support and there exists $C>0$ such that $\vert g_1(x,v)\vert \leq C\vert v\vert$ for all $(x,v)\in\R^d\times\R^d$, it follows $g\in L^p(\Gamma_0;\lambda^v)$ for each $p\in [1,\infty)$. So the improved Ruelle bound (Corollary \ref{cor:iRB}), Proposition \ref{prop:tolltolltoll} and Remark \ref{rem:moments} imply (\ref{eqn:tightness1}).\\[1.5ex]
Concerning $g_2^n$ we have that for $n\in\N$
$$
( K\vert g_2^n\vert (\cdot))^3\leq \left(\sup_{(x,v)\in \R^d\times\R^d} \vert \nabla_v f(x,v)\vert_2^3 \right) \vert K\tilde{g}^n_2(\pr_x\cdot)\vert^3
$$
where $\tilde{g}^n_2(x,x'):=\vert\nabla\hat\phi_{\lambda_n}(x-x')\vert_2(1_{\supp_s(f)}(x)+1_{\supp_s(f)}(x'))$, $x,x'\in \Lambda_{\lambda_n}$, $\vert\cdot\vert_2$ denotes Euclidan norm and $\supp_s$ denotes the spatial support of $f$. So we are left to estimate $\mu_n (\vert K\tilde{g}^n_2\circ \sym_n\vert^3)$. By Proposition \ref{prop:tolltolltoll} and using the improved Ruelle bound (\ref{eqn:imprRBinf}) of the $\mu_n\circ \sym_n^{-1}$ we have to prove that for any $M\in \{2,\cdots,6\}$ and any $A,B,C,D,E,F\in \{1,\cdots,M\}$ such that $A\neq B$, $C\neq D$, $E\neq F$ and $\{A,B,C,D,E,F\}=\{1,\cdots,M\}$, it holds
\begin{multline*}
\sup_n \int_{\Lambda_{\lambda_n}^M} \vert \nabla\hat\phi_{\lambda_n}(x_A-x_B)\vert_2\,\vert \nabla\hat\phi_{\lambda_n}(x_C-x_D)\vert_2\,\vert \nabla\hat\phi_{\lambda_n}(x_E-x_F)\vert_2\inf_{1\leq i\leq M} e^{-\beta\sum_{j\neq i} \hat\phi_{\lambda_n}(x_i-x_j)}\\1_{\supp_s(f)}(x_A)1_{\supp_s(f)}(x_C)1_{\supp_s(f)}(x_E)\,dx_1\cdots dx_M<\infty
\end{multline*}
which by uniform boundedness of the $\hat\phi_{\lambda_n}$ from below we may replace by
\begin{multline}\label{eqn:unibound2}
\sup_n \int_{\Lambda_{\lambda_n}^M} c(x_A,x_B)c(x_C,x_D)c(x_E,x_F)\\ 1_{\supp_s(f)}(x_A)1_{\supp_s(f)}(x_C)1_{\supp_s(f)}(x_E)\,dx_1\cdots dx_M<\infty
\end{multline}
with $c(x,x'):=\vert \nabla\hat\phi_{\lambda_n}(x-x')\vert_2 e^{-\frac{\beta}{3}\hat\phi_{\lambda_n}(x-x')}$, $x,x'\in\Lambda_{\lambda_n}$. \\
To prove (\ref{eqn:unibound2}), we integrate successively over all $x_Y$ with $Y\in \{B,D,F\}\setminus \{A,C,E\}$. The integration yields finite values bounded independently of $n$ even if $Y$ appears more than once in the tupel $(A,\cdots,F)$ due to Lemma \ref{lem:momentsofhatphilambda} and H\"older inequality. We continue to integrate over the remaining variables until there is no $\nabla\hat\phi_{\lambda_n}$-term left. For every variable which then remains (there is at least one), there is a $1_{\supp_s(f)}$ left. It follows $\sup_n \mu_n (\vert Kg_2\circ\sym_n\vert^3)<\infty$ and together with (\ref{eqn:tightness1}) the assertion is shown.
\end{proof}

A much simpler case than in Lemma \ref{lem:tness1lemma} is considered in the following corollary. (Note that the first estimate is immediate.)
\begin{corollary}\label{cor:tightnessab}
Let $f: \R^d\times\R^d$ be continuous and continuously differentiable in the $v$-directions such that all the derivatives are bounded, whereas $f$ may be unbounded. Then it holds
$$
\sup_n \mu_n(\vert \nabla_v(Kf\circ \per_n)\vert_2^3)\leq \sup_n \mu_n(\vert \nabla_v(Kf\circ \per_n)\vert_1^3)<\infty
$$
where $\vert \cdot\vert_1$ denotes norm defined by $\vert (y_1,\cdots,y_l)\vert:=\sum_{j=1}^l \vert y_j\vert$ for $(y_1,\cdots,y_l)\in \R^l$, $l\in\N$.
\end{corollary}

\begin{remark}\label{rem:tightnessaa}
For $k\in\N$ consider $\chi_k$ defined as in Section \ref{sec:metric} such that $a$ is twice continuously differentiable and $\nabla_v a, \Delta_v a$ are bounded (and, as before, $a(v)\to\infty$ as $\vert v\vert\to\infty$. Then Lemma \ref{lem:tness1lemma} and Corollary \ref{cor:tightnessab} apply to $f=\chi_k$.
\end{remark}

If a function $f$ is only dependent on $x$-coordinates, $\overline{L_n} (Kf\circ \per_n)$ does not contain $\nabla\hat\phi_{\lambda_n}$, $n\in\N$. This enables us to deal also with a function $S^{\Phi,h_k}$, defined as in Section \ref{sec:metric}. 
\begin{lemma}\label{lem:Phitightness}
Let $\widetilde\Phi$ be as in Lemma \ref{lem:uniform1} (corresponding to the potential $\phi$) and assume (w.l.o.g.) that $\widetilde\Phi(r)=0$ for $r\geq \lambda_1/4$. Let $k\in\N$. It holds $S^{\beta\widetilde\Phi/6,h_k}\circ\per_n\in D(\overline{L_n})$ and
$$
\sup_n \mu_n(\vert \overline{L_n} (S^{\beta\widetilde\Phi/6, h_k}\circ \per_n)\vert)<\infty.
$$
\end{lemma}
\begin{proof}
We first note that $S^{\beta\widetilde\Phi/6, h_k}$ is the $K$-transform of $g: \Gamma_2\to\R$, given by $g(\{x,x'\}):=h_k(x)h_k(x')e^{\beta\widetilde\Phi(x-x')/6}$, $x,x'\in\R^d$. $g(\{x,x'\})$ is equal to $h_k(x)h_k(x')$, when $\vert x-x'\vert\geq \lambda_1/4$. We choose a locally finite open cover $\mathcal U$ of $\R^d$ such that any $\Delta\in\mathcal U$ has diameter $<\lambda_1/4$ and we choose a corresponding partition of unity $(\eta_\Delta)_{\Delta\in\mathcal U}$ consisting of $C^1$-functions. Using this partition (and noting that $h_k$ has compact support) we see that we may replace $S^{\beta\widetilde\Phi/6,h_k}$ by $Kg_{\Delta_1,\Delta_2}$, where $g_{\Delta_1,\Delta_2}: \Gamma_2\to\R$ is defined by
\begin{equation}\label{eqn:definitionofg}
g_{\Delta_1,\Delta_2}(\{x,x'\}):=e^{\frac{\beta}{6}\Phi(\vert x-x'\vert)}\varphi_{\Delta_1}\varphi_{\Delta_2},\quad x,x'\in\R^d,
\end{equation}
for $\Delta_1,\Delta_2\in\mathcal U$, where $\varphi_{\Delta_{1/2}}:=\eta_{\Delta_{1/2}}h_k$. \\[1.5ex]
We first consider the case where $\textnormal{dist}(\Delta_1,\Delta_2)\leq \lambda_1/4$, where $\textnormal{dist}$ denotes the $\vert\cdot\vert$-distance of subsets of $\R^d$. We may assume that $\Delta:=\Delta_1\cup\Delta_2$ is relatively compact in $(-\lambda_1,\lambda_1)^d$ and replace $\per_n$ by $\sym_n$ using spatial translations in $E_{\lambda_n}$ as in the proof of Lemma \ref{lem:tness1lemma} above. By Lemma \ref{lem:functionsinDL}(ii) we have that $(Kg_{\Delta_1,\Delta_2}\circ\sym_n)\in D(\overline{L_n})$, $n\in\N$. For $1\leq i\leq N_n$ and $(x,v)\in E_{\lambda_n}^{N_n}$ we make the following estimate.
\begin{eqnarray*}
\lefteqn{\vert \nabla_{x_i}(Kg_{\Delta_1,\Delta_2}\circ \sym_n)(x,v)\vert}\\
& & =\left\vert \nabla_{x_i} \sum_{j\neq i} e^{\frac{\beta}{6}\widetilde\Phi(\vert x_i-x_j\vert)}\varphi_{\Delta}(x_i)\varphi_\Delta(x_j)\right\vert\\
& & \leq \sum_{j\neq i} \left\vert e^{\frac{\beta}{6}\widetilde\Phi(\vert x_i-x_j\vert)} \varphi_\Delta(x_j)\left(\frac{\beta}{6}\widetilde\Phi'(\vert x_i-x_j\vert)\frac{x_i-x_j}{\vert x_i-x_j\vert}\varphi_\Delta(x_i)+\nabla \varphi_\Delta(x_i)\right)\right\vert \\
& & \leq \sum_{j\neq i} e^{\frac{\beta}{6}\widetilde\Phi(\vert x_i-x_j\vert)}1_\Delta(x_j)\left(\left\vert\frac{\beta}{6}\widetilde\Phi'(\vert x_i-x_j\vert)\right\vert 1_\Delta(x_i)+C 1_\Delta (x_i)\right)
\end{eqnarray*}
for some $C<\infty$. Since $\widetilde\Phi'e^{-\frac{\beta}{6}\widetilde\Phi}$ is bounded and $\widetilde\Phi\geq 0$ the r.h.s.~is estimated by
$$
C'\sum_{j\neq i} e^{\frac{\beta}{3}\widetilde\Phi(\vert x_i-x_j\vert)}1_\Delta(x_i)1_\Delta(x_j)
$$
for some $C'<\infty$ and thus
$$
\vert \overline{L_n}(S^{\frac{\beta}{6}\Phi,\eta_\Delta}\circ\sym_n)(x,v)\vert\leq \sum_{\{i,j\}\subset\{1,\cdots,N\}}(\vert v_i\vert_1+\vert v_j\vert_1)e^{\frac{\beta}{3}\Phi(\vert x_i-x_j\vert)}1_\Delta(x_i)1_\Delta(x_j).
$$
Using Proposition \ref{prop:tolltolltoll} we find that we only have to prove that for all $M\in\{2,\cdots,6\}$, $A,B,C,D,E,F\in \{1,\cdots,M\}$, $\{A,\cdots,F\}=\{1,\cdots,M\}$, $A\neq B$, $C\neq D$, $E\neq F$ the expression
$$
\int_{\Lambda_{\lambda_n}^M} c(x_A,x_B)c(x_C,x_D)c(x_E,x_F)e^{-\frac{2}{M}\beta \sum_{i<j} \hat\phi_{\lambda_n}(x_i-x_j)}\,dx_1\cdots dx_M
$$
is bounded independently of $n$, where $c(x,x')=e^{\frac{\beta}{3}\Phi(\vert x-x'\vert)}1_\Delta(x)1_\Delta(x')$, $x,x'\in\R^d$. But since the integrand is bounded by the properties $\widetilde\Phi$ and the uniform boundedness from below of $\hat\phi_{\lambda_n}$, $n\in\N$, the above integral is estimated by $D\max\{\vol(\Delta)^2,\vol (\Delta)^6\}$ for some $D<\infty$.\shortspacing
Now assume that $\textnormal{dist}(\Delta_1,\Delta_2)>\lambda_1/4$ in (\ref{eqn:definitionofg}). In this case $\widetilde\Phi(\vert x-x'\vert)=0$ for $x\in \Delta_1$, $x'\in\Delta_2$, so we have $g_{\Delta_1,\Delta_2}(\{x,x'\})= \varphi_{\Delta_1}(x) \varphi_{\Delta_2}(x')$, $x,x'\in\R^d$. Using periodicity of the image configurations w.r.t.~$\per_n$ and spatial shifts in $E_{\lambda_n}$ we may assume that $\Delta_1$ and $\Delta_2$ are relatively compact in $(-\lambda_1,\lambda_1)^d$, so $\per_n$ may be replaced by $\sym_n$ and the case $\textnormal{dist}(\Delta_1,\Delta_2)>\lambda_1/4$ is reduced to a (trivial) special case ($\widetilde\Phi\equiv 0$) of the one we treated above.
\end{proof}

Now we arrive at the concluding tightness estimate. The expectation w.r.t.~$P^{(n)}$, $n\in\N$, we denote by $E^{(n)}$ and in the sequel we also use similar notations for expectations w.r.t.~other probability laws.
\begin{lemma}
Let $\widetilde\Phi$ be chosen as in Lemma \ref{lem:Phitightness} and $a$ be chosen as in Remark \ref{rem:tightnessaa}. For each $T>0$ there is a constant $C>0$ such that for $0\leq s<t\leq T<\infty$ it holds
$$
\sup_n E^{(n)}\left[\left(d^{\frac{\beta}{6}\widetilde\Phi,h,a} (\gamma_t,\gamma_s)\right)^3\right]\leq C(t-s)^{3/2},
$$
when $f_k$, $g_k$, $r_k$ and $q_k$ in the definition of $d^{\frac{\beta}{6}\widetilde\Phi,h,a}$ are chosen in a suitable way (see the proof below).
\end{lemma}
\begin{proof}
It holds by the Minkowski inequality and the fact that $\frac{r}{1+r}\leq r$ for $r\geq 0$
\begin{eqnarray}\label{eqn:tightness}
\lefteqn{\left( E^{(n)}\left[\left(d^{\frac{\beta}{6}\widetilde\Phi,h,a}(\gamma_t,\gamma_s)\right)^3\right]\right)^{1/3}}\\
& & \leq \sum_{k=1}^\infty 2^{-k} \left(E^{(n)}\left[\left\vert Kf_k(\gamma_t)-Kf_k(\gamma_s)\right\vert^3\right]\right)^{1/3}\nonumber\\
& & +\sum_{k=1}^\infty 2^{-k} \left(E^{(n)}\left[\left\vert Kg_k(\pr_x\gamma_t)-Kg_k(\pr_x\gamma_s)\right\vert^3\right]\right)^{1/3}\nonumber\\
& & +\sum_{k=1}^\infty 2^{-k} q_k \left(E^{(n)}\left[\left\vert(K\chi_k)(\gamma_t)-(K\chi_k)(\gamma_s)\right\vert^3\right]\right)^{1/3}\nonumber\\
& & +\sum_{k=1}^\infty 2^{-k} r_k \left(E^{(n)}\left[\left\vert S^{\frac{\beta}{6}\widetilde\Phi,hI_k}\circ \pr_x(\gamma_t)-S^{\frac{\beta}{6}\widetilde\Phi,hI_k}\circ \pr_x(\gamma_s)\right\vert^3\right]\right)^{1/3}.\nonumber
\end{eqnarray}
Concerning the first three summands on the r.h.s.~we need to estimate 
\begin{equation}\label{eqn:konsequenz}
E^{(n)}\left[\left\vert K f(\gamma_t)-Kf(\gamma_s)\right\vert^3\right]
\end{equation}
for $f$ as in Lemma \ref{lem:tness1lemma}. It suffices to prove that this expression is bounded by $(t-s)^{3/2}C(f)D(T)$ where $C(f)$ is a constant depending only on $f$ and $D(T)$ depends only on $T$. Then by replacing $f_k$ by $\frac{f_k}{C(f_k)^{1/3}}$ and $g_k$ by $\frac{g_k}{C(g_k)^{1/3}}$ and setting $q_k:=\min\{{C(\chi_k)^{-1/3}},1\}$, $k\in\N$, in the definition of the metric the first three summands in (\ref{eqn:tightness}) are convergent and less or equal than $(t-s)^{1/2}D(T)^{1/3}$. \\
So let $f$ be as in Lemma \ref{lem:tness1lemma}. It holds
$$
E^{(n)}\left[\left\vert Kf(\gamma_t)-Kf(\gamma_s)\right\vert^3\right]=E_n\left[\left\vert (Kf)\circ \per_n(X_t,V_t)-(Kf)\circ \per_n(X_s,V_s) \right\vert^3\right]
$$
It holds $Kf\circ\per_n\in D(\overline{L_n})$. Since $P_n$ solves the martingale problem for $\overline{L_n}$ we find that
$$
M_t^{[Kf\circ\per_n],n}:=Kf\circ\per_n(X_t,V_t)-Kf\circ\per_n (X_0,V_0)-\int_0^t \overline{L_n}(Kf\circ\per_n)(X_r,V_r)\,dr,
$$
$t\geq 0$, defines a $P_n$-martingale. By \cite[Remark 3.13]{CG07a} the quadratic variation process of $(M_t^{[Kf\circ\per_n],n})_{t\geq 0}$ is given by $\left(\int_0^t \frac{2\brotkiste}{\beta}\vert \nabla_v (Kf\circ\per_n)\vert_2^2(X_r,V_r)\,dr\right)_{t\geq 0}$, where $\vert \cdot\vert_2$ denotes Euclidean norm. Using the Burkholder-Davies-Gundy inequality and the H\"older inequality, we find
\begin{eqnarray*}
E_n\left[\left\vert M_t^{[Kf\circ\per_n],n}-M_s^{[Kf\circ\per_n],n}\right\vert^3\right]&\leq& E_n\left[ \left\vert\int_s^t \frac{2\brotkiste}{\beta}\vert \nabla_v (Kf\circ\per_n)\vert_2^2(X_r,V_r)\,dr\right\vert^{3/2}\right]\\&\leq& \left(\frac{2\brotkiste}{\beta}\right)^{3/2}(t-s)^{3/2} \mu_n(\vert \nabla_v (Kf\circ\per_n)\vert_2^3),
\end{eqnarray*}
which can be estimated using Corollary \ref{cor:tightnessab}.
Moreover, it holds by H\"older inequality
\begin{align*}
E_n\left[\left\vert\int_s^t \overline{L_n}(Kf\circ\per_n)(X_r,V_r)\,dr\right\vert^3\right]&\leq (t-s)^3\mu_n(\vert \overline{L_n} (Kf\circ\per_n)\vert^3)\\
&\leq T^{3/2}(t-s)^{3/2}\mu_n(\vert \overline{L_n} (Kf\circ\per_n)\vert^3).
\end{align*}
This can be estimated using Lemma \ref{lem:tness1lemma}. Altogether we have independently of $n$ an estimate of (\ref{eqn:konsequenz}) by $(1+T^{1/2})^3 C(f)(t-s)^{3/2}$ for some constant $C(f)$, concluding the consideration of the first three summands in (\ref{eqn:tightness}).\shortspacing
Concerning the fourth summand we first note that (denoting $S^{\frac{\beta}{6}\widetilde\Phi,hI_k}\circ \pr_x$ also by $S^{\frac{\beta}{6}\widetilde\Phi,hI_k}$)
\begin{multline*}
E^{(n)}\left[\left\vert S^{\frac{\beta}{6}\widetilde\Phi,hI_k}(\gamma_t)-S^{\frac{\beta}{6}\widetilde\Phi,hI_k}(\gamma_s)\right\vert^3\right]\\
=E_n \left[\left\vert(S^{\frac{\beta}{6}\widetilde\Phi,hI_k}\circ\per_n)(X_t,V_t)-(S^{\frac{\beta}{6}\widetilde\Phi,hI_k}\circ\per_n)(X_s,V_s)\right\vert^3\right].
\end{multline*}
$S^{\frac{\beta}{6}\widetilde\Phi,hI_k}\circ\per_n$ is an element of $D(\overline{L_n})$ (cf.~Lemma \ref{lem:Phitightness}) such that for the corresponding martingale it holds $M^{[S^{\frac{\beta}{6}\widetilde\Phi,hI_k}\circ\per_n],n}_t=0$ $P_n$-a.s.~for all $t\geq 0$. So
\begin{eqnarray*}
\lefteqn{E_n \left[\left\vert(S^{\frac{\beta}{6}\widetilde\Phi,hI_k}\circ\per_n)(X_t,V_t)-(S^{\frac{\beta}{6}\widetilde\Phi,hI_k}\circ\per_n)(X_s,V_s)\right\vert^3\right]}\\
& &= E_n \left[\left\vert\int_s^t \vert \overline{L_n} (S^{\frac{\beta}{6}\widetilde\Phi,hI_k}\circ\per_n)\vert^3(X_r,V_r)\,dr\right\vert^3\right]\leq (t-s)^3\mu_n (\vert \overline{L_n} (S^{\frac{\beta}{6}\widetilde\Phi,hI_k}\circ\per_n)\vert^3)\\
& &\leq T^{3/2} (t-s)^{3/2}\mu_n (\vert \overline{L_n} (S^{\frac{\beta}{6}\widetilde\Phi,hI_k}\circ\per_n)\vert^3),
\end{eqnarray*}
which can be estimated with the help of Lemma \ref{lem:Phitightness} by $T^{3/2}R_k (t-s)^{3/2}$ for some $R_k\in\R^+$. Setting $r_k:=\min\{R_k^{-1/3},1\}$ in the definition of $d^{\frac{\beta}{6}\widetilde\Phi,a,h}$ we have an estimate for the fourth summand in (\ref{eqn:tightness}). This completes the proof.
\end{proof}

\begin{remark}
In \cite{GKR04} the Lyons-Zheng decomposition was used in order to obtain the estimate corresponding to the above lemma. At first sight, using such a decomposition seems to be a significant simplification of the proof given above, since one avoids having to estimate the bounded variation terms. Therefore, we should mention that this is not possible here, since we are in a non-reversible situation and the method of proving tightness by a forward/backward martingale decomposition depends heavily on reversibility of the processes (cf.~\cite[Proof of Lemma 5.3]{GKR04}).
\end{remark}

We obtain the desired tightness result.
\begin{theorem}\label{thm:tightness}
Let $\phi$ be a symmetric pair interaction fulfilling (RP), (T), (BB), (WD) and (IDF). Let $(N_n)_n\subset \N$ and $(\lambda_n)_n\subset\R^+$ be sequences such that $\lambda_n\uparrow\infty$ and $\frac{N_n}{(2\lambda_n)^d}\to\rho\in [0,\infty)$ as $n\to\infty$. Let $P^{(n)}:=P_n\circ (\per_n^{\otimes [0,\infty)})^{-1}$, $n\in\N$, be defined as in Section \ref{sub:fddyn}. Then the sequence $(P^{(n)})_{n\in\N}$ is a tight sequence of probability laws on $C([0,\infty),\Gamma^v)$.
\end{theorem}
\begin{proof}
Using standard tightness results (cf.~\cite[Theorem 3.8.6 and Theorem 3.8.8]{EK86}) we obtain tightness of $(P^{(n)})_{n\in\N}$ as probability laws on $D([0,\infty),\Gamma^v)$, the space of cadlag paths in $\Gamma^v$. By \cite[Exercise 3.25(c)]{EK86} we find that any weak accumulation point of $(P^{(n)})_{n\in\N}$ assigns full measure to the space $C([0,\infty),\Gamma^v)$ of continuous paths, hence by \cite[Exercise 3.25(d)]{EK86} the assertion follows.
\end{proof}

\end{subsection}

\begin{subsection}{The martingale problem}\label{sub:MP}

By now we know that in the situation of Theorem \ref{thm:tightness} we have at least one accumulation point $P$ of $(P^{(n)})_n$. Let $P^{(n_k)}\to P$ weakly as $k\to\infty$. Note that then also the sequence $\mu^{(n_k)}$ converges weakly and its weak limit $\mu$ is the invariant initial distribution of $P$. Moreover, also $\mu_{n_k}\circ\sym_{n_k}^{-1}\to \mu$ weakly as $k\to\infty$. In this section we verify that $P$ is the law of an infinite particle Langevin dynamics.\shortspacing
\noindent We first prove some preliminary technical properties of the generator $L$ (cf.~(\ref{eqn:generatorL})). We do not bother about the possibility of generalizing assertions in Lemma \ref{lem:cool} below to all $n\in\N$, since in this section we are only interested in asymptotic properties of the sequences $(P^{(n)})_{n\in\N}$, $(\mu^{(n)})_{n\in\N}$.
\begin{lemma}\label{lem:cool}
Let $F=g_F(\langle\{f_i\}_{i=1}^K,\cdot\rangle)\in \mathcal FC_b^\infty(\mathcal D_s,\Gamma^v)$. Choose $n_0\in\N$ such that each $f_i$, $i=1,\cdots,K$, has support in $(-\lambda_{n_0},\lambda_{n_0})^d\times\R^d$.
\begin{enumerate}
\item Let $n\in\N$. The expression $LF(\gamma)$ is well-defined for $\mu_n\circ\sym_n^{-1}$-a.e.~$\gamma\in\Gamma^v$, i.e.~the sums in the definition of $LF$ converge absolutely $\mu_n\circ\sym_n^{-1}$-a.s.~and $LF$ is $\mu_n\circ\sym_n^{-1}$-a.s.~independent of the version one chooses for $\nabla\phi$. Moreover, it holds $\sup_{n\in\N} \Vert LF\Vert_{L^1(\Gamma^v;\mu_n\circ\sym_n^{-1})}<\infty$.
\item Let $n\geq n_0$. Then $LF$ is well-defined $\mu^{(n)}$-a.e.~and $\sup_{n\geq n_0} \Vert LF\Vert_{L^1(\Gamma^v;\mu^{(n)})}<\infty$ with the restriction that the versions for $\nabla\phi$ have to be chosoen in a way such that $\sum_{r\in\Z^d} \nabla\phi(2\lambda_n r)=0$ for all $n\in\N$.
\item With the restriction from (ii), for any $n\geq n_0$, $t\geq 0$ the integral $\int_0^t LF(\gamma_r)\,dr$ is well-defined $P^{(n)}$-a.s.
\item If $P$ is the weak limit of a sequence $(P^{(n)})_{n\in\N}$ as above, then (i) holds with $\mu^{(n)}$ replaced by $\mu$ and (iii) holds with $P^{(n)}$ replaced by $P$.
\end{enumerate}
\end{lemma}
\begin{remark}
The restriction in Lemma \ref{lem:cool}(ii) means that the forces acting between a particle and its periodic copies sum up to $0$, which is a quite natural assumption. Note that it is not an additional assumption on $\phi$. It is only introduced for technical reasons (cf.~(\ref{eqn:zonk}) below) and it can be dropped when only considering the limiting process, as one sees in (the proof of) Lemma \ref{lem:cool}(iv).
\end{remark}
\begin{proof}
$LF$ consists (except of multiplication by bounded continuous partial derivatives of $g_F$) of two types of sums. The first (e.g.~$\langle \Delta_v f_i,\cdot\rangle$) are $K$-transforms of functions in $\mathcal D_s$. By Lemma \ref{lem:vonkuna} and the uniform Ruelle bound the assertion is easily shown for these sums. We concentrate on the second type of sums, which are $K$-transforms of functions $g: \Gamma_2^v\to\R$ of the form $g(\{(x,v),(x',v')\})=\nabla\phi(x-x')(\nabla_v f(x,v)-\nabla_v(f(x',v'))$, $(x,v),(x',v')\in\R^d\times\R^d$, with $f\in \mathcal D_s$. Let us prove that $g\in L^1(\Gamma^v_2; \overline{h}d\lambda^v)$ with $\overline{h}$ as in Remark \ref{rem:overlinephi} (here we identify $\overline{h}$ and $\overline{h}\circ\pr_x$). It holds
\begin{eqnarray}\label{eqn:uniforminteg}
\lefteqn{\frac{1}{(2\pi/\beta)^d} \int_{\R^{4d}} \vert g(\{(x,v),(x',v')\})\vert e^{-(\beta/2)(v^2+{v'}^2)}e^{-\beta \overline{\phi}(x-x')}\,dx\,dx'\,dv\,dv'}\\
& & \leq 2\Vert \,\vert \nabla_v f\vert_1 \,\Vert_\infty \int_{\R^{2d}} \vert \nabla \phi(x-x')\vert e^{-\beta \overline{\phi}(x-x')} 1_{\supp_s(f)}(x)dxdx'\nonumber\\
& & \leq 2\Vert \,\vert \nabla_v f\vert_1 \,\Vert_\infty \vol(\supp_s(f))\Vert \nabla \phi\Vert_{L^1(\R^d;e^{-\beta\overline{\phi}}dx)}.\nonumber
\end{eqnarray}
The proof of (i) and of the first assertion of (iv) are now completed by Lemma \ref{lem:vonkuna}, the improved Ruelle bound, (WD) and the properties of $\overline{\phi}$ (cf.~Remark \ref{rem:overlinephi}).\shortspacing
Let us now prove (ii). Let $n\geq n_0$. Due to the condition (IDF) and the Lebesgue-a.e.~boundedness of $\sum_{r\in\Z^d}\bratwurst(\cdot+2\lambda r)$ in $\Lambda_{\lambda_n}$ (cf.~the proof of Lemma \ref{lem:intofhatphilambda}(ii)), hence in $\Lambda_{2\lambda_n}$, the sum $\sum_{r\in\Z^d} \nabla\phi(\cdot+2\lambda r)$ defining $\nabla\hat\phi_{\lambda_n}$ (as element of $L^1(\Lambda_{2\lambda_n};e^{-\hat\phi_{\lambda_n}}dx)$) converges Lebesgue-a.e. Any version of $\nabla\phi$ uniquely determines by this sum a corresponding version of $\nabla\hat\phi_{\lambda_n}$, when we set $\nabla\hat\phi_{\lambda_n}$ equal to $0$ where it does not converge. Fixing a version of $\nabla\phi$ we define for $\gamma\in\Gamma^v$
\begin{align*}
L_{(n)}F(\gamma):=&\sum_{l,l'=1}^K \frac{\brotkiste}{\beta}\partial_{l}\partial_{l'} g_F(\langle \{f_i\}_{i=1}^K,\gamma\rangle) \langle (\nabla_v f_l)(\nabla_v f_{l'}),\gamma\rangle \\
&+\sum_{l=1}^K \partial_l g_F(\langle \{f_i\}_{i=1}^K,\gamma\rangle)\Bigg(\left\langle \frac{\brotkiste}{\beta}\Delta_v f_l-\brotkiste v\nabla_v f_l+v\nabla_x f_l,\gamma\right\rangle\nonumber\\
&-\sum_{\{(x,v),(x',v')\}\subset\gamma} \nabla\hat\phi_{\lambda_n}(x-x')(\nabla_v f_l(x,v)-\nabla_v f_l(x',v'))\Bigg)\nonumber
\end{align*}

We prove $\mu_n\circ\sym_n^{-1}$-a.s.~absolute convergence of the sums occurring in this definition (including the summation defining $\nabla\hat\phi_{\lambda_n}$), which reduces to prove that $g_n\in L^1(\Gamma_2;\overline{h}d\lambda)$ for $g_n: \Gamma_2\to\R$ of the form $g_n(\{x,x'\})=\sum_{r,r'\in\Z^d} \vert \nabla\phi(x-2\lambda_n r'-x'+\lambda_n r)\vert_2 (1_{\supp_s(f)}(x+2\lambda_n r)+1_{\supp_s(f)}(x'+2\lambda_n r'))$, $x,x'\in \Lambda_{\lambda_n}$, with $f\in \mathcal D_s$, $\supp_s(f)\subset \Lambda_{\lambda_{n_0}}$. Note that the summands are equal to $0$ whenever $r\neq 0\neq r'$. We make the following estimate, using the abbreviation $\vert \nabla\hat\phi\vert_{\lambda_n}:=\sum_{r\in\Z^d} \vert \nabla\phi(\cdot+2\lambda_n r)\vert$:
\begin{eqnarray}\label{eqn:sokrates}
\lefteqn{\int_{\Lambda_{\lambda_n}^2} \sum_{r\in\Z^d} \vert \nabla\phi(x+2\lambda_n r-x')\vert 1_{\supp_s(f)}(x')e^{-\beta\overline{\phi}(x-x')}\,dx\,dx'}\\
& & \leq \vol(\supp_s(f)) \left(\Vert \,\vert \nabla\hat\phi\vert_{\lambda_n} \Vert_{L^1(\Lambda_{\lambda_n};e^{-\beta\overline{\phi}}dx)}+D \Vert  \,\vert \nabla\hat\phi\vert_{\lambda_n} \Vert_{L^1(\Lambda_{\lambda_n+\delta}\setminus\Lambda_{\lambda_n};dx)}\right)\nonumber\\
& & \leq \vol(\supp_s(f)) \Big(\Vert \nabla\phi \Vert_{L^1(\Lambda_{\lambda_n};e^{-\beta\overline{\phi}}dx)}
+D \Vert \nabla\phi \Vert_{L^1(\R^d\setminus\Lambda_{\lambda_n};dx)}
\nonumber\\
& & \quad\quad\quad\quad\quad\quad\quad\quad+2^dD \Vert \,\vert \nabla\hat\phi\vert_{\lambda_n} \Vert_{L^1(\Lambda_{\lambda_n}\setminus\Lambda_{\lambda_n-\delta};dx)}\Big),\nonumber\\
& & \leq \vol(\supp_s(f)) \Big(\Vert \nabla\phi \Vert_{L^1(\R^d;e^{-\beta\overline{\phi}}dx)}
+D \Vert \nabla\phi \Vert_{L^1(\R^d\setminus\Lambda_{\lambda_n};dx)}\nonumber\\
& & \quad\quad\quad\quad\quad\quad\quad\quad+2^dD \Vert \nabla\phi \Vert_{L^1(\R^d\setminus\Lambda_{\lambda_n-\delta};dx)}\Big).\nonumber
\end{eqnarray}
Here $0<\delta<\lambda_{n_0}$ shall be such that $\supp_s(f_i)\subset(-\delta,\delta)^d$, $1\leq i\leq K$, and $D:=\sup_{x\in\R^d}e^{-\beta\overline{\phi}(x)}$. Note that (\ref{eqn:sokrates}) yields an estimate which is independent of $n\geq n_0$. In fact, the last two summands on the r.h.s.~tend to $0$ as $n\to\infty$. By the integrability assumption in (WD), by Lemma \ref{lem:vonkuna}, the improved Ruelle bound and the properties of $\overline{\phi}$ we now have shown that (i) holds with $L$ replaced by $L_{(n)}$. Due to the $\mu_n$-a.s.~absolute convergence of the sums in the definition of $L_{(n)}F$ we may change the order of summation. Using a version of $\nabla\phi$ as specified in (ii), we obtain that 
\begin{equation}\label{eqn:zonk}
L_{(n)}F\circ\sym_n=LF\circ\per_n\quad \mbox{$\mu_n$-a.s.}
\end{equation}
Hence (ii) follows.\shortspacing
(iii) follows from Fubini's theorem, (ii) and the fact that
$$
E^{(n)}\int_0^t \vert LF(\gamma_r)\vert\,dr\leq t\Vert LF\Vert_{L^1(\Gamma^v;\mu^{(n)})}.
$$
The second assertion of (iv) follows in the same mannner from the first assertion.
\end{proof}
From the above proof we can conclude the following uniform approximation result.
\begin{lemma}\label{corlem:cool}
Let $F$ be as in Lemma \ref{lem:cool}. Then there exists a sequence $(H_l)_{l\in\N}$ of bounded continuous cylinder functions $H_l: \Gamma^v\to\R$ (cf.~Section \ref{sub:weaklimits}) such that 
$$
\lim_{l\to\infty} \limsup_{k\to\infty} \Vert H_l-LF\Vert_{L^1(\Gamma^v;\mu^{(n_k)})}=0
$$ 
and $H_l\to LF$ as $l\to\infty$ in $L^1(\Gamma^v;\mu)$.
\end{lemma}
\begin{proof}
When $\mu^{(n_k)}$ is replaced by $\mu_{n_k}\circ\sym_{n_k}^{-1}$, the assertion is  a consequence of Lemma \ref{lem:improvedRB}(ii)  (cf.~Remark \ref{rem:overlinephi}) and the proof of Lemma \ref{lem:cool} (in particular (\ref{eqn:uniforminteg})). Let $(H_l)_{l\in\N}$ be a corresponding sequence of bounded continuous cylinder functions. \\
Fix $l\in\N$ and let $n_0\in\N$ be as in the proof of Lemma \ref{lem:cool}. Choose $k_0$ large enough such that $n_{k_0}\geq n_0$ and such that $H_l$ depends only on the configuration in $\Lambda_{\lambda_{n_k}}$ for $k\geq k_0$. By (\ref{eqn:zonk}) it holds for $k\geq k_0$
\begin{align*}
\Vert H_l-LF\Vert_{L^1(\Gamma^v;\mu^{(n_k)})}&=\Vert H_l-L_{(n_k)}F\Vert_{L^1(\Gamma^v;\mu_{n_k}\circ\sym_{n_k}^{-1})}\\
&\leq \Vert H_l-LF\Vert_{L^1(\Gamma^v;\mu_{n_k}\circ\sym_{n_k}^{-1})}+\Vert LF-L_{(n_k)}F\Vert_{L^1(\Gamma^v;\mu_{n_k}\circ\sym_{n_k}^{-1})}.
\end{align*}
So, we are left to prove that $\lim_{k\to\infty}\Vert LF-L_{(n_k)}F\Vert_{L^1(\Gamma^v;\mu_{n_k}\circ\sym_{n_k}^{-1})}=0$. This reduces to proving that $\lim_{n\to\infty}\Vert Kg\Vert_{L^1(\Gamma^v;\mu_{n}\circ\sym_{n}^{-1})}=0$ for $g_n: \Gamma_2^v\to\R$ of the form 
$$
g_n(\{(x,v),(x',v')\}):=(\nabla\phi-\nabla\hat\phi_{\lambda_n})(x-x')(\nabla_v f(x,v)-\nabla_v f(x',v')),
$$ 
$(x,v),(x',v')\in E_{\lambda_{n}}^2$, with $f\in\mathcal D_s$, $\supp(f)\subset (\delta,\delta)^d$ for some $0<\delta<\lambda_{n_0}$. Applying Lemma \ref{lem:vonkuna} and the improved Ruelle bound we find that it suffices to verify that with $\overline{\phi}$ as in Remark \ref{rem:overlinephi} it holds
\begin{equation}\label{eqn:Huchsowas}
\lim_{n\to\infty}\Vert \nabla\phi-\nabla\hat\phi_{\lambda_{n}}\Vert_{L^1(\Lambda_{\lambda_{n}+\delta};e^{-\beta \overline{\phi}}dx)}=0.
\end{equation}
For $k\geq k_0$ it holds
\begin{eqnarray*}
\lefteqn{\Vert \nabla\phi-\nabla\hat\phi_{\lambda_{n}}\Vert_{L^1(\Lambda_{\lambda_{n}+\delta};e^{-\beta \overline{\phi}}dx)}}\\
& & \leq \Vert \nabla\phi-\nabla\hat\phi_{\lambda_{n}}\Vert_{L^1(\Lambda_{\lambda_{n}};e^{-\beta \overline{\phi}}dx)}+\Vert \nabla\phi\Vert_{L^1(\Lambda_{\lambda_{n}+\delta}\setminus \Lambda_{\lambda_{n}};e^{-\beta \overline{\phi}}dx)}\\& &\quad+\Vert \nabla\hat\phi_{\lambda_{n}}\Vert_{L^1(\Lambda_{\lambda_{n}+\delta}\setminus \Lambda_{\lambda_{n}};e^{-\beta \overline{\phi}}dx)}\\
& & \leq D \Vert \nabla\phi\Vert_{L^1(\R^d\setminus\Lambda_{\lambda_{n}};dx)}+D \Vert \nabla\phi\Vert_{L^1(\R^d\setminus \Lambda_{\lambda_{n}};dx)}+2^dD\Vert \nabla\hat\phi_{\lambda_{n}}\Vert_{L^1(\Lambda_{\lambda_{n}}\setminus \Lambda_{\lambda_{n}-\delta};dx)}\\
& &\leq 2D \Vert \nabla\phi\Vert_{L^1(\R^d\setminus\Lambda_{\lambda_{n}};dx)}+ 2^d D\Vert \nabla\phi\Vert_{L^1(\R^d\setminus\Lambda_{\lambda_{n}-\delta};dx)},
\end{eqnarray*}
where $D:=\sup_{x\in\R^d} e^{-\beta\overline{\phi}(x)}$. (\ref{eqn:Huchsowas}) follows, since $\nabla\phi\in L^1(\R^d;e^{-\beta\phi} dx)$ implies $\nabla\phi\in L^1(\{\vert x\vert>a\}; dx)$ for any $a>0$ due to boundedness of $\phi$ in $\{\vert x\vert>a\}$.
\end{proof}
\ \\
The laws $P^{(n)}$, $n\in\N$, behave nicely w.r.t.~the operator $L$ restricted to functions $F=g_F(\langle f_1,\cdot\rangle,\cdots,\langle f_K,\cdot\rangle)\in \mathcal FC_b^\infty(\mathcal D_s,\Gamma^v)$ depending only on the particles in $\Lambda_{\lambda_n}$: We find that $M_t^{[F]}:=F(\gamma_t)-F(\gamma_0)-\int_0^t LF(\gamma_r)\,dr$, $t\geq 0$, defines a martingale w.r.t.~$P^{(n)}$. To see this, first note that for such $F$ we have by (\ref{eqn:zonk}) that $LF\circ\per_n=\overline{L_n}(F\circ\sym_n)$ holds $\mu_n$-a.s. We obtain for any $0\leq s\leq t$ and any bounded function $G: C([0,\infty),\Gamma^v)\to\R$ which is $\sigma(\gamma_r: 0\leq r\leq s)$-measurable
\begin{multline*}
E^{(n)}\left[G\left(F(\gamma_t)-F(\gamma_s)-\int_s^t LF(\gamma_r)\,dr\right)\right]\\
=E_{n}\Bigg[G\circ\per_n^{\otimes [0,\infty)}\Bigg( F\circ\sym_n(X_t,V_t)-F\circ\sym_n(X_s,V_s)\\-\int_s^t \overline{L_{n}} (F\circ\sym_n) (X_r,V_r)\,dr\Bigg)\Bigg],
\end{multline*}
which is equal to $0$ due to the fact that $P_{n}$ solves the martingale problem for $(\overline{L_{n}},D(\overline{L_{n}}))$.\\[2ex]
We arrive at the result completing the construction.
\begin{theorem}\label{thm:martingaleproblemlemma}
Assumptions as in Theorem \ref{thm:tightness}. Let $P$ be an accumulation point of the sequence $(P^{(n)})_{n\in\N}$. Then $P$ solves the martingale problem for $(L,\mathcal FC_b^\infty(\mathcal D_s,\Gamma^v))$.
\end{theorem}

\begin{proof}
Let $P^{(n_k)}\to P$ weakly as $k\to\infty$ and denote the weak limit of $(\mu^{(n_k)})_{k\in\N}$ by $\mu$. Let $F=g_F(\langle \{f_i\}_{i=1}^K,\cdot\rangle)\in \mathcal FC_b^\infty(\mathcal D_s,\Gamma^v)$. Then there exists $k_0\in\N$ such that $(-\lambda_{n_k},\lambda_{n_k})^d$ contains the support of all $f_i$, $1=1,\cdots,K$, for all $k\geq k_0$. We have to prove that for any $0\leq s<t<\infty$ and for any $\sigma(\gamma_r: 0\leq r\leq s)$-measurable $G: C([0,\infty),\Gamma^v)\to\R$ being bounded and continuous it holds
$$
E\left[ (M_t^{[F]}-M_s^{[F]})G\right]=0,
$$
where $(M_{t'}^{[F]})_{t'\geq 0}$ is as defined above. We already know that $(M_{t'}^{[F]})_{t'\geq 0}$ is a martingale w.r.t.~$P^{(n_k)}$, $k\geq k_0$. Hence we are left to prove that $E^{(n_k)} [G M_r^{[F]}]\to E[GM_r^{[F]}]$ as $k\to\infty$ for $r\in \{t,s\}$. Due to the fact that $G$ and $F$ are continuous and bounded, this reduces to proving
\begin{equation}\label{eqn:thelastthingtoprove}
E^{(n_k)} \left[G\int_0^r LF(\gamma_{r'})\,dr'\right]\to E \left[G\int_0^r LF(\gamma_{r'})\,dr'\right]\quad\mbox{as $k\to\infty$ for each $r\in \{t,s\}$.}
\end{equation}
Choosing $(H_l)_{l\geq 0}$ according to Lemma \ref{corlem:cool} we find that for any $r'\geq 0$, $l\in\N$ and $k\geq k_0$ it holds
\begin{eqnarray*}
\lefteqn{\big\vert E^{(n_k)} [GLF(\gamma_{r'})]-E[GLF(\gamma_{r'})]\big\vert}\\
& & \leq \big\vert E^{(n_k)} [G(LF(\gamma_{r'})-H_l(\gamma_{r'}))]\big\vert+\big\vert E^{(n_k)}[GH_l(\gamma_{r'})]-E[GH_l(\gamma_{r'})]\big\vert\\
& & \quad+\big\vert E [G(LF(\gamma_{r'})-H_l(\gamma_{r'}))]\big\vert\rule{0cm}{0.45cm}\\
& & \leq \Vert G\Vert_\infty \left(\Vert LF-H_l\Vert_{L^1(\Gamma^v;\mu^{(n_k)})}+\Vert LF-H_l\Vert_{L^1(\Gamma^v;\mu)}\right)\\
& & \quad+\big\vert E^{(n_k)}[GH_l(\gamma_{r'})]-E[GH_l(\gamma_{r'})]\big\vert,
\end{eqnarray*}
where we used the fact that the one-dimensional distributions of $P^{(n_k)}$, $P$ are given by $\mu^{(n_k)}$, $\mu$, respectively. Hence by continuity and boundedness of $G$ and $H_l$ and by weak convergence of $P^{(n_k)}$ towards $P$
\begin{multline*}
\limsup_{k\to\infty}\left\vert E^{(n_k)} [GLF(\gamma_{r'})]-E[GLF(\gamma_{r'})]\right\vert\\
\leq \Vert G\Vert_\infty \left(\limsup_{k\to\infty}\Vert LF-H_l\Vert_{L^1(\Gamma^v;\mu^{(n_k)})}+\Vert LF-H_l\Vert_{L^1(\Gamma^v;\mu)}\right),
\end{multline*}
which by Lemma \ref{corlem:cool} can be made arbitrarily small by choosing $l$ large. $E^{(n)}[GLF(\gamma_{r'})]$ is bounded uniformly in $r'\in [0,\infty)$ by $\Vert G\Vert_\infty \Vert LF\Vert_{L^1(\Gamma^v;\mu^{(n)})}$ which is bounded uniformly in $n\geq n_{k_0}$ due to Lemma \ref{lem:cool}(ii). Hence (\ref{eqn:thelastthingtoprove}) follows by Lebesgue's dominated convergence theorem.
\end{proof}

\begin{remark}\label{rem:GKR04}
The results from Section \ref{sub:weaklimits} which we used to prove Theorem \ref{thm:martingaleproblemlemma} also generalize some of the results in \cite{GKR04}. In the differentiability assumption (D) given there we replace continuous differentiability of the potential $\phi$ on $\R^d\setminus\{0\}$ by weak differentiability and continuity. (Moreover, the assumption on $\Phi$ in (D) can be dropped, cf.~Lemma \ref{lem:Theothercapitalphi} above.) Then the existence of the approximating dynamics (cf.~\cite[Theorem 4.1]{GKR04} and \cite{FG04}) is still ensured. Moreover, the tightness result \cite[Theorem 5.1]{GKR04} does not depend on continuous differentiability of $\phi$. We do not consider the question whether the results of \cite[Section 5.2]{GKR04} are still valid, but focus on the martingale problem \cite[Theorem 5.10]{GKR04}.\\
Since the invariant initial distributions $\mu^{(N)}$, $N\in\N$, of the approximating dynamics $P^{(N)}$ in \cite{GKR04} fulfill a uniform improved Ruelle bound, this bound extends to the invariant distribution $\mu$ of any weak accumulation point $P$ by Lemma \ref{lem:improvedRB}(i). This can be used to prove well-definedness of the expressions in \cite[(5.19),(5.20)]{GKR04} a.s.~w.r.t.~the $\mu^{(N)},\mu$ resp. $P^{(N)},P$ similarly to Lemma \ref{lem:cool}(i),(iii),(iv). Moreover, we also obtain an approximation of elements $H_\mu F$, $F\in\mathcal FC_b^\infty(\mathcal D,\Gamma)$, of the range of the generator $(H_\mu,\mathcal FC_b^\infty(\mathcal D,\Gamma))$ (cf.~\cite[(5.19)]{GKR04}) by bounded continuous cylinder functions in $L^1(\Gamma;\mu^{(N)})$ uniformly in $N$ and in $L^1(\Gamma;\mu)$ (using Lemma \ref{lem:improvedRB}(ii)). This replaces the only argument in the proof of \cite[Theorem 5.10]{GKR04} making use of the continuity of the derivatives of $\phi$ (cf.~\cite[p.~150]{GKR04}).
Finally, note that also \cite[Theorem 6.1]{GKR04} is valid in the new setting.
\end{remark}

\end{subsection}

\end{section}
\setcounter{equation}{0}

\begin{section}{The initial configuration}\label{sec:georgii}

Consider the situation of Theorems \ref{thm:tightness}, \ref{thm:martingaleproblemlemma}. We now focus on the initial distribution $\mu$ of the process constructed there. We already saw in Remark \ref{rem:periodictightness} that it is an accumulation point of the sequence $(\mu_n\circ\sym_n^{-1})_{n\in\N}$ or, which is equivalent, of the sequence $(\mu_n\circ\per_n^{-1})_{n\in\N}$. Our aim is to prove that $\mu$ is a tempered grand canonical Gibbs measure (cf.~\cite[p.1348]{Ge95} for the definition). In order to do so, we adapt considerations from \cite{Ge95} on the equivalence of the microcanonical and the grand canonical ensemble in order to extend some results obtained there to the canonical ensemble. We use results and notations from \cite{GZ93}, \cite{Ge94} and \cite{Ge95} in order to do so. As in \cite{Ge95} we restrict to the case where $\lambda_n=n+\frac{1}{2}$, $n\in\N$.\shortspacing
Let $\mathcal P$ be the space of probability measures $P$ on $\Gamma^v$ having finite density and kinetic energy density, i.e.~$\int_{\Gamma^v} \sum_{(x,v)\in\gamma\cap\mathcal C} (1+\vert v\vert^2)\,dP(\gamma)<\infty$, where $\mathcal C:=[0,1]^d\times\R^d$. Denote by $\mathcal P_\theta\subset\mathcal P$ the subset of probability measures which are invariant w.r.t.~spatial translations $\vartheta_a(\gamma):=\gamma-(a,0)$, $\gamma\in\Gamma^v$, $a\in\R^d$. The set of \emph{tame} cylinder functions $F$, i.e.~functions $F: \Gamma^v\to\R$ such that there exist $\Lambda\subset\R^d$ and $C<\infty$ with $F(\gamma)=F(\gamma\cap (\Lambda\times\R^d))$ and $\vert F(\gamma)\vert \leq C+C\sum_{(x,v)\in\gamma\cap(\Lambda\times\R^d)} (1+\vert v\vert)$, $\gamma\in\Gamma^v$, is denoted by $\mathcal L$. On $P_\theta$ the topology $\tau_{\mathcal L}$ is defined as the weakest topology such that all mappings $P\mapsto P(F)=\int_{\Gamma^v} F\,dP$, $F\in\mathcal L$, are continuous. This topology is finer than the weak topology on the space of probability measures on $\Gamma^v$.\shortspacing
We now state the result which is shown in the course of this section.
\begin{theorem}\label{thm:georgii}
Let a symmetric measurable function $\phi: \R^d\to\R\cup\{\infty\}$ fulfilling (RP), (T), (BB) as given in Section \ref{sub:conditions}. Let $\lambda_n:=n+\frac{1}{2}$, $n\in\N$, and $(N_n)_{n\in\N}\subset\N$ be such that $\frac{N_n}{(2\lambda_n)^d}\to \rho\in (0,\infty)$ as $n\to\infty$. Define $\mu^{[n]}:=\mu_n\circ\sym_n$ with $\mu_n:=\mu_{\lambda_n,N_n}$ and $\sym_n:=\sym_{\lambda_n,N_n}$ as defined in Section \ref{sub:fddyn}. Then $(\mu^{[n]})_{n\in\N}$ is relatively compact w.r.t.~$\tau_{\mathcal L}$ and any accumulation point is a grand canonical Gibbs measure.
\end{theorem}

\begin{remark}
The conditions (RP), (T), (BB) imply conditions (A1) and the non-hard-core version of (A2) from \cite{Ge95}. The proof given below works for the latter conditions. We exclude the case of hard-core potentials for convenience and since it is not treated in the preceding sections of this article.
\end{remark}

The proof of the above theorem mainly follows the lines of arguments in \cite{Ge95}, in particular the beginning of \cite[Section 6]{Ge95}. However, there are some modifications to be made which can only be explained in the presence of some details. Note that here we only deal with the case of periodic boundary condition (this simplifies the considerations).\shortspacing
We introduce some more notations from \cite{Ge95}. By $\rho: \mathcal P_\theta\to [0,\infty)$ we denote the $\tau_{\mathcal L}$-continuous function assigning to each $P\in \mathcal P_\theta$ its average particle density $\rho(P):=\int_{\Gamma^v} \sharp(\gamma\cap \mathcal C)\,dP(\gamma)$. $U_{\pot}: \mathcal P_\theta\to \R\cup\{\infty\}$ denotes the mean potential energy, which is given by 
$$
U_{\pot}(P):=\lim_{n\to\infty} \frac{1}{(2\lambda_n)^d} \int_{\Gamma^v} \sum_{\{x,x'\}\subset(\pr_x\gamma)\cap \Lambda_{\lambda_n}} \phi(x-x')\,dP(\gamma).
$$
The mean kinetic energy $U_{\kin}: \mathcal P_\theta\to\R$ is defined by 
$$
U_{\kin}(P):=\int_{\Gamma^v} \frac{1}{2}\sum_{(x,v)\in(\gamma\cap \mathcal C)}\vert v\vert^2\,dP(\gamma).
$$
Both functions $U_{\pot}, U_{\kin}$ are measure affine and lower semicontinuous w.r.t.~$\tau_{\mathcal L}$ (cf.~\cite[p.~1349]{Ge95} and also \cite{Ge94}). Set $U:=U_{\kin}+U_{\pot}$.\\
Moreover, we need to make use of the mean entropy $S: \mathcal P_\theta\to\R\cup\{-\infty\}$, which is an upper semicontinuous measure affine function and such that for $c\in\R$, $\varepsilon\geq 0$ the sets
$$
\{P\in \mathcal P_\theta: S(P)\geq c, U_{\kin}(P)\leq \varepsilon\}
$$
are compact and sequentially compact w.r.t.~$\tau_{\mathcal L}$ (cf.~\cite[p.~1349 and Lemma 4.2]{Ge95}). \shortspacing
We also need to consider entropy functionals $I_\beta: \mathcal P_\theta\to[0,\infty]$. They are defined by $I_\beta(P):=\lim_{n\to\infty} \frac{I(P_n,Q^{\beta}_n)}{(2\lambda_n)^d}$, where $P_n:=P\circ(\pr_{\Lambda_n}^v)^{-1}$, $Q^\beta_n:=Q^\beta\circ(\pr_{\Lambda_n}^v)^{-1}$ and $Q^\beta$ denotes the Poisson point random field with intensity measure $dx e^{-\beta v^2/2}dv$ (cf.~\cite[Section 4]{Ge95}, also for the definition of $I(\cdot,\cdot)$, which denotes relative entropy). We will sometimes make use of the identity (\cite[(4.3)]{Ge95})
\begin{equation}\label{eqn:blablabla}
S(P)=-I_\beta(P)+\beta U_{\kin}(P)+c(\beta)
\end{equation}
for $P\in\mathcal P_\theta$. Here $c(\beta)={\sqrt{2\pi/\beta}^d}$. So, $S(P)<\infty$ for all $P\in\mathcal P_\theta$.\\[1.5ex]
In \cite[Theorem 3.2]{Ge95} it is shown that the function
\begin{align*}
s(\rho',\varepsilon):&=\sup\{S(P)| P\in \mathcal P_\theta, U(P)\leq \varepsilon, \rho(P)=\rho'\}\\
&=\sup\{S(P)| P\in \mathcal P_\theta, U(P)= \varepsilon, \rho(P)=\rho'\},\quad \rho'\geq 0, \varepsilon\in\R,
\end{align*}
is upper semicontinuous and concave and coincides in the convex set $\Sigma=\{(\rho',\varepsilon)| \varepsilon>\varepsilon_{\min}(\rho')\}$ with the thermodynamic entropy density $\lim_{n\to\infty} \frac{\log M_n}{(2\lambda_n)^d}$, where $(M_n)_{n\in\N}$ denotes a sequence of microcanonical partition functions in $\Lambda_{\lambda_n}$ with periodic boundary such that the densities converge towards $\rho'$ and the energy densities converge towards $\varepsilon$. Here $\varepsilon_{\min}(\rho')=\inf\{U(P)|P\in\mathcal P_\theta, \rho(P)=\rho',S(P)>-\infty\}$. \shortspacing
We now state a variational principle for the thermodynamic free energy density, which we derive below as a direct consequence of the above mentioned corresponding result from \cite{Ge95} on the thermodynamic entropy density and some considerations from \cite{Ru69}. 
\begin{lemma}\label{lem:freeenergy}
For $\beta>0$, $\rho>0$, let the free energy $f(\rho,\beta)$ be defined by $\beta f(\rho,\beta)=\inf_{\varepsilon>\varepsilon_{\min}(\rho)}(\beta \varepsilon-s(\rho,\varepsilon))$. $f(\rho,\beta)$ is finite and it holds
\begin{equation}\label{eqn:tollesresultat}
\beta f(\rho,\beta)=-\lim_{n\to\infty}\frac{\log Z_n}{(2\lambda_n)^d}=\inf\big\{\beta U(P)-S(P)| P\in \mathcal P_\theta, \rho(P)=\rho\big\}.
\end{equation}
where for $n\in\N$
$$
Z_n=\frac{1}{N_n!}\int_{\Lambda_{\lambda_n}^{N_n}\times\R^{dN_n}}e^{-\beta \widetilde{U}_{\lambda_n}(x_1,\cdots,x_{N_n})}e^{-\frac{\beta}{2}(v_1^2+\cdots,v_{N_n}^2)} dx_1\cdots dx_{N_n}\,dv_1\cdots dv_{N_n}
$$
is the canonical partition function with periodic boundary condition.
\end{lemma}
\begin{proof}
By \cite[p.~1350]{Ge95} for $\varepsilon>\varepsilon_{\min}(\rho)$ it holds $s(\rho,\varepsilon)>-\infty$ and moreover $\varepsilon_{\min}(\rho)$ is finite, so we conclude that $f(\rho,\beta)<\infty$. Furthermore, due to \cite[Lemma 4.1, Equation (4.4)]{Ge95} and (\ref{eqn:blablabla}) the set $\{S(P): P\in\mathcal P_\theta, \rho(P)=\rho, U(P)=\varepsilon\}$ is bounded from above by $\beta\varepsilon+\beta (\widetilde{B}^2/4\widetilde{A})+c(\beta)$ with constants $\widetilde{A},\widetilde{B}$ as in Lemma \ref{lem:uniformprops}. This implies that for any $\varepsilon>\varepsilon_{\min}(\rho)$ it holds $\beta\varepsilon-s(\rho,\varepsilon)\geq -\beta (\widetilde{B}^2/4\widetilde{A})-c(\beta)$, hence $f(\rho,\beta)>-\infty$.\shortspacing
The arguments in \cite[p.~55]{Ru69} also work in the case of periodic boundary condition and including velocities, which together with \cite[Theorem 3.2]{Ge95} yields the first equality in (\ref{eqn:tollesresultat}). \\
To prove the second one, first note that for any $P\in\mathcal P_\theta$ with $\rho(P)=\rho$ and $U(P)<\infty$ it holds
$$
\beta U(P)-S(P)\geq \beta U(P)-s(\rho,U(P))\geq \beta f(\rho,\beta).
$$
This follows from the definition of $f(\rho,\beta)$, when $U(P)>\varepsilon_{\min}(\rho)$. Moreover, since $\lim_{\varepsilon\downarrow\varepsilon_{\min}(\rho)} s(\rho,\varepsilon)=s(\rho,\varepsilon_{\min}(\rho))$ due to upper semicontinuity and concavity of $s(\cdot,\cdot)$, this extends also to $U(P)=\varepsilon_{\min}$. For $U(P)<\varepsilon_{\min}(\rho)$ it is implied by the definition of $\varepsilon_{\min}(\rho)$. Thus $\beta f(\rho,\beta)\leq \inf\{\beta U(P)-S(P)| P\in\mathcal P_\theta,\rho(P)=\rho,U(P)<\infty\}$. \\
To prove the converse inequality, for $\delta>0$ choose $\tilde\varepsilon\in(\varepsilon_{\min}(\rho),\infty)$ such that $\beta f(\rho,\beta)+\delta\geq \beta \tilde\varepsilon-s(\rho,\tilde\varepsilon)$. By \cite[(3.9) and Theorem 3.2(b)]{Ge95} we may choose $P\in \mathcal P_\theta$ such that $\rho(P)=\rho$, $U(P)=\tilde\varepsilon$ such that $S(P)\geq s(\rho,\tilde\varepsilon)-\delta$. Hence $\beta U(P)-S(P)\leq \beta\tilde\varepsilon-s(\rho,\tilde\varepsilon)+\delta\leq f(\rho,\beta)+2\delta$. Since $\delta$ may be chosen arbitrarily small, the second equality is shown.
\end{proof}
\ \\
Let $n\in\N$. We define the measure $\hat{\mu}^{[n]}$ on $\Gamma^v$ such that the configurations in $(\Lambda_{\lambda_n}+2\lambda_n r)\times\R^d$, $r\in\Z^d$ are independent and distributed as shifts of $\mu^{[n]}$ by $2r\lambda_n$. One defines the translation invariant measure $\tilde{\mu}^{[n]}$ as spatial average of the $\hat\mu^{[n]}$, i.e.~$\tilde{\mu}^{[n]}:=\int_{\Lambda_{\lambda_n}} \hat{\mu}^{[n]}\circ\vartheta_x^{-1}\,dx$. It holds $\tilde{\mu}^{[n]}\in \mathcal P_\theta$, cf.~\cite[(6.3)]{Ge95}.\\
Define for $\gamma\in\Gamma^v_{\Lambda_n}$ the measure $R_{n,\gamma}\in\mathcal P_\theta$ defined by $R_{n,\gamma}:=\frac{1}{(2\lambda_n)^d}\int_{\Lambda_{\lambda_n}} \delta_{\vartheta_x\gamma^{(n)}}\,dx$, where $\delta_{\cdot}$ denotes Dirac measure and $\gamma^{(n)}$ denotes the $2\lambda_n$-periodic continuation of $\gamma$. (See \cite{GZ93} for details on these \emph{translation invariant empirical fields}). We will below also have to consider the mixture $\mu^{[n]}R_{n}:=\int_{\Gamma^v}R_{n,\gamma}d\mu^{[n]}(\gamma)$, which, on the other hand, is equal to $\mu^{(n)}:=\mu_n\circ\per_n^{-1}$ due to translation invariance of $\mu^{(n)}$ resulting from the (spatial) translation invariance of $\mu_n$ as a measure on the manifold $E_{\lambda_n}^{N_n}$ (defined in Section \ref{sub:fddyn}). Keeping this in mind, we will nevertheless use the notation $\mu^{[n]}R_n$ in the sequel.\shortspacing
Note that by \cite[(6.3)]{Ge95} it holds for $n\in\N$
\begin{equation}\label{eqn:constantekin}
U_{\kin}(\tilde{\mu}^{[n]})=\frac{1}{(2\lambda_n)^d} \int_{\Gamma_{\lambda_n}^v} \sum_{(x,v)\in\gamma}\frac{1}{2}\beta v^2\,d\mu^{[n]}(\gamma)=U_{\kin}(\mu^{[n]}R_n).
\end{equation}
and by the definition of $\mu^{[n]}$ this expression is equal to $C(\beta)\frac{N_n}{(2\lambda_n)^d}$ for some $C(\beta)>0$ not depending on $n\in\N$.\shortspacing
The proof of Theorem \ref{thm:georgii} is based on the inequality given in the following Lemma. 
\begin{lemma}
It holds for $n\in\N$
\begin{equation}\label{eqn:cooleqn}
S(\tilde{\mu}^{[n]})\geq \beta U(\mu^{[n]} R_n)+\frac{\log(Z_n)}{(2\lambda_n)^d}.
\end{equation}
\end{lemma}
\begin{proof}
It holds $\frac{d(\mu^{[n]}\circ(\pr_{\Lambda_{\lambda_n}}^v)^{-1})}{dQ_n^\beta}=\frac{1}{Q_n^\beta(e^{-F_n})}e^{-F_n}$, where $F_n: \Gamma^v_{\Lambda_{\lambda_n}}\to \R\cup\{\infty\}$ is defined by $F_n(\gamma):=\beta \widetilde{U}_{\lambda_n}(\gamma)$ for $\gamma\in\Gamma^v_{\Lambda_{\lambda_n},N_n}$ and $=\infty$ else. These functions do not form an asymptotic empirical functional in the sense of \cite{GZ93}. Nevertheless, the proof of the second assertion in \cite[Lemma 5.5]{GZ93} is valid for the measures $\mu^{[n]}$, $\tilde{\mu}^{[n]}$. Therefore, using also (\ref{eqn:blablabla}) and (\ref{eqn:constantekin}) we find that
\begin{align*}
S(\tilde{\mu}^{[n]})&=-I_\beta(\tilde{\mu}^{[n]})+\beta U_{\kin}({\mu}^{[n]}R_n)+c(\beta)\\
&\geq -\frac{1}{(2\lambda_n)^d} I(\mu^{[n]}\circ(\pr_{\Lambda_{\lambda_n}}^v)^{-1};Q^\beta_n)+\beta U_{\kin}({\mu}^{[n]}R_n)+c(\beta)\\
&= \frac{1}{(2\lambda_n)^d}  \mu^{[n]}\circ(\pr_{\Lambda_{\lambda_n}}^v)^{-1}(F_n)+\beta U_{\kin}({\mu}^{[n]}R_n)+c(\beta)+\frac{1}{(2\lambda_n)^d} \log(Q_n^\beta(e^{-F_n}))\\
&= \frac{1}{(2\lambda_n)^d} \mu^{[n]}\circ(\pr_{\Lambda_{\lambda_n}}^v)^{-1}(F_n)+\beta U_{\kin}({\mu}^{[n]}R_n)+\frac{1}{(2\lambda_n)^d} \log{Z_n}
\end{align*}
By \cite[(2.16)]{Ge94} and since $U_{\pot}$ is measure affine it holds 
$$
\frac{1}{(2\lambda_n)^d} \mu^{[n]}\circ(\pr_{\Lambda_{\lambda_n}}^v)^{-1}(F_n)=\beta \mu^{[n]}(U_{\pot}(R_{n,\cdot}))=\beta U_{\pot}(\mu^{[n]}R_n).
$$ 
This completes the proof.
\end{proof}
\ \\
Since $\tilde\mu^{[n]}=C(\beta)\frac{N_n}{(2\lambda_n)^d}$, $n\in\N$, the kinetic energy density of all $\tilde{\mu}^{[n]}$ is bounded. Therefore, boundedness of $U_{\pot}$ from below (by $-\widetilde{B}^2/4\widetilde{A}$) together with convergence of $\frac{\log(Z_n)}{(2\lambda_n)^d}$ as $n\to\infty$ (cf.~Lemma \ref{lem:freeenergy}) imply that $(S(\tilde{\mu}^{[n]}))_{n\in\N}$ is bounded from below. This together with the boundedness of the kinetic energy implies relative compactness of the sequence $(\tilde{\mu}^{[n]})_{n}$ (see the properties of $S$ mentioned above). The following lemma shows that asymptotically one can treat $\tilde{\mu}^{[n]}$, $\mu^{[n]}$ and $\mu^{[n]}R_{n}$, $n\in\N$, as equal.

\begin{lemma}\label{lem:gabelstapler}
The sequences $(\mu^{[n]})_n$, $(\tilde{\mu}^{[n]})_n$ and $(\mu^{[n]}R_n)_n$ are asymptotically equivalent, i.e.~for any two of them, say $(\nu_1^n)_n,(\nu_2^n)_n$, and any $f\in\mathcal L$ it holds $\lim_{n\to\infty} \vert\nu_1^n(f)-\nu_2^n(f)\vert=0$. In particular, convergence of $(\tilde{\mu}^{[n_k]})_{k\in\N}$ to some $\mu\in \mathcal P_\theta$ w.r.t.~$\tau_{\mathcal L}$ implies that also $\lim_{k\to\infty}\mu^{[n_k]}=\lim_{k\to\infty}\mu^{[n_k]}R_{n_k}=\mu$.
\end{lemma}
\begin{proof}
This is shown as in the proof of \cite[Lemma 6.2]{Ge95}: The asymptotic equivalence of $\mu^{[n]}$ and $\mu^{[n]}R_n$ is clear. For the second asymptotic equivalence note that $\sup_n I_\beta(\tilde{\mu}^{n})<\infty$ by (\ref{eqn:cooleqn}), (\ref{eqn:blablabla}) and (\ref{eqn:constantekin}), so \cite[Lemma 5.7]{GZ93} can be applied. For convenience of the reader we remark that to apply the mentioned lemma a function $\psi: \R^d\to\R$, defined as function of velocities, has to be chosen appropriately. (This function is used in the definitions of $\mathcal L$ and $\mathcal P$ in \cite{GZ93}.) Setting $\psi(v):=1+\vert v\vert$, $v\in\R^d$, is the standard choice here. Then the definition of $\mathcal P_\theta$ from \cite{GZ93} does not coincide with the one given above, but denotes a larger space. This, however, does not affect the considerations made in the proof of \cite[Lemma 5.7]{GZ93}.
\end{proof}
\ \\
From Lemma \ref{lem:gabelstapler} above, the preceding considerations, (\ref{eqn:cooleqn}) and the properties of $U$, $S$ and $\rho$ we find that there exists an accumulation point $\mu\in\mathcal P_\theta$ of $(\mu^{[n]})_{n\in\N}$ w.r.t.~$\tau_{\mathcal L}$ fulfilling $\rho(\mu)=\rho$ and any such accumulation point fulfills
$$
\beta U(\mu)-S(\mu)\leq -\lim_{n\to\infty}\frac{\log(Z_n)}{(2\lambda_n)^d}.
$$
(Note that since $U$ is bounded from below and $S$ cannot take the value $+\infty$ in $\mathcal P_\theta$ (one may see this from (\ref{eqn:blablabla}) and since $I_\beta$ only takes nonnegative values), both $U(\mu)$ and $S(\mu)$ are finite.)\\
From Lemma \ref{lem:freeenergy} we see that $\beta U(\mu)-S(\mu)=\beta f(\rho,\beta)$. $\mu$ is a minimizer of the canonical free energy density $U(\cdot)-\beta^{-1}S(\cdot)$ under the constraint $\rho(\cdot)=\rho$.\\[1.5ex]
We finally make considerations similar to those in \cite[p.~1351]{Ge95}: For $\beta>0$ the function $f(\beta,\cdot)$ is convex. (This follows from its definition, the concavity of $s(\cdot,\cdot)$ and the convexity of the effective domain $\Sigma$ of $s(\cdot,\cdot)$ defined in \cite[(3.7)]{Ge95}.) Hence we may choose some $z>0$ and $p\in\R$ such that $\rho'\mapsto -p+\rho'\beta^{-1}\log(z)$ is a tangent to $f(\beta,\cdot)$ at $\rho$. This and (\ref{eqn:tollesresultat}) imply for any $P\in \mathcal P_\theta$ it holds
\begin{align}\label{eqn:totalerunsinn}
\beta U(P)-\rho(P)\log(z)-S(P)&\geq \beta f(\rho(P),\beta)-\rho(P)\log(z)\\
&\geq -\beta p=f(\rho,\beta)-\rho\log(z)=\beta U(\mu)-\rho(\mu)\log(z)-S(\mu)\nonumber
\end{align}
In order to prove this inequality for $\rho(P)=0$ (i.e.~$P=\delta_\emptyset$), note that $U(\delta_\emptyset)=0$, $S(\delta_\emptyset)=0$ and $\rho(\delta_\emptyset)=0$. Hence we only need to verify that $p\geq 0$. This, however, follows e.g.~from the fact that for any measure $Q\in \mathcal P_\theta$ fulfilling $U(Q)<\infty$, $\rho(Q)>0$ and $S(Q)>-\infty$ (such a measure exists) it holds
$$
\lim_{\rho'\to 0} \beta f(\rho',\beta)\leq \lim_{\alpha\to 0} \beta U(\alpha Q+(1-\alpha)\delta_\emptyset)-S(\alpha Q+(1-\alpha)\delta_\emptyset)=0,
$$
where we used (\ref{eqn:tollesresultat}) and the fact that $U$ and $S$ are affine functions.\shortspacing
(\ref{eqn:totalerunsinn}) implies that $\mu$ is a minimizer of the mean free energy $U(\cdot)-\rho(\cdot) \log(z)-S(\cdot)$. By \cite[Theorem 3.4]{Ge95} we conclude that $\mu$ is a tempered grand canonical Gibbs measure, and the proof of Theorem \ref{thm:georgii} is completed.

\end{section}

\bibliography{NVL_bibfile}
\bibliographystyle{alpha}

\end{document}